\let\accentvec\vec
\documentclass[smallextended]{svjour3}  

\let\vec\accentvec     
\smartqed  
\usepackage{amsmath}

\usepackage{lmodern}
\usepackage{bbding}
\usepackage{amsfonts}
\usepackage{float}
\usepackage{graphicx}
\usepackage{bm}
\usepackage{xspace}
\usepackage{setspace}
\usepackage{subeqnarray}
\usepackage{graphicx,amssymb,epsfig}
\usepackage[caption=false]{subfig}
\usepackage{mathrsfs}
\usepackage{extarrows}
\usepackage{color}
\usepackage{booktabs}
\usepackage{algorithm}
\usepackage{threeparttable}
\usepackage{dsfont}
\usepackage{miniplot,epstopdf}

\newfloat{program}{thp}{lop}
\graphicspath{{pics/}}
\DeclareMathOperator*{\argmin}{arg\,min}

\newcommand{\vC}{{\mathbf{C}}}
\newcommand{\vD}{{\mathbf{D}}}

\newcommand{\vI}{{\mathbf{I}}}

\newcommand{\vU}{{\mathbf{U}}}

\newcommand{\vW}{{\mathbf{W}}}

\journalname{Journal of Scientific Computing}
\begin{document}

\title{A Multiphase Image Segmentation Based on Fuzzy Membership Functions and L1-norm Fidelity
\thanks{}
}
\titlerunning{Multiphase Image Segmentation via Fuzzy and L1}

\author{Fang Li \and   Stanley Osher \and Jing Qin \and Ming Yan
}


\institute{F. Li (\Envelope) \at
              Department 
of Mathematics, East China Normal University, and Shanghai Key Laboratory of Pure Mathematics and Mathematical Practice, Shanghai, China.\\
              \email{e-mail: fli@math.ecnu.edu.cn}             \\
           S. Osher, J. Qin \at
              Department of Mathematics, University of California, Los Angeles, CA 90095, USA.\\
\email{sjo@math.ucla.edu,jxq@ucla.edu} \\
             and M. Yan \at
             Department of Computational Mathematics, Science and Engineering and Department of Mathematics, Michigan State University, East Lansing, MI 48824, USA. \\
             \email{e-mail:yanm@math.msu.edu}
}

\date{Received: date / Accepted: date}

\maketitle

\begin{abstract}
In this paper, we propose a variational multiphase image segmentation model based on fuzzy membership functions and L1-norm fidelity. Then we apply the alternating direction method of multipliers to solve an equivalent problem. All the subproblems can be solved efficiently. Specifically, we propose a fast method to calculate the fuzzy median. Experimental results and comparisons show that the L1-norm based method is more robust to outliers such as impulse noise and keeps better contrast than its L2-norm counterpart. Theoretically, we prove the existence of the minimizer and analyze the convergence of the algorithm.
\end{abstract}


\section{Introduction}
As a fundamental step in image processing, image segmentation partitions an image into several disjoint regions such that pixels in the same region share some uniform characteristics such as intensity, color, and texture. During the last two decades, image segmentation methods using variational methods and partial differential equations are very popular due to their flexibility in problem modeling and algorithm design. There are two key ingredients of variational segmentation models. One is how to describe the regions or boundaries between these regions, and the other is how to model the noise and describe the uniform characteristics of each region.


The Mumford-Shah model~\cite{MS}, a well-known variational segmentation model, penalizes the combination of the total length of the segmentation boundaries and the L2-norm error of approximating the observed image with an unknown piecewise smooth approximation. In other words, the Mumford-Shah model seeks an optimal piecewise smooth function with smooth boundaries to approximate the observed image.

However, the Mumford-Shah model is hard to implement in practice because the discretization of the unknown set of boundaries is very complex. Therefore, a parametric/explicit active contour method is used to represent the segmentation boundaries~\cite{zhu1996region}. In addition, the special Mumford-Shah model with a piecewise constant approximation is studied by Chan and Vese~\cite{CV}, and the level set method~\cite{osher1988fronts} is applied to solve this problem. Using an implicit representation of boundaries, the level set method has many advantages over methods using explicit representations of boundaries. For instance, the level set method handles the topological change of zero level set automatically~\cite{caselles1997geodesic,aubert2006mathematical,Qin2014}. Both the parametric/explicit active contour method and the level set method assume that each pixel belongs to a unique region. An alternative way to represent various regions is to use fuzzy membership functions~\cite{chan2006algorithms,bresson2007fast,mory2007fuzzy,houhou2008fast}, which describe the levels of possible membership. Fuzzy membership functions assume that each pixel can be in several regions simultaneously with probability in [0,1]. These probabilities satisfy two constraints: (i) nonnegativity constraint, i.e., the membership functions are nonnegative at all pixels; (ii) sum-to-one constraint, i.e., the sum of all the membership functions at each pixel equals one. Then the length of boundaries can be approximated by the Total Variation (TV) of fuzzy membership functions. The main advantages of using fuzzy membership functions over other methods include: i) the energy functional is convex with respect to fuzzy membership functions, guaranteeing the convergence and stability of many numerical optimization methods. ii) fuzzy membership function has a larger feasible set, and it is possible to find better segmentation results.

For two-phase segmentation, where there are only two regions, we only need one level set function or one fuzzy membership function. Multiphase segmentation is more challenging than two-phase segmentation since more variables and constraints are involved in representing multiple regions and their boundaries effectively. The two-phase Chan-Vese model~\cite{CV} has been generalized to multiphase segmentation by using multiple level set functions to represent multiple regions~\cite{vese2002multiphase}. Partitioning an image into $N$ disjoint regions requires $\log_2N$ level set functions. The advantage of using multiple level set functions is that it automatically avoids the problems of vacuum and overlap of regions. However, the implementation is not easy, and special numerical schemes are needed to ensure stability~\cite{PCLSM,lie2006binary,softMS,jung2007multiphase}. For fuzzy membership functions, the sum-to-one constraint is not satisfied automatically in multiphase segmentation. However, this constraint is easy to deal with in many cases, e.g., Fuzzy C-Mean (FCM) and its variants have closed-form solutions for the membership functions and are widely used in data mining and medical image segmentation~\cite{FCM,ahmed2002modified,pham2002fuzzy,FCMS2,FLICM,li2010softseg,li2010competition,TVFCM,cai2015variational}. Variable splitting schemes are used in both~\cite{li2010competition} and~\cite{li2010softseg} to get efficient numerical algorithms. The Alternating Direction Method of Multipliers (ADMM) method is used in~\cite{TVFCM} to derive the algorithm with two sets of extra variables. Primal-dual type algorithm is derived in~\cite{CP} to solve the TV regularized FCM segmentation model. Both~\cite{TVFCM}  and~\cite{CP} use projection to simple to handle the constraints of membership functions. Other segmentation approaches include a convex approach~\cite{chambolle2012convex}, two-stage methods~\cite{chan2014two,storath2014fast}, one single level set function approach~\cite{dubrovina2015multi}, et.al.


Noise is unavoidable in images, and it is important to develop segmentation methods that work on noisy images. Among many types of noise, the Gaussian white noise is frequently assumed, and the L2-norm fidelity is adopted. However, when images are corrupted by non-Gaussian noise, in order to obtain a faithful segmentation, one has to model the noise according to its specific type~\cite{sawatzky2013variational,nikolova2004variational,chan2005salt,guo2009fast,Yan13,Yan13b,cai2015variational}. Particularly, the L1-norm fidelity is used for both salt-and-pepper impulse noise and random-valued impulse noise in image processing~\cite{nikolova2004variational,chan2005salt,guo2009fast}. In addition, it is robust to outliers and able to preserve contrast because the denoising process of L1-norm models is determined by the geometry such as area and length rather than the contrast in the L2-norm case~\cite{chan2005aspects}.


Inspired by the fact that L1-norm is more robust to impulse noise and outliers and can better preserve contrast, in this paper, we propose a variational {\it multiphase} fuzzy segmentation model based on {\it L1-norm fidelity} and {\it fuzzy membership functions}. This model can also deal with missing data in images. When there are missing pixels in an image, we randomly assign 0 or 255 at these pixels by considering these pixels as corrupted by salt-and-pepper impulse noise. ADMM~\cite{ADMM1,ADMM2}, which was rediscovered as split Bregman~\cite{SB} and found to be very useful for L1 and TV type optimization problems, is applied to solve this nonconvex optimization problem. By introducing two sets of auxiliary variables, we derive an efficient algorithm with all the subproblems having closed-form solutions.  In the theoretical aspect, we prove the existence of the minimizer and analyze the convergence of the algorithm.
We note that the proposed method is closely related to the method in~\cite{TVFCM} since both methods use TV regularization and ADMM algorithm. The difference is that L1-norm fidelity is
considered in the proposed method, while  L2-norm fidelity is used in~\cite{TVFCM}.

The outline of this paper is as follows. In Section~\ref{sec:preliminary}, we review some closely related existing works.
In Section~\ref{sec:model}, the proposed model is described in detail, and the existence of the minimizer to the model is proved. In Section~\ref{sec:algorithm}, a numerical algorithm is derived, and its convergence analysis is presented. In Section~\ref{sec:experiment}, experimental results and comparisons are presented to illustrate the effectiveness of the proposed method. Finally, the paper is concluded in Section~\ref{sec:conclusion}.

\section{Related works}\label{sec:preliminary}
Let $\Omega\subset \mathbb{R}^2$ be a bounded open subset with Lipschitz boundary, and let $I:\Omega\rightarrow \mathbb{R}^s$ be the given clean or noisy image. Let $s=1$ for grayscale images and $s=3$ for color images. Our goal is to find an $N$-phase ``optimal'' partition $\{\Omega_i\}_{i=1}^N$ such that $\Omega_i \bigcap \Omega_j=\emptyset$ for all  $ i\neq j$ and $\bigcup_{i=1}^N \Omega_i=\Omega$. Define the set of $N$-phase fuzzy membership functions as
\begin{equation*}\label{delta}
\Delta:=\left\{(u_1,...,u_N)\Big|u_i\in BV(\Omega), u_i(x)\ge 0,
\sum_{i=1}^N u_i(x)=1, \forall x\in \Omega\right\},
\end{equation*}
where $BV(\Omega)$ is the space of functions with bounded variation~\cite{aubert2006mathematical}. The closely related works are listed in the following and will be compared with our proposed method in Section~\ref{sec:experiment}. For the sake of simplicity, we use the notations $\mathbf{U} = (u_1,\cdots,u_N)$ and $\mathbf{C} = (c_1,\cdots,c_N)\in \mathbb{R}^{sN}$, where $c_i\in\mathbb{R}$ for grayscale images and $c_i\in\mathbb{R}^3$ for color images.

\begin{itemize}
\item FCM~\cite{FCM}-- Fuzzy c-means clustering method that solves
\begin{equation*}\label{FCM}
 \min_{(\mathbf{U},\mathbf{C})\in \Delta\times\mathbb{R}^{sN}}
\sum_{i=1}^N\int_\Omega \left|I(x) - c_i\right|^2u_i^p(x) \ dx
 \end{equation*}
using the alternating minimization algorithm. Though $p$ can be any number no smaller than $1$, it is commonly set to 2.
\item FCM\_S2~\cite{FCMS2} -- A variant of FCM that solves
\begin{equation*}\label{FCMS2}
 \min_{(\mathbf{U},\mathbf{C})\in \Delta\times\mathbb{R}^{sN}}
\sum_{i=1}^N\left\{\int_\Omega \left|I(x) - c_i\right|^2u_i^p(x) \ dx+\alpha\int_\Omega \left|\bar{I}(x) - c_i\right|^2u_i^p(x) \ dx\right\},
 \end{equation*}
where $\bar{I}$ is obtained by applying the median filter on $I$ and $\alpha>0$ is a weight parameter. It can also be solved by the alternating minimization algorithm, and it is more robust to impulse noise than FCM.


\item FLICM~\cite{FLICM} -- Fuzzy Local Information C-Means clustering method that solves
\begin{align*}
 \min_{(\mathbf{U},\mathbf{C})\in \Delta\times\mathbb{R}^{sN}} &
\sum_{i=1}^N\left\{\int_\Omega \left|I(x) - c_i\right|^2u_i^p(x) \ dx\right.\\
&\left.+\alpha\int_\Omega\int_{y\in\mathcal{N}(x)}
(1-u_i(y))^p\left|I(y)-c_i\right|^2u^p_i(x)dydx\right\},
\end{align*}
where $\mathcal{N}(x)$ is a neighborhood of $x$. FLICM is more robust to both Gaussian noise and impulse noise than FCM.

\item L2FS~\cite{TVFCM} -- L2-norm fidelity based Fuzzy Segmentation method that solves 
\begin{equation}\label{TVFCML2}
 \min_{(\mathbf{U},\mathbf{C})\in \Delta\times\mathbb{R}^{sN}}
\sum_{i=1}^N\left\{\int_\Omega\left\|\nabla u_i(x)\right\|dx +
 \lambda\int_\Omega \left|I(x) - c_i\right|^2u_i(x) \ dx\right\},
 \end{equation}
by ADMM. Here $\lambda>0$ is a parameter and $\int_\Omega\left\|\nabla u_i(x)\right\|dx$ denotes the TV of $u_i$ with $\left\|\nabla u_i(x)\right\|:=\sqrt{(\nabla_{x_1}u_{i}(x))^2+(\nabla_{x_2}u_{i}(x))^2}$ for $x=(x_1,x_2)$. For fixed $\mathbf{C}$,~\cite{TVFCM} is related to the popular TV denoising method~\cite{ROF}.
Note that the similar model is solved by other fast numerical methods in~\cite{li2010competition}.


\item L1SS~\cite{TVSEGL1} -- L1-norm fidelity based Soft Segmentation method, in which $\log_2N$ soft functions are introduced to represent $N$ phases. Since the model for multiphase segmentation is complicated for more than four phases, we show the four-phase model as follows:
\begin{equation*}\label{TVL1segold}
\min_{u_1,u_2\in[0,1],\mathbf{C}\in\mathbb{R}^{sN}}
\left\{\sum_{i=1}^2\int_\Omega\left\|\nabla u_i(x)\right\|dx +
\lambda\sum_{j=1}^4\int_\Omega \left|I(x) - c_j\right|M_j(x)dx\right\},
\end{equation*}
where the membership functions $M_j$, $j=1,\cdots,4$, are represented by soft functions $u_1(x),u_2(x)\in[0,1]$ in the following way:
\begin{align*}
M_1(x)&=u_1(x)u_2(x), &M_2(x)&=u_1(x)(1-u_2(x)), \\
M_3(x)&=(1-u_1(x))u_2(x), &M_4(x)&=(1-u_1(x))(1-u_2(x)).
\end{align*}
{
\item L2L0~\cite{storath2014fast} -- L2-norm fidelity and L0-norm regularization based image partition model:
\begin{equation*}
\min_u \|\nabla u\|_0 +\lambda \|u-I\|_2^2.
\end{equation*}
This model seeks a piecewise constant approximation of the original image $I$. Since this model can not specify the number of classes, a second step is applied to combine some classes if this model returns more classes than required. Here we apply FCM on the piecewise constant approximation $u$ to obtain the final segmentation result.
}
\end{itemize}
{\bf Remark:} There are two advantages of our proposed method over L1SS. Firstly, we use fuzzy membership functions to represent regions, where $N$ fuzzy membership functions are required  for an $N$-phase segmentation. Hence, the solution space is much larger than L1SS, which ensures the higher possibility to obtain optimal segmentation. 
Secondly, the proposed method can take use of other commonly used segmentation methods such as FCM to gain good initialization of fuzzy membership functions. Multiphase segmentation is sensitive to initialization, and a good initialization is very important for a successful segmentation. However, it is hard to use the existing segmentation methods to get a good initialization for soft membership functions in L1SS.



\section{The proposed model}\label{sec:model}

In this paper, we assume that the given image can be approximated by a piecewise constant function, i.e.,
\begin{equation*}
I(x) =\sum_{i=1}^N c_i\chi_{\Omega_i}(x)+ n(x).
\end{equation*}
Here $\chi_{\Omega_i}$ denotes the indicator function of region $\Omega_i$, i.e.,
\[
\chi_{\Omega_i}(x)=\left\{
\begin{array}{l}
1, \ \mbox{if}\ x\in \Omega_i;\\
0, \ \mbox{otherwise},
\end{array}\right.
\]
$c_i$ is a constant that represents the given data in region $\Omega_i$, and $n$ can be outliers, impulse noise or other mixture types rather than Gaussian noise.

Instead of using the L2-norm fidelity to measure the approximation error when the noise is the Gaussian white noise, we use the L1-norm fidelity. Same as in the Mumford-Shah model, we require the segmentation boundaries to be smooth. Then we have the following model
\begin{equation}\label{TVL1}
 \min\limits_{\{\Omega_i\},\vC} \sum_{i=1}^N\int_\Omega\|\nabla \chi_{\Omega_i}(x)\|dx + \lambda\int_\Omega \left|I(x) -\sum_{i=1}^N c_i\chi_{\Omega_i}(x)\right|dx,
 \end{equation}
where $\lambda>0$ is a tuning parameter. Note that the TV of $\chi_{\Omega_i}$ in the first term is equal to the length of boundary $\partial \Omega_i$. An equivalent form of model (\ref{TVL1}) is
\begin{equation}\label{TVL12}
 \min\limits_{\{\Omega_i\},\vC}  \sum_{i=1}^N\left\{\int_\Omega\|\nabla \chi_{\Omega_i}(x)\|dx + \lambda\int_\Omega \left|I(x) - c_i\right|\chi_{\Omega_i}(x)dx\right\}.
 \end{equation}
Because $\chi_{\Omega_i}$ can take values 0 and 1 only and $\{\chi_{\Omega_i}\}$ is a partition, $(\chi_{\Omega_1},\dots,\chi_{\Omega_N})$ belongs to the set
$$\Delta_0:=\left\{(u_1,...,u_N)\big|u_i\in BV(\Omega), u_i(x)\in \{0,1\}, \sum_{i=1}^N u_i(x)=1, \forall x\in \Omega\right\}.$$
At any point $x\in\Omega$, there is only one function having value 1, and all the other functions have value 0. Thus set $\Delta_0$ is not continuous, which leads to difficulties and instabilities in numerical implementations. However, we can relax binary indicator functions $\{\chi_{\Omega_i}\}$ to fuzzy membership functions $\{u_i\}$, which satisfy the nonnegativity constraint and the sum-to-one constraint, i.e., $(u_1,...,u_N)$ belongs to the set $\Delta$ defined in (\ref{delta}).
It is obvious that $u_i(x)\in[0,1]$ and $\Delta$ is a simplex at any $x\in\Omega$. So $u_i(x)$ can be understood as the probability of pixel $x$ to be in the $i$th class. Then we can rewrite our model (\ref{TVL12}) as
\begin{equation}\label{TVFCML1}
 \min_{(\mathbf{U},\mathbf{C})\in \Delta\times\mathbb{R}^{sN}}
 E(\mathbf{U},\mathbf{C})=\sum_{i=1}^N\left\{\int_\Omega\left\|\nabla u_i(x)\right\|dx +
 \lambda\int_\Omega \left|I(x) - c_i\right|u_i(x) \ dx\right\}.
 \end{equation}
Note that model (\ref{TVFCML1}) is convex with respect to $\mathbf{U}$ and $\mathbf{C}$ separately but not jointly. The difference between~\eqref{TVFCML1} and~\eqref{TVFCML2} is that the L2 fidelity term in~\eqref{TVFCML2} is replaced by the L1 fidelity term. The existence of a minimizer for $E(\mathbf{U},\mathbf{C})$ in (\ref{TVFCML1}) is stated in Theorem~\ref{minimizer}.

\begin{theorem}[Existence of a minimizer]\label{minimizer}
Given an image $I\in L^\infty(\Omega)$, there exists a minimizer of $E(\mathbf{U},\mathbf{C})$ in $\Delta\times\mathbb{R}^{sN}$ for any parameter $\lambda>0$.
\end{theorem}
\begin{proof}
Since $E(\vU,\vC)$ is positive and proper, the infimum of $E(\vU,\vC)$ must be finite. Let us pick a minimizing sequence $(\mathbf{U}^n,\mathbf{C}^n)\in \Delta\times\mathbb{R}^{sN}$ that satisfies $E(\mathbf{U}^n,\mathbf{C}^n)\rightarrow \inf_{\vU,\vC} E(\mathbf{U},\mathbf{C})$ as $n\rightarrow\infty$. Then there exists a constant $M>0$ such that
$$E(\mathbf{U}^n,\mathbf{C}^n)=\sum_{i=1}^N\left\{\int_\Omega\left\|\nabla u_i^n(x)\right\|dx +
 \lambda\int_\Omega \left|I(x) - c_i^n\right|u_i^n(x) \ dx\right\}\le M.$$
Then each term in $E(\vU^n,\vC^n)$ is bounded, i.e., for each $i=1,\cdots,N$,
\begin{equation}\label{tvisbound}
\int_\Omega \|\nabla u_i^n(x)\|dx\le M,\
\int_\Omega |I(x)-c_i|u^n_i(x)dx\le M.
\end{equation}
Since $u_i^n(x)\in[0,1]$, we have $\int_\Omega u_i^n(x)dx\le |\Omega|$, where $|\Omega|$ is the area of $\Omega$.  Together with the first equality in (\ref{tvisbound}), we have that $u_i^n$ is uniformly bounded in $BV(\Omega)$ for all $i = 1,..., N$. By the compactness property of $BV(\Omega)$ and the relative compactness of $BV(\Omega)$ in $L^1(\Omega)$, up to a subsequence also denoted by $\{u_i\}^n$ after relabeling, there exists a function $u_i^*\in BV(\Omega)$ such that
\begin{align*}
u_i^n&\rightarrow u_i^* \ \ \mbox{strongly in} \  L^1(\Omega),\\
\hspace{-1.1cm}u_i^n&\rightarrow u_i^* \ \ \mbox{a.e.}\  x\in \Omega,\\
\hspace{1.3cm}\nabla u_i^n &\rightharpoonup \nabla u_i^* \ \
\mbox{in the sense of measure.}
\end{align*}
 Then by the lower semicontinuity of the TV semi-norm,
\begin{equation}\label{lsc1}
\int_\Omega \|\nabla u_i^*(x)\|dx\le
\liminf_{n\rightarrow\infty}\int_\Omega \|\nabla u_i^n(x)\|dx.
\end{equation}
Meanwhile since $\mathbf{U}^n = (u_1^n,...,u_N^n)\in \Delta$, we have $\mathbf{U}^* = (u_1^*,...,u_N^*)\in\Delta$.

It is easy to derive the first order optimality condition about $c_i^n$, which is
\begin{equation*}\label{ck}
0\in\int_{\Omega}\partial|I(x)-c_i^n|u_i^n(x)dx,
\end{equation*}
where $\partial|\cdot|$ is the subdifferential of $|\cdot|$. Since $u_i^n(x)\ge 0$ and $\int_\Omega u_i^n(x)dx>0$, the above equation implies that the constant $c_i^n$ satisfies
$$|c_i^n|\le \|I\|_\infty.$$
By the boundedness of sequence $\{c_i^n\}$, we can extract a subsequence also denoted by $\{c_i^n\}$  such that  for some constant $c_i^*$
$$c_i^n\rightarrow c_i^*, \ \mbox{as}\  n\to\infty.$$

Finally, since $u_i^n(x)\rightarrow u_i^*(x)$, a.e. $x\in \Omega$ and $c_i^n\rightarrow c_i^*$ as $n\to\infty$, Fatou's lemma yields
\begin{equation}\label{lsc3}
\int_\Omega |I(x)-c_i^*|u_i^*(x)dx \le
\liminf_{n\rightarrow\infty}\int_\Omega |I(x)-c_i^n|u_i^n(x)dx.
\end{equation}
Combining inequalities (\ref{lsc1}) and (\ref{lsc3}) for all $i$, on a suitable subsequence, we established
\begin{equation*}
E(\mathbf{U}^*, \mathbf{C}^*)\le \liminf_{n\rightarrow \infty}E(\mathbf{U}^n,\mathbf{C}^n) = \inf_{\vU,\vC} E(\mathbf{U},\mathbf{C}),
\end{equation*}
and hence $(\mathbf{U}^*,\mathbf{C}^*)$ must be a minimizer of the energy $E$. This completes the proof.\qed
\end{proof}

The minimizer of $E(\mathbf{U},\mathbf{C})$ is not unique due to the following hidden symmetry property. Denote $S_N$ as the permutation group of $\{1,...,N\}$, i.e., each permutation $\sigma\in S_N$ is defined as a one-to-one map $\sigma:\{1,...,N\}\rightarrow \{1,...,N\}$ such that $\{\sigma(1),...,\sigma(N)\}$ is a rearrangement of $\{1,...,N\}$. Denote $\mathbf{U}_\sigma = (u_{\sigma(1)},...,u_{\sigma(N)})$, $\mathbf{C}_\sigma = (c_{\sigma(1)},...,c_{\sigma(N)})$. It is straightforward to show the following theorem.

\begin{theorem}[Symmetry of minimizer]\
For any $(\mathbf{U},\mathbf{C})\in \Delta\times \mathbb{R}^N$ and any $\sigma\in S_N$,
$$E(\mathbf{U}_\sigma, \mathbf{C}_\sigma)=E(\mathbf{U},\mathbf{C}).$$
In particular, suppose that $(\mathbf{U}^*, \mathbf{C}^*)$ is a minimizer of $E(\mathbf{U},\mathbf{C})$, i.e.,
$$(\mathbf{U}^*, \mathbf{C}^*)= \mathop{\arg\min}_{(\mathbf{U},\mathbf{C})\in \Delta\times \mathbb{R}^N } E(\mathbf{U}, \mathbf{C}).$$
Then, for any $\sigma\in S_N$, we have
$$(\mathbf{U}^*_\sigma, \mathbf{C}^*_\sigma)= \mathop{\arg\min}_{(\mathbf{U},\mathbf{C})\in \Delta\times \mathbb{R}^N } E(\mathbf{U}, \mathbf{C}).$$
\end{theorem}

\section{The numerical algorithm and its convergence analysis}\label{sec:algorithm}
In this section, we provide an efficient algorithm based on ADMM and discuss its convergence.

\subsection{The algorithm}

ADMM is applied, in this subsection, to solve the proposed fuzzy segmentation model~\eqref{TVFCML1}. We introduce two sets of auxiliary variables
$\mathbf{D}=(d_1,...,d_N), \mathbf{W}=(w_1,...,w_N)$ such that $\nabla u_i = d_i, u_i=w_i $. Then the
 model (\ref{TVFCML1}) is equivalent to the following minimization problem with equality constraints:
\begin{equation}\label{cons}
 \begin{array}{rl}\displaystyle
 \min_{\mathbf{D}, \mathbf{W},\mathbf{C},\mathbf{U}} &\displaystyle\sum_{i=1}^N\left\{\int_\Omega\| d_i(x)\|dx + \lambda\int_\Omega \left|I(x) - c_i\right|w_i(x) \ dx\right\}+\delta_\Delta(\mathbf{W}),\\\displaystyle
  \mbox{subject to} &\nabla u_i = d_i, u_i=w_i, \forall i=1,\dots,N,
 \end{array}
 \end{equation}
 where $\delta_\Delta$ is the indicator function of the set $\Delta$, i.e.,
\[\delta_\Delta(\mathbf{W})= \left\{
 \begin{array}{l}
 0, \ \ \ \ \  \mbox{if} \ \mathbf{W}\in\Delta,\\
 +\infty, \ \mbox{otherwise}.
 \end{array}\right.\]
The augmented Lagrangian function for problem (\ref{cons}) is:

\begin{align*}\label{Lag}
&\mathscr{L}(\mathbf{D,W,C,U;\Lambda_D,\Lambda_W})\\&=\sum_{i=1}^N\left\{\int_\Omega\left\|d_i(x)\right\|dx + \lambda\int_\Omega \left|I(x) - c_i\right|w_i(x) \ dx \right\}+\delta_\Delta(\mathbf{W})\\
&+\sum_{i=1}^N\left\{\left\langle\lambda_{d_i},\nabla u_i-d_i\right\rangle+\frac{r}{2}\int_\Omega \left\|\nabla u_i(x)-d_i(x)\right\|^2dx\right\}\\
&+\sum_{i=1}^N\left\{\left\langle\lambda_{w_i}, u_i-w_i\right\rangle+\frac{r}{2}\int_\Omega \left|u_i(x)-w_i(x)\right|^2dx\right\},
\end{align*}
where $\Lambda_\mathbf{D} = (\lambda_{d_1},...,\lambda_{d_N}),\Lambda_\mathbf{W} = (\lambda_{w_1},...,\lambda_{w_N})$ are the Lagrangian multipliers and $r$ is a positive constant. Here $\langle \lambda_{d_i}, \nabla u_i-d_i\rangle=\int_\Omega \lambda_{d_i}^T(x)(\nabla u_i(x)-d_i(x))dx$, and $\langle\lambda_{w_i}, u_i-w_i\rangle=\int_\Omega \lambda_{w_i}(x)(u_i(x)-w_i(x))dx$.




The ADMM solves the primal variables one by one in the Gauss-Seidel style and then updates the dual variables (Lagrangian multipliers). It can be summarized as follows:	
\begin{align*}
\mathbf{D}^{k+1} &=\argmin\limits_{\mathbf{D}}\mathscr{L}({\vD,\vW^k,\vC^{k},\vU^{k};\Lambda_\vD^k,\Lambda_\vW^k}),\\
\mathbf{W}^{k+1} &=\argmin\limits_{\mathbf{W}}\mathscr{L}({\vD^{k+1},\vW,\vC^{k},\vU^{k};\Lambda_\vD^{k},\Lambda_\vW^k}),\\
\mathbf{C}^{k+1} &=\argmin\limits_{\mathbf{C}}\mathscr{L}({\vD^{k+1},\vW^{k+1},\vC,\vU^{k};\Lambda_\vD^{k},\Lambda_\vW^{k}}),\\
\mathbf{U}^{k+1} &=\argmin\limits_{\mathbf{U}}\mathscr{L}({\vD^{k+1},\vW^{k+1},\vC^{k+1},\vU;\Lambda_\vD^{k},\Lambda_\vW^{k}}),\\
\lambda_{d_i}^{k+1} &=\lambda_{d_i}^{k}+r\left(\nabla u_i^{k+1}-d_i^{k+1}\right),\\
\lambda_{w_i}^{k+1} &=\lambda_{w_i}^{k}+r\left(u_i^{k+1}-w_i^{k+1}\right).
\end{align*}
In the following, we show how to solve each subproblem and then give the algorithm.


\subsection*{$\mathbf{D}$-subproblem}
The subproblem for $\mathbf{D}$ is
equivalent to
\[
\vD^{k+1}=\argmin_{\mathbf{D}}\sum_{i=1}^N\left\{\int_\Omega\|d_i(x)\|dx+\frac{r}{2}\int_\Omega \left\|d_i(x)-\nabla
u^{k}_i(x)-\frac{\lambda^k_{d_i}(x)}{r}\right\|^2dx\right\}.
\]
This is separable and the optimal solution of $d^{k+1}_i$ can be explicitly expressed using shrinkage operators. We simply compute
\begin{equation*}\label{solveD}
d^{k+1}_i(x) = \mathcal{S}\left(\nabla u_i^k(x)+\frac{\lambda_{d_i}^k(x)}{r},\frac{1}{r}\right),
\end{equation*}
where $\mathcal{S}$ is the shrinkage operator defined as
$$\mathcal{S}(v,\tau):={\frac{v}{\|v\|}}*\max\left({\|v\|-\tau,0}\right).$$
For the sake of simplicity, we denote this step as
\begin{align*}\vD^{k+1}=\mathcal{S}\left( \nabla \vU^{k}+{\frac{\Lambda_\vD^{k}}{r}},{\frac{1} {r}}\right).\end{align*}


\subsection*{$\mathbf{W}$-subproblem}
The subproblem for $\mathbf{W}$ is
equivalent to
\[\min\limits_{\mathbf{W}}\sum_{i=1}^N\left\{\frac{r}{2}\int_\Omega \left|w_i(x)-u^{k}_i(x)-\frac{\lambda^k_{w_i}(x)}{r}+\frac{\lambda|I(x) - c^{k}_i|}{r}\right|^2dx\right\}+\delta_\Delta(\mathbf{W}).\]
Since $\Delta$ is a convex simplex at any $x\in\Omega$, the solution is given by
\begin{equation*}\label{solveW}
\mathbf{W}^{k+1} = \Pi_\Delta \left(\left[u_i^{k}+\frac{\lambda_{w_i}^k}{r}-\frac{\lambda|I-c_i^{k}|}{r}\right]_{i=1}^N\right),
\end{equation*}
where $\Pi_\Delta$ is the projection onto the simplex $\Delta$, for which more details can be found in~\cite{chen2011projection}. We denote the step as
\begin{equation*}
\vW^{k+1} = \Pi_\Delta \left(\vU^k+{\frac{\Lambda^k_\vW} {r}}-{\lambda \frac{|I-\vC^k|}{r}}\right).
\end{equation*}



\subsection*{$\mathbf{C}$-subproblem}\label{psi}
The subproblem for $\mathbf{C}$ is
\begin{equation*}\label{subC}
\vC^{k+1}=\argmin_{\mathbf{C}}\sum_{i=1}^N\int_\Omega \left|I(x) - c_i\right|w^{k+1}_i(x) \ dx.
\end{equation*}
It is separable, and $c^{k+1}_i$ can be solved independently.
The optimality condition for each $c^{k+1}_i$ is
\begin{equation}\label{subC2}
0\in -\int_\Omega \partial|I(x)-c^{k+1}_i|w^{k+1}_i(x)dx.
\end{equation}
Because the right-hand side of~\eqref{subC2} is maximal monotone~\cite{bauschke2011convex}, the bisection method and ADMM are applied to solve it~\cite{fuzzy_median,TVSEGL1}. The next lemma provides a new way to find an optimal solution for discrete cases.

\begin{lemma}\label{lemma1}
Given a finite non-decreasing sequence $\{I_{[j]}\}_{j=1}^n$, i.e.,
$$I_{[1]}\le I_{[2]}\le ... \le I_{[n]},$$
and a non-negative sequence $\{w_{[j]}\}_{j=1}^n$ with $A=\sum_{j=1}^nw_{[j]}>0$, the optimal solution set for
\begin{equation}\label{lemma_fcm}
\min_c\sum_{j=1}^n\left|I_{[j]}-c\right|w_{[j]},
\end{equation}
is $\left[I_{[j^*]},I_{[j_*+1]}]\right]$, where $j^*$ and $j_*$ satisfy
\begin{align*}
A-2\sum_{j=1}^{j^*}{w}_{[j]}\leq 0< A-2\sum_{j=1}^{j^*-1}{w}_{[j]},\\
A-2\sum_{j=1}^{j_*+1}{w}_{[j]}<0\leq A-2\sum_{j=1}^{j_*}{w}_{[j]}.
\end{align*}
The fuzzy median of $\{I_j\}_{j=1}^n$ with respect to the weight $\{w_{[j]}\}_{j=1}^n$~\cite{fuzzy_median}, which is defined as $(I_{[j^*]}+I_{[j_*+1]}])/2$, is an optimal solution. If, in addition, there exists $j^*$ such that
\begin{align*}
A-2\sum_{j=1}^{j^*}{w}_{[j]} < 0 < A-2\sum_{j=1}^{j^*-1}{w}_{[j]},
\end{align*}
then $j_*=j^*-1$, and the optimal solution is unique.
\end{lemma}

\begin{proof}
The optimality condition of~\eqref{lemma_fcm} is
\begin{equation*}
0\in h(c):=\sum_{j=1}^n\partial|I_{[j]}-c|{{w_{[j]}}}.
\end{equation*}
We can see that $h(c)$ is non-increasing and it can be expressed as
\begin{align*}
h(c) = \left\{\begin{array}{ll}A, & c< I_{[1]},\\
{[A-2{w}_{[1]},A]}, & c = I_{[1]},\\
A-2{w}_{[1]}, & c\in (I_{[1]},I_{[2]}) ,\\
\cdots&\\
{[A-2\sum_{j=1}^s{w}_{[j]},A-2\sum_{j=1}^{s-1}{w}_{[j]}]}, & c = I_{[s]},\\
A-2\sum_{j=1}^s{w}_{[j]}, & c_i\in (I_{[s]},I_{[s+1]}) ,\\
\cdots &\\
-A, & c>I_{[n]}.\\
\end{array}\right.
\end{align*}
The range of $h$ is $[-A,A]$, and $h(c)$ can take multiple values when $c=I_{[j]}$ for any $j$ with $w_{[j]}\neq0$. Therefore, we can always find $j^*$ such that
\begin{align*}
A-2\sum_{j=1}^{j^*}{w}_{[j]}\leq 0<A-2\sum_{j=1}^{j^*-1}{w}_{[j]}.
\end{align*}
Thus $0\in h(I_{[j^*]})=\left[A-2\sum_{j=1}^{j^*}{w}_{[j]},A-2\sum_{j=1}^{j^*-1}{w}_{[j]}\right]$, and $I_{[j^*]}$ is an optimal solution. In addition, we have that $h(c)>0$ when $c<I_{[j^*]}$. Similarly, we can find $j_*$ such that
\begin{align*}
A-2\sum_{j=1}^{j_*+1}{w}_{[j]}<0\leq A-2\sum_{j=1}^{j_*}{w}_{[j]},
\end{align*}
and $I_{[j_*+1]}$ is an optimal solution. In addition, we have that $h(c)<0$ when $c>I_{[j_*+1]}$. Then $h(c)$ being non-increasing gives us that the set of optimal solutions for~\eqref{lemma_fcm} is $\left[I_{[j^*]},I_{[j_*+1]}\right]$. When $j^*=j_*+1$, the optimal solution is unique.\qed
\end{proof}

{\bf Remark:} When there are missing pixels in images, we can put a mask on the data fidelity term as in image inpainting problems~\cite{shen2002mathematical} or assign a value to each missing pixel. In~\cite{TVFCM}, the missing pixels are assigned with zero, and $\mathbf{C}$ changes based on the percentage of missing pixels. While, this lemma tells us that assigning 0 or 255 (for grayscale images) randomly to these missing pixels will not change the value of $c$ with a high probability because $c_i$ is a weighted median. Also, by assigning 0 or 255 randomly, this algorithm is able to deal with cases where more than half of the pixels are missing. See the numerical experiments in Section~\ref{sec:experiment}.

Assume that $I$ and $w_i^{k+1}$ are vectors in $\mathbb{R}^n$, where $n$ is the total number of pixels. We can reorder the components of $I$ and $w^{k+1}_i$ such that the monotonicity of $I$ in Lemma~\ref{lemma1} is satisfied. If there are multiple optimal solutions for $c_i$, we choose the smallest one. We denote this step as
$$
\vC^{k+1}=\psi(\vW^{k+1}).
$$

\subsection*{$\mathbf{U}$-subproblem}
The subproblem for $\mathbf{U}$ is equivalent to
\begin{align*}
\vU^{k+1}=\argmin\limits_{\mathbf{U}}&\sum_{i=1}^N\int_\Omega \left\|\nabla u_i(x)-d_i^{k+1}(x)+\frac{\lambda_{d_i}^k(x)}{r}\right\|^2dx\\
&+\int_\Omega \left|u_i(x)-w_i^{k+1}(x)+\frac{\lambda_{w_i}^k(x)}{r}\right|^2dx.
\end{align*}
It is separable for $u^{k+1}_i$, and, from the following first order optimality condition for each $u^{k+1}_i$:
\begin{align*}
\nabla^T\left(\nabla u^{k+1}_i(x)-d^{k+1}_i(x)+\frac{\lambda^k_{d_i}(x)}{r}\right)+\left(u^{k+1}_i(x)-w^{k+1}_i(x)+\frac{\lambda^k_{w_i}(x)}{r}\right)=0,
\end{align*}
we can derive the closed-form solution of $u^{k+1}_i$ from the equation:
\begin{equation*}
\left(\nabla^T\nabla+\vI\right)u^{k+1}_i(x) = \nabla^Td^{k+1}_i(x)+w^{k+1}_i(x)-\frac{\nabla^T\lambda^k_{d_i}(x)}{r}-\frac{\lambda^k_{w_i}(x)}{r}.
\end{equation*}
A diagonalization trick can be applied to find $u^{k+1}_i$ efficiently~\cite{wang2008new}. We denote this step as
\begin{align*}
\vU^{k+1}=\left(\nabla^T\nabla+\vI\right)^{-1}\left(\nabla^T\vD^{k+1}+\vW^{k+1}-\frac{\nabla^T\Lambda^k_{\vD}}{r}-\frac{\Lambda^k_{\vW}}{r}\right).
\end{align*}

{\bf Updating dual variables.}  We denote the steps as
\begin{align*}
\Lambda_{\mathbf{D}}^{k+1}=&\Lambda_{\mathbf{D}}^{k}+r\left(\nabla \mathbf{U}^{k+1}-\mathbf{D}^{k+1}\right),\\
\Lambda_{\mathbf{W}}^{k+1}=&\Lambda_{\mathbf{W}}^{k}+r\left(\mathbf{U}^{k+1}-\mathbf{W}^{k+1}\right).
\end{align*}

Combining all these steps together, we obtain the proposed L1 Fuzzy Segmentation algorithm (L1FS) in Algorithm 1 for solving~\eqref{TVFCML1}.

\begin{algorithm}
\begin{minipage}{12cm}
\begin{itemize}
\item
Initialization: $\mathbf{U}^{0}$ and $\mathbf{C}^0$ \mbox{are specified}, 
$\Lambda_{\mathbf{D}}^0=\mathbf{0}, \Lambda_{\mathbf{W}}^0=\mathbf{0}$.
\item
For $k=0,1,2,\cdots$, repeat until the stopping criterion is reached.
\begin{eqnarray*}
\mathbf{D}^{k+1}\ &=&\ \mathcal{S}\left(\nabla \mathbf{U}^{k}+\frac{\Lambda_{\mathbf{D}}^k}{r},\ \frac{1}{r}\right),\\
\mathbf{W}^{k+1}\ &=&\ \Pi_\Delta \left(\mathbf{U}^{k}+\frac{\Lambda_{\mathbf{W}}^k}{r}-\frac{\lambda|I-\mathbf{C}^{k}|}{r}\right),\\
\mathbf{C}^{k+1}\ &=&\ \psi(\vW^{k+1}),\\
\mathbf{U}^{k+1}\ &=&\ \left(\nabla^T\nabla+\vI\right)^{-1}\left(\nabla^T\mathbf{D}^{k+1}+\mathbf{W}^{k+1}
-\frac{\nabla^T\Lambda_\mathbf{{D}}^k}{r}-\frac{\Lambda_\mathbf{{W}}^k}{r}\right) ,\\
\Lambda_{\mathbf{D}}^{k+1}\ &=&\ \Lambda_{\mathbf{D}}^{k}+r\left(\nabla \mathbf{U}^{k+1}-\mathbf{D}^{k+1}\right),\\
\Lambda_{\mathbf{W}}^{k+1} \ &=&\ \Lambda_{\mathbf{W}}^{k}+r\left(\mathbf{U}^{k+1}-\mathbf{W}^{k+1}\right).
\end{eqnarray*}
\item
Output: $\mathbf{C}^{k+1}, \mathbf{U}^{k+1}$.
\end{itemize}
\end{minipage}\\
\caption{ The proposed L1FS algorithm}
\end{algorithm}

{\bf Remark:} Though there are four variables $ \vD$, $\vW$, $\vC$ and $\vU$, they can be grouped into two variables $(\vD,\vW)$ and $(\vC,\vU)$ and  the subproblems for these two variables can be decoupled into the four subproblems. So it is essentially a two block ADMM applied on a nonconvex optimization problem.

Because problem~\eqref{TVFCML1} is nonconvex, a good initial guess of $\mathbf{U}^{0}$ and $\mathbf{C}^0$,
which can be obtained from other methods using fuzzy membership functions, is helpful for the stability of L1FS.
Thus we update $\vD^1$ and $\vW^1$ first because both of them can use the initial guess of $\vU^{0}$.

\subsection{Convergence analysis}
If $\vC$ is given, problem~\eqref{TVFCML1} is convex. Then, the algorithm is a standard ADMM by considering $(\vD,\vW)$ together, and its convergence is guaranteed~\cite{ADMM}. In this section, we give a convergence result of Algorithm 1 for unknown $\vC$. To simplify notations, let us define the sextuples:
$$\mathbf{Z}:=(\mathbf{D,W,C,U},\Lambda_\mathbf{D},\Lambda_\mathbf{W}).$$
A point $\mathbf{Z}$ is a KKT point of (\ref{cons}) if it satisfies the following KKT conditions:
\begin{subequations}
\begin{align}\label{KKT1}
\partial \|d_i^*(x)\|-{\lambda_{d_i}^*(x)}&\ni 0,\\
\lambda |I-\mathbf{C}|-\Lambda_{\mathbf{W}}+\partial \delta_\Delta(\mathbf{W}) &\ni\mathbf{0},\\
\int_\Omega \partial |I(x)-c_i^*|w^*_i(x)dx & \ni 0,\\
\nabla^T\lambda_{d_i}(x)+\lambda_{w_i}(x)&=0,\\
\nabla u_i(x)-d_i(x) &=0,\\
\label{KKT2}u_i(x)-w_i(x) &=0.
\end{align}
\end{subequations}
where $\partial\|\cdot\|$ and $\partial \delta_\Delta(\cdot)$ are subdifferentials of $\|\cdot\|$ and $\delta_\Delta(\cdot)$, respectively.

\begin{theorem}[Convergence to stationary points of L1FS]\label{conKKT}
Let $\left\{\mathbf{Z}^k\right\}_{k=1}^\infty$ be a sequence generated by
Algorithm 1 that satisfies the condition
\begin{equation*}
\lim\limits_{k\rightarrow\infty}\left(\mathbf{Z}^{k+1}-\mathbf{Z}^k\right)=\mathbf{0}.
\end{equation*}
Then any accumulation point of $\left\{\mathbf{Z}^k\right\}_{k=1}^\infty$ is a KKT point of problem (\ref{cons}). Consequently,
any accumulation point of $\left(\mathbf{C}^k,\mathbf{U}^k\right)$ is a stationary point of problem (\ref{TVFCML1}).
\end{theorem}


\begin{proof}
From the updating formulas in Algorithm 1, we obtain the following inequalities related to the subsequent iteration:
\begin{subequations}\label{difall}
\begin{align}\label{difU}
\mathbf{D}^{k+1}-\mathbf{D}^k &= \mathcal{S}\left(\nabla \mathbf{U}^{k}+\frac{\Lambda_{\mathbf{D}}^{k}}{r},\ \frac{1}{r}\right)-\mathbf{D}^k,\\
\mathbf{W}^{k+1}-\mathbf{W}^k &= \Pi_\Delta \left(\mathbf{U}^{k}+\frac{\Lambda_{\mathbf{W}}^{k}}{r}-\frac{\lambda|I-\mathbf{C}^{k}|}{r}\right)-\mathbf{W}^k,\\
\mathbf{C}^{k+1}-\mathbf{C}^k &= \psi(\vW^{k+1})-\mathbf{C}^k,\\
\left(\nabla^T\nabla+I\right)\left(\vU^{k+1}-\vU^{k}\right) &= \nabla^T(\vD^{k+1}-\nabla \vU^k)+(\vW^{k+1}-\vU^k)\\
&\ \ \ \  -\frac{\nabla^T\Lambda_{\vD}^k}{r}-\frac{\Lambda_{\vW}^k}{r},\\
\Lambda_{\mathbf{D}}^{k+1}-\Lambda_{\mathbf{D}}^k &= r\left(\nabla\mathbf{U}^{k+1}-\mathbf{D}^{k+1}\right),\\
\label{diflW}\Lambda_{\mathbf{W}}^{k+1}-\Lambda_{\mathbf{W}}^k &= r\left(\mathbf{U}^{k+1}-\mathbf{W}^{k+1}\right).
\end{align}
\end{subequations}

By the assumption $\lim\limits_{k\rightarrow\infty}\left(\mathbf{Z}^{k+1}-\mathbf{Z}^k\right)=\mathbf{0}$,
the left-hand side and  right-hand side of each equality in~\eqref{difall} go to zero as $k\rightarrow\infty$.
Then we have
\begin{subequations}\label{limall}
\begin{align}\label{lim1}
\mathcal{S}\left(\nabla \mathbf{U}^{k}+\frac{\Lambda_{\mathbf{D}}^{k}}{r},\ \frac{1}{r}\right)-\mathbf{D}^k &\rightarrow \mathbf{0},\\
\Pi_\Delta \left(\mathbf{U}^{k}+\frac{\Lambda_{\mathbf{W}}^{k}}{r}-\frac{\lambda|I-\mathbf{C}^{k}|}{r}\right)-\mathbf{W}^k &\rightarrow \mathbf{0},\\
\psi(\mathbf{W}^{k+1})-\mathbf{C}^k& \rightarrow \mathbf{0},\\
\nabla^T(\vD^{k+1}-\nabla \vU^k)+(\vW^{k+1}-\vU^k)-\frac{\nabla^T\Lambda_{\vD}^k}{r}-\frac{\Lambda_{\vW}^k}{r} &\rightarrow \mathbf{0},\\
\nabla \vU^{k+1}-\vD^{k+1} &\rightarrow \mathbf{0},\\
\label{lim2}\mathbf{U}^{k+1}-\mathbf{W}^{k+1} &\rightarrow \mathbf{0}.
\end{align}
\end{subequations}

Assume  $\mathbf{Z}^*=(\mathbf{D^*,W^*,C^*,U^*},\Lambda_\mathbf{D}^*,\Lambda_\mathbf{W}^*)$ is an accumulation point of $\mathbf{Z}^k$. \eqref{limall} gives us
\begin{subequations}
\begin{align}
\label{limm1}\mathcal{S}\left(\nabla \mathbf{U}^{*}+\frac{\Lambda_{\mathbf{D}}^{*}}{r},\ \frac{1}{r}\right)-\mathbf{D}^*&= \mathbf{0},\\
\label{limm2}\Pi_\Delta \left(\mathbf{U}^{*}+\frac{\Lambda_{\mathbf{W}}^{*}}{r}-\frac{\lambda|I-\mathbf{C}^{*}|}{r}\right)-\mathbf{W}^*&= \mathbf{0},\\
\label{limm3}\psi(\mathbf{W}^{*})-\mathbf{C}^* &= \mathbf{0},\\
\label{limm4}\nabla^T\Lambda_{\vD}^*+\Lambda_{\vW}^* &= \mathbf{0},\\
\label{limm5}\nabla\mathbf{U}^{*}-\mathbf{D}^{*} &= \mathbf{0},\\
\label{limm6}\mathbf{U}^{*}-\mathbf{W}^{*} &= \mathbf{0}.
\end{align}
\end{subequations}
By (\ref{limm1}), it follows that $\mathbf{D}^*$ is a solution of the minimization problem
\begin{align*}
\min_{\mathbf{D}}\sum_{i=1}^N\left\{\int_\Omega\|d_i(x)\|dx+\frac{r}{2}\int_\Omega \left\|d_i(x)-\nabla
u^{*}_i(x)-\frac{\lambda^*_{d_i}(x)}{r}\right\|^2dx\right\}.
\end{align*}
Thus $\vD^*$ satisfies the first order optimality condition
\begin{equation*}
0\in \partial \|d^*_i(x)\|+r\left(d_i^*(x)-\nabla u_i^*(x)-\frac{\lambda_{d_i}^*(x)}{r}\right).
\end{equation*}
Using (\ref{limm5}), we simplify the above condition as
\begin{equation}\label{limmm3}
0\in \partial \|d_i^*(x)\|-{\lambda_{d_i}^*(x)}.
\end{equation}
By (\ref{limm2}), $\mathbf{W}^*$ is a solution of the following minimization problem:
\begin{align*}
\min\limits_{\mathbf{W}}\sum_{i=1}^N\left\{\frac{r}{2}\int_\Omega \left|w_i(x)-u^{*}_i(x)-\frac{\lambda^*_{w_i}(x)}{r}+\frac{\lambda|I(x) - c^{*}_i|}{r}\right|^2dx\right\}+\delta_\Delta(\mathbf{W}).
\end{align*}
Hence $\mathbf{W}^*$ satisfies the optimality condition
$$
\mathbf{0}\in \partial \delta_\Delta(\mathbf{W}^*)+r(\mathbf{W}^*-\mathbf{U}^*)-\Lambda_\mathbf{W}^*+\lambda|I-\mathbf{C}^*|.
$$
Together with the equality in (\ref{limm6}), the above equation can be simplified as
\begin{equation}\label{limmm4}
\mathbf{0}\in \partial \delta_\Delta(\mathbf{W}^*)-\Lambda_\mathbf{W}^*+\lambda|I-\mathbf{C}^*|.
\end{equation}
The equation (\ref{limm3}) implies that $\mathbf{C}^*$ is a solution of the minimization problem
\begin{equation*}
\min_{\mathbf{C}}\sum_{i=1}^N\int_\Omega \left|I(x) - c_i\right|w^*_i(x) \ dx,
\end{equation*}
and the optimal condition is
\begin{equation}\label{limmm2}
0\in \int_\Omega \partial |I(x)-c_i^*|w^*_i(x)dx.
\end{equation}

The equivalence of equations (\ref{limm1})-(\ref{limm3}) with (\ref{limmm3})-(\ref{limmm2}), together with
equations (\ref{limm4})-(\ref{limm6}) shows that the accumulation point
$\mathbf{Z}^*$ satisfies the KKT condition in equations (\ref{KKT1})-(\ref{KKT2}), thus $\mathbf{Z}^*$ is a KKT point
of problem (\ref{cons}).

Since problem (\ref{cons}) and problem (\ref{TVFCML1}) are equivalent, the convergence of $\left(\mathbf{C}^k,\mathbf{U}^k\right)$ in the statement follows directly.\qed\end{proof}

The convergence analysis is motivated by~\cite{zhang2010alternating}. This theorem implies that whenever $\left\{\mathbf{Z}^k\right\}_{k=1}^\infty$ converges, it converges to a KKT point of problem (\ref{cons}). However, since the minimization problem (\ref{cons}) is nonconvex, the KKT point is not necessary to be an optimal solution.

\section{Experimental results}\label{sec:experiment}
In order to demonstrate the effectiveness of the proposed method, we compare our method with some existing methods on both synthetic and real images. These methods (FCM, FCM\_S2, FLICM, L2FS, L2L0, and L1SS) are discussed in Section~\ref{sec:preliminary}. Since the segmentation models in these methods are not jointly convex, and they are easy to get stuck in local minima, ``good'' initialization is crucial for all algorithms, especially when the given image is noisy. For FCM, FCM\_S2, and FLICM, the initial $\mathbf{U}$ is uniformly distributed in $[0,1]$ and normalized to satisfy the sum-to-one constraint. While for TV based methods L2FS and L1FS, one can also use the results of FCM and FCM\_S2 as the initial $\mathbf{U}$ and $\mathbf{C}$. Here we consider three ways for choosing the initial $\mathbf{U}$ and $\mathbf{C}$: the result of FCM, the result of FCM\_S2, and $\mathbf{U}$  as functions uniformly distributed in $[0,1]$ and $\mathbf{C}$ as the result of FCM. In all the experiments, we choose the one with the highest performance among all the three initializations. For L1SS, we use the initialization method as described in the original paper. The stopping criterion, which is the same for all the methods except L1SS, is defined as
\begin{equation*}\label{eps}
\frac{\|\mathbf{U}^{k+1}-\mathbf{U}^k\|_2}{\|\mathbf{U}^k\|_2}<\epsilon
\end{equation*}
where $\epsilon$ is a very small number. For L1SS, this stopping criterion does not work since L1SS leads to contrast loss due to the error in calculating class centers $\{c_i\}_{i=1}^N$ at early iterations. However, the contrast of L1SS will be enhanced gradually if the number of iterations increases. To gain satisfactory results, we choose to stop L1SS by setting the maximum number of iterations to be 1000. 

The parameters of the methods being compared are set as follows. In FCM, we set $p=2$. In FCM\_S2, we set $p=2,\alpha=5$, and the window size for the median filter as $5\times 5$. In FLICM, we set $p=2$ and the window size of the neighborhoods as $3\times3$ or $5\times5$, which depends on the noise level.
However, the weight parameter $\lambda$ for L2FS, L1SS, and L1FS are tuned for each experiment to achieve optimal results. In all experiments, the range of $\lambda$ for L2FS is $[0.00005, 0.0005]$, for L1SS is $[0.03, 1]$, and for L1FS is $[0.001, 0.05]$.
For all methods, $\epsilon=10^{-6}$ for the two-phase segmentation and $\epsilon=10^{-4}$ for the multiphase segmentation. We use the default parameters in L2L0 except that we set $\lambda=10$.

The clustering results of all methods are obtained by checking the maximum value of their membership functions. Then we display the recovered piecewise constant image as the final result. To compare segmentation results quantitatively, we consider the Segmentation Accuracy (SA) defined by
$$
SA = \frac{\# \mbox{correctly classified pixels}}{\# \mbox{all pixels}}.
$$

\begin{figure}[!htbp]
\begin{center}
\subfloat[]{\includegraphics[width=1.893cm]{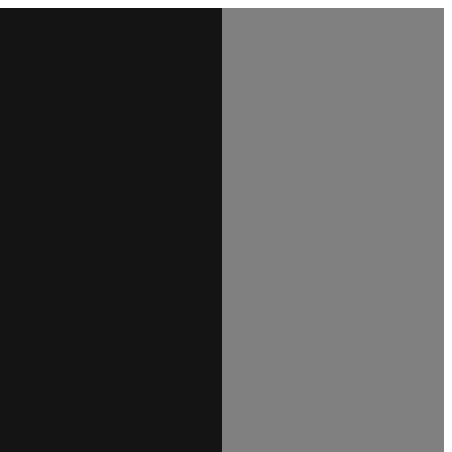}\label{fig:1a}}\
\subfloat[]{\includegraphics[width=1.893cm]{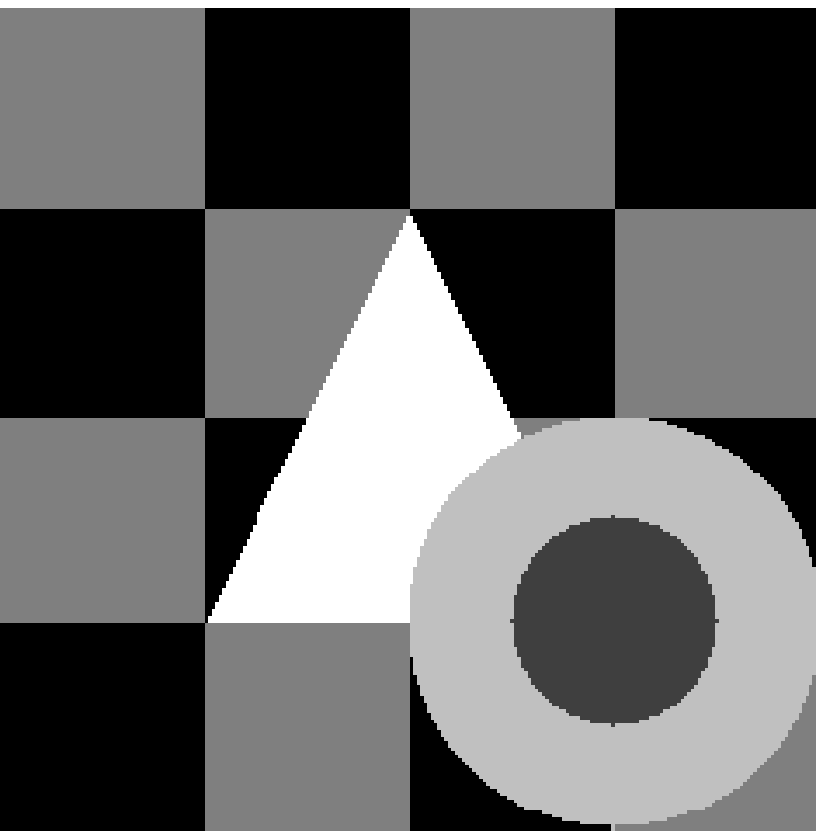}\label{fig:1b}}\
\subfloat[]{\includegraphics[width=1.893cm]{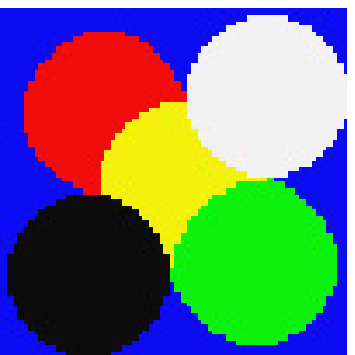}\label{fig:1c}}
\end{center}
\caption{Synthetic piecewise constant test images. (a) A two-phase grayscale image with intensities 20 and 128, size $128\times 128$;
(b) A five-phase grayscale image with intensities 0, 63, 127, 192, and 255, size $235\times237$; (c) A six-phase color image with color vectors (12 11 242), (242 12 11), (242 241 242), (243 241 12), (12 12 12), and (12 242 12), size $100\times100$. }
\label{fig:1}
\end{figure}

The synthetic piecewise constant test images are displayed in Fig.~\ref{fig:1}.
The noisy images are contaminated by three types of noise: Gaussian Noise (GN) with zero mean and standard deviation varying from 10 to 80, Salt-and-Pepper Impulse Noise (SPIN) with 10\% to 60\% pixels contaminated, and Random-Valued Impulse Noise (RVIN) with 10\% and 60\% pixels contaminated.


\subsection{Test on Fig.~\ref{fig:1a}}
\begin{table}[htbp]\centering\scriptsize{
 \begin{threeparttable}
 \caption{ SA performance of different methods applied on Fig.~\ref{fig:1a} contaminated by different levels of GN, SPIN, and RVIN.} \label{tab:1}
\begin{tabular}{lcccccccc} \toprule GN ($\sigma$) & 10 & 20&30 & 40 & 50&60&70&80\\
\midrule FCM& \textbf{1}&0.9970 &0.9642 &0.9146 &0.8627 &0.8162&0.7811&0.7524\\
FLICM & \textbf{1}&0.9999&0.9996&0.9990&0.9975&0.9968&0.9960&0.9954\\
L2FS&\textbf{1}&\textbf{1}&\textbf{1}&\textbf{1}&\textbf{1}&\textbf{1}&0.9999&\textbf{0.9998}\\
L1SS&\textbf{1}&\textbf{1}&\textbf{1}&0.9998&0.9996&0.9984&0.9975&0.9965\\
L1FS &\textbf{1}&\textbf{1}&\textbf{1}&\textbf{1}&\textbf{1}&\textbf{1}&\textbf{1}&\textbf{0.9998}\\
\midrule SPIN (\%) & 10\% & 20\%&30\% & 40\% & 50\%&60\%&&\\
\midrule FCM& 0.9480&0.8983&0.8478&0.7979&0.7486&0.6980&&\\
FLICM & 0.9984&0.9921&0.9738&0.8002&0.7313&0.6554&&\\
L2FS&0.9999&0.9998&0.9982&0.9983&-&-&&\\
L1SS&0.9998&0.9990&0.9977&0.9967&0.9956&0.9953&&\\
L1FS &\textbf{1}&\textbf{1}&\textbf{1}&\textbf{1}&\textbf{1}&\textbf{0.9995}&&\\
\midrule RVIN (\%) & 10\% & 20\%&30\% & 40\% & 50\%&60\%&&\\
\midrule FCM\_S2& 0.9985&0.9972&0.9945&0.9862&0.9630&0.9042&&\\
FLICM & 0.9987&0.9985&0.9970&0.9958&0.9948&0.9919&&\\
L2FS&\textbf{1}&\textbf{1}&\textbf{1}&0.9998&0.9995&0.9974&&\\
L1SS&1&0.9995&0.9985&0.9979&0.9966&0.9891&&\\
L1FS &\textbf{1}&\textbf{1}&\textbf{1}&\textbf{1}&\textbf{1}&\textbf{0.9976}&&\\
\bottomrule
\end{tabular}
\end{threeparttable}}
\end{table}

\begin{figure}[!htbp]
\begin{center}
\subfloat[{GN $\sigma$=30}]{\includegraphics[width=1.893cm]{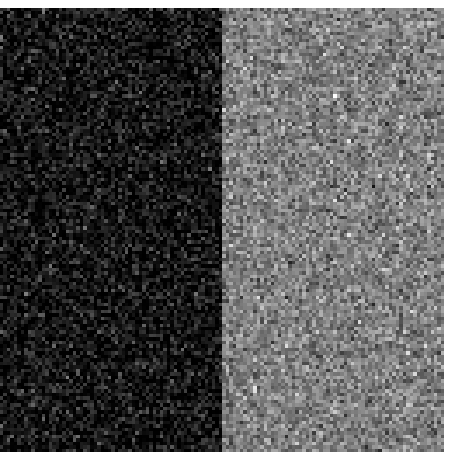}\label{fig:2a}}\
\subfloat[{0.9642}]{\includegraphics[width=1.893cm]{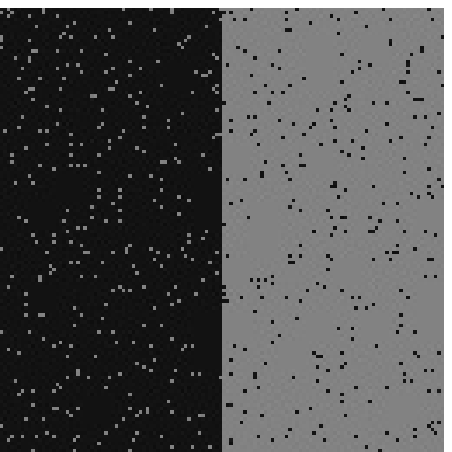}\label{fig:2b}}\
\subfloat[{0.9996}]{\includegraphics[width=1.893cm]{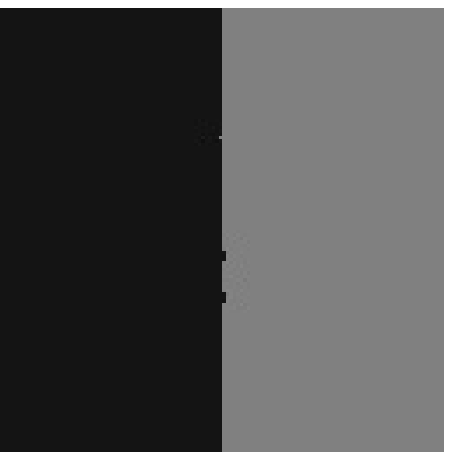}}\label{fig:2c}\
\subfloat[{1}]{\includegraphics[width=1.893cm]{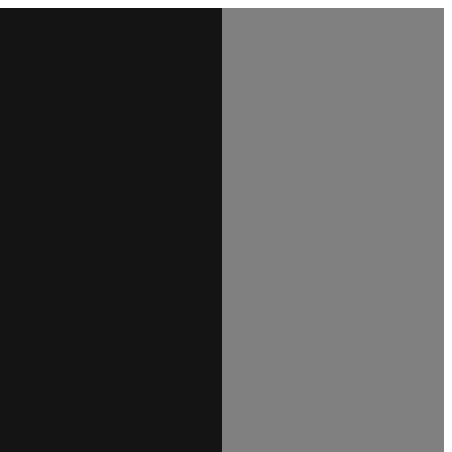}\label{fig:2d}}\
\subfloat[{1}]{\includegraphics[width=1.893cm]{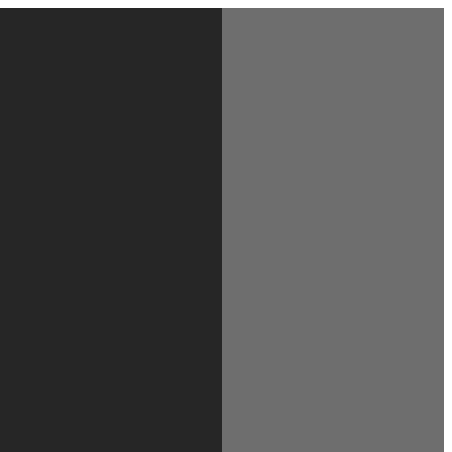}\label{fig:2e}}\
\subfloat[{1}]{\includegraphics[width=1.893cm]{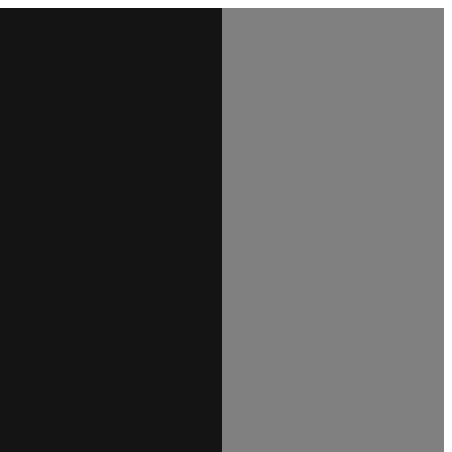}\label{fig:2f}}\\
\subfloat[{GN $\sigma$=60}]{\includegraphics[width=1.893cm]{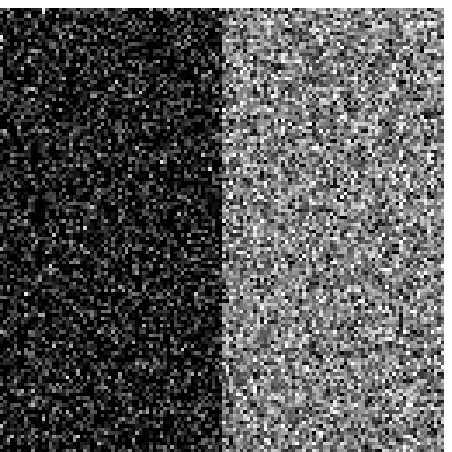}\label{fig:2g}}\
\subfloat[{0.8162}]{\includegraphics[width=1.893cm]{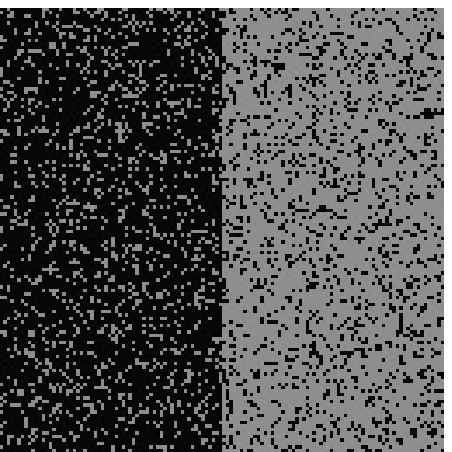}\label{fig:2h}}\
\subfloat[{0.9968}]{\includegraphics[width=1.893cm]{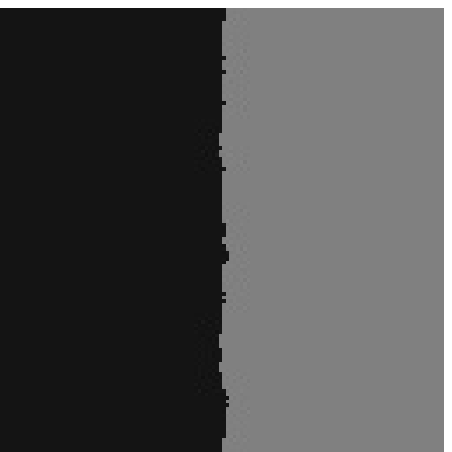}\label{fig:2i}}\
\subfloat[{1}]{\includegraphics[width=1.893cm]{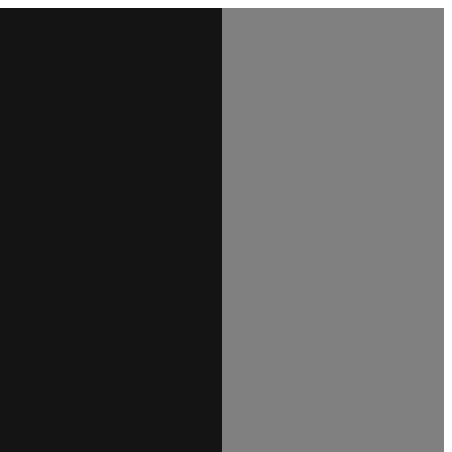}\label{fig:2j}}\
\subfloat[{0.9984}]{\includegraphics[width=1.893cm]{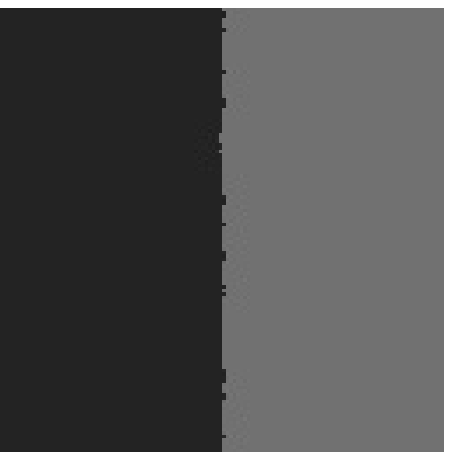}\label{fig:2k}}\
\subfloat[{1}]{\includegraphics[width=1.893cm]{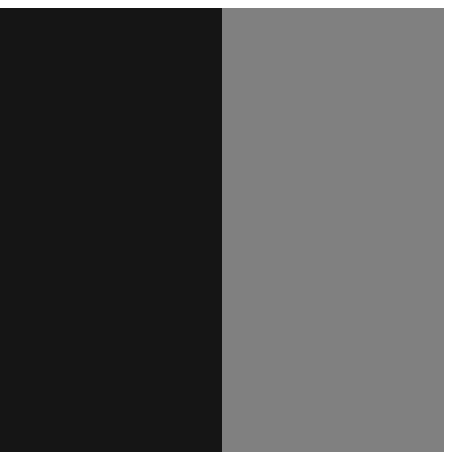}\label{fig:2l}}
\end{center}
\caption{Two-phase segmentation on the synthetic image Fig.~\ref{fig:1a} with Gaussian noise.
First column: images contaminated by Gaussian noise with standard deviations 30 and 60, respectively; Second column to last column: results of FCM, FLICM, L2FS, L1SS, and L1FS, respectively. The SA values are reported below each segmentation result.}\label{fig:2}
\end{figure}

\begin{figure}[!htbp]
\begin{center}
\subfloat[{SPIN 20\%} ]{\includegraphics[width=1.893cm]{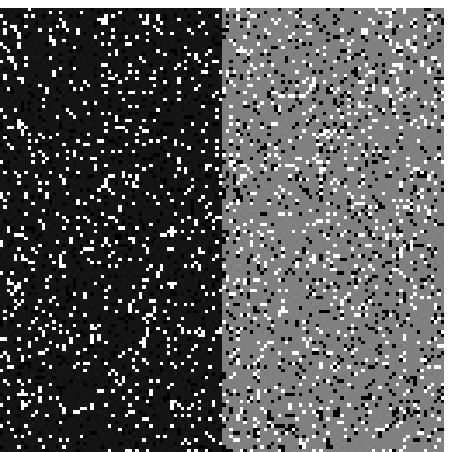}\label{fig:3a}}\
\subfloat[{0.8983}]{\includegraphics[width=1.893cm]{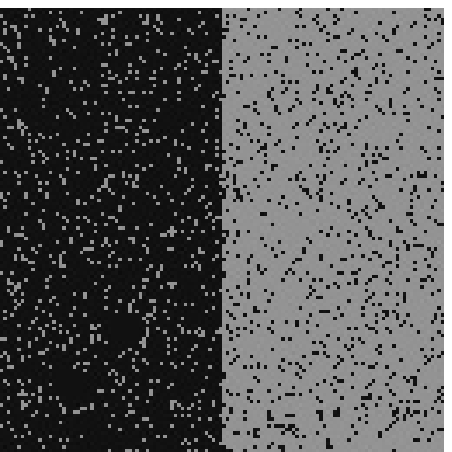}\label{fig:3b}}\
\subfloat[{0.9921}]{\includegraphics[width=1.893cm]{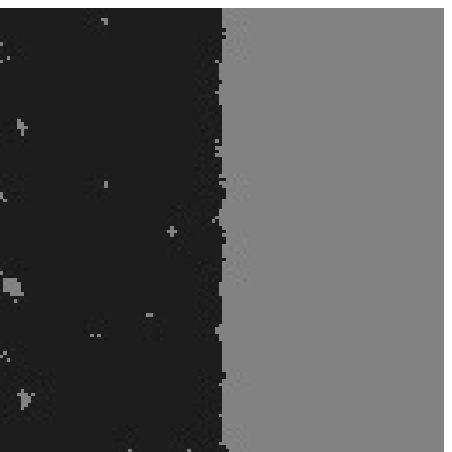}\label{fig:3c}}\
\subfloat[{0.9998}]{\includegraphics[width=1.893cm]{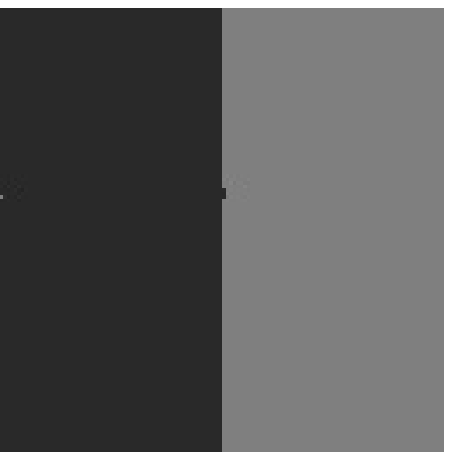}\label{fig:3d}}\
\subfloat[{0.9990}]{\includegraphics[width=1.893cm]{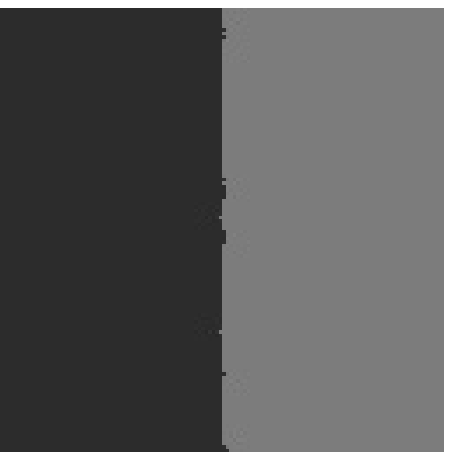}\label{fig:3e}}\
\subfloat[{1}]{\includegraphics[width=1.893cm]{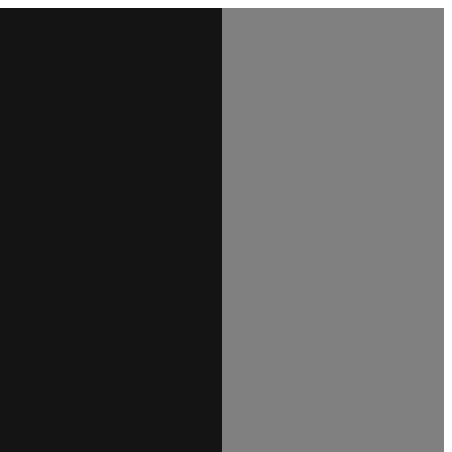}\label{fig:3f}}\\
\subfloat[{SPIN 40\%}]{\includegraphics[width=1.893cm]{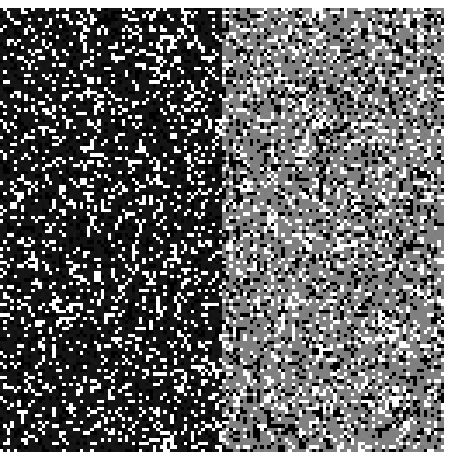}\label{fig:3g}}\
\subfloat[{0.7979}]{\includegraphics[width=1.893cm]{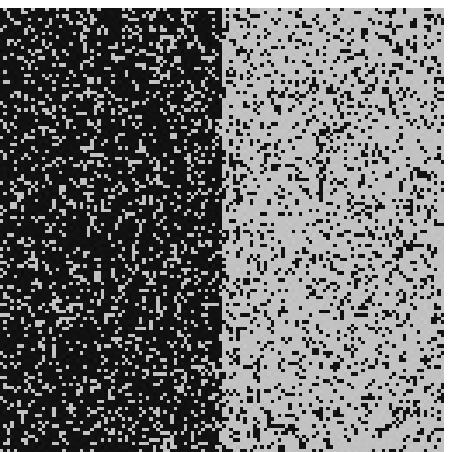}\label{fig:3h}}\
\subfloat[{0.8002}]{\includegraphics[width=1.893cm]{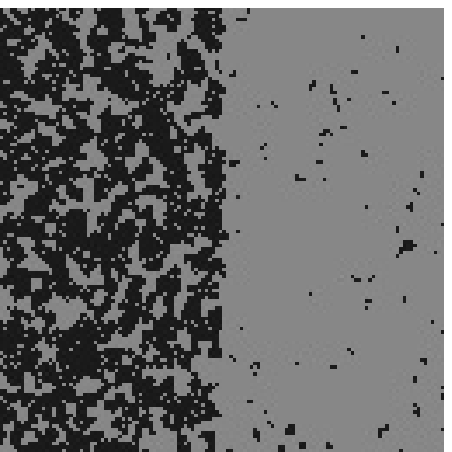}\label{fig:3i}}\
\subfloat[{0.9983}]{\includegraphics[width=1.893cm]{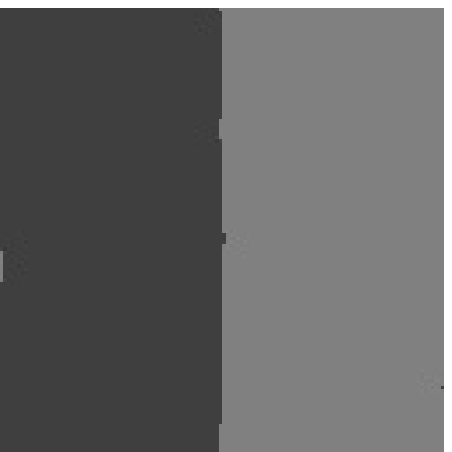}\label{fig:3j}}\
\subfloat[\small{0.9967}]{\includegraphics[width=1.893cm]{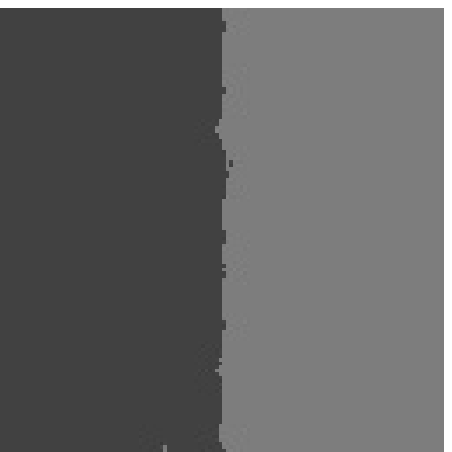}\label{fig:3k}}\
\subfloat[\small{1}]{\includegraphics[width=1.893cm]{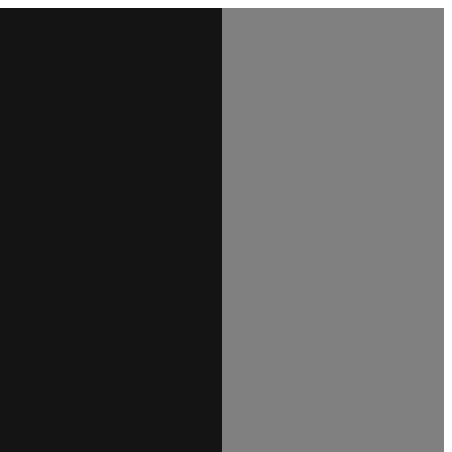}\label{fig:3l}}
\end{center}
\caption{Two-phase segmentation on the synthetic image Fig.~\ref{fig:1a} with SPIN.
First column: images contaminated by 20\% and 40\% SPIN, respectively; Second column to last column: results of FCM, FLICM, L2FS, L1SS,
and L1FS, respectively. The SA values are reported below each segmentation result.}\label{fig:3}
\end{figure}

\begin{figure}[!htbp]
\begin{center}
\subfloat[{RVIN 20\% }]{\includegraphics[width=1.893cm]{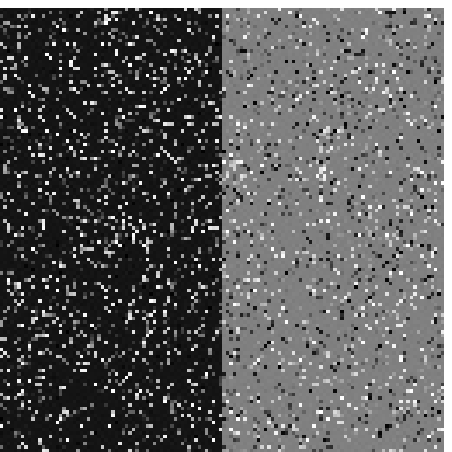}\label{fig:4a}}\
\subfloat[{0.9972}]{\includegraphics[width=1.893cm]{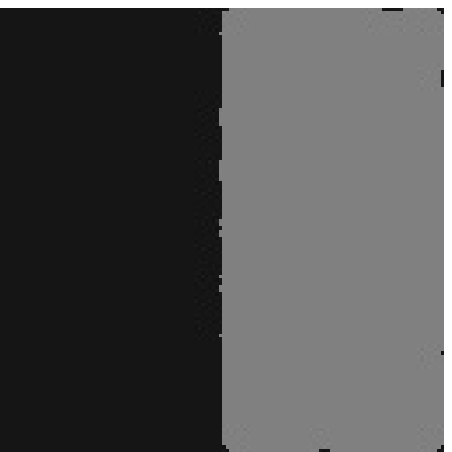}\label{fig:4b}}\
\subfloat[{0.9985}]{\includegraphics[width=1.893cm]{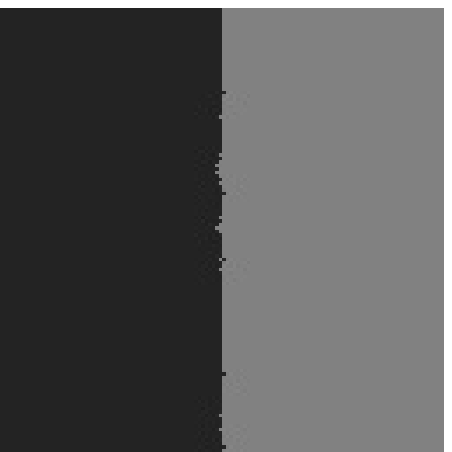}\label{fig:4c}}\
\subfloat[{1}]{\includegraphics[width=1.893cm]{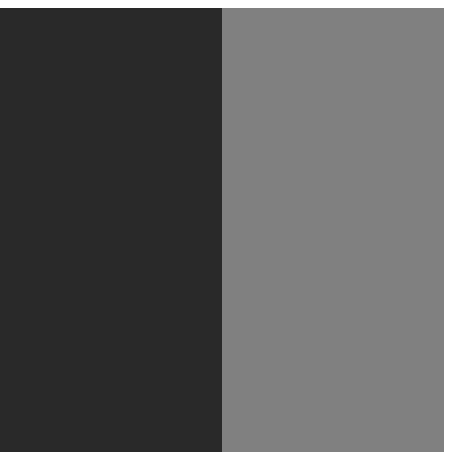}\label{fig:4d}}\
\subfloat[{0.9995}]{\includegraphics[width=1.893cm]{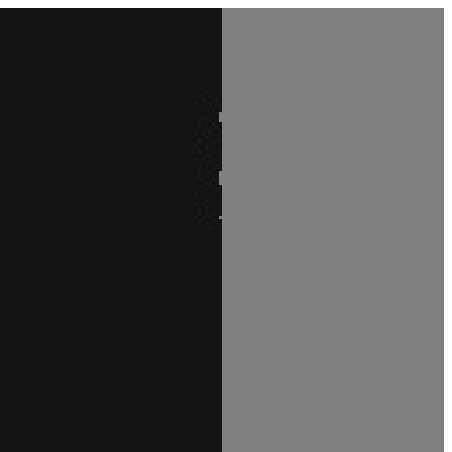}\label{fig:4e}}\
\subfloat[{1}]{\includegraphics[width=1.893cm]{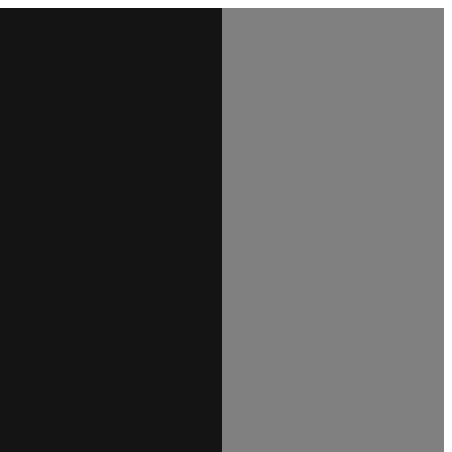}\label{fig:4f}}\\
\subfloat[{RVIN 40\%}]{\includegraphics[width=1.893cm]{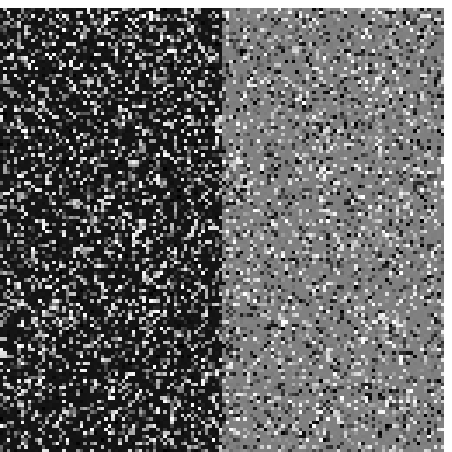}\label{fig:4g}}\
\subfloat[{0.9862}]{\includegraphics[width=1.893cm]{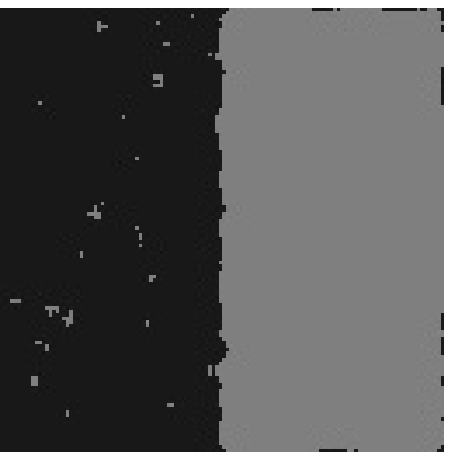}\label{fig:4h}}\
\subfloat[{0.9958}]{\includegraphics[width=1.893cm]{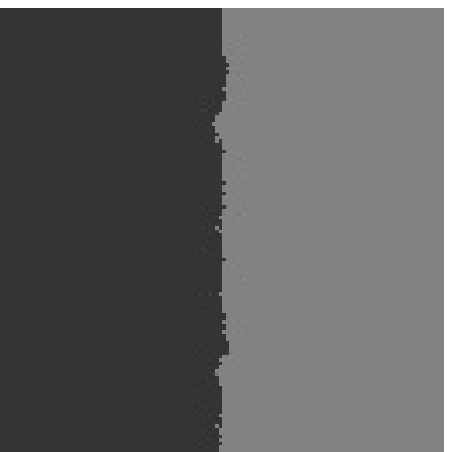}\label{fig:4i}}\
\subfloat[{0.9998}]{\includegraphics[width=1.893cm]{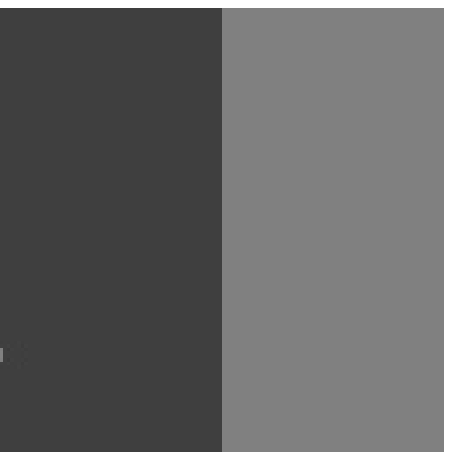}\label{fig:4j}}\
\subfloat[{0.9979}]{\includegraphics[width=1.893cm]{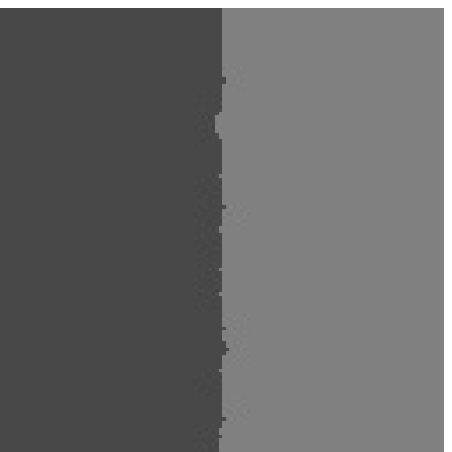}\label{fig:4k}}\
\subfloat[{1}]{\includegraphics[width=1.893cm]{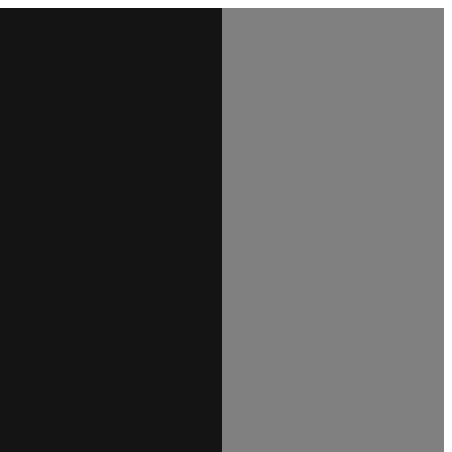}\label{fig:4l}}
\end{center}
\caption{Two-phase segmentation on the synthetic image Fig.~\ref{fig:1a} with RVIN. First column: images contaminated by 20\% and 40\%
RVIN, respectively; Second column to last column: results of FCM\_S2, FLICM, L2FS, L1SS, and L1FS, respectively. The SA values are reported below each segmentation result.}\label{fig:4}
\end{figure}

Tab.~\ref{tab:1} provides the SA comparison of these six algorithms, and Figs.~\ref{fig:2}-\ref{fig:4}
show some of the corresponding results.


For GN, all methods being tested give correct segmentation results for small standard deviations (e.g., $\sigma=10$). As the standard deviation increases, the SA value of FCM decays very fast. All the other algorithms have very large SA values even when the standard deviation is $80$. L1FS has the best performance for all cases, and it is able to give correct segmentation results even when $\sigma\le70$. L2FS is the second best algorithm, and it is able to give correct segmentation results when $\sigma\le 60$.


The results of all methods when $\sigma=30$ and $60$ are displayed in Fig.~\ref{fig:2}.
The results of FCM are relatively ``noisy'' (the second column). For FLICM (the third column) and L1SS (the fifth column), the segmentation
error occurs on the middle edge.

For SPIN, Tab.~\ref{tab:1} shows that FLICM  performs much better than FCM
for noise levels 10\%-30\%.
However, if the noise level is higher than 30\%, both FCM and FLICM have very poor performance.
L1SS achieves much better performance than FLICM, even when the noise
level is higher than 30\%. L2FS is slightly better than L1SS. Meanwhile,
if the noise level is higher than 50\%, L1SS fails to give a satisfactory result.
L1FS has the highest SA among all methods. It gives completely correct segmentation results
for noise levels 10\%-50\%. Fig.~\ref{fig:3} shows the results of all methods for noise levels 20\% and 40\%,
respectively.
For L2FS and L1SS, the segmentation errors occur along the middle edge. We also observe that
for high noise levels such as 40\%, both L2FS and L1SS suffer from slight contrast loss,
e.g., Fig.~\ref{fig:3j} and Fig.~\ref{fig:3k}.
However, L1FS is still able to preserve contrast, e.g., Fig.~\ref{fig:3f} and Fig.~\ref{fig:3l}.

For RVIN, FCM\_S2 is used to initialize $\mathbf{U}$ and $\mathbf{C}$ for TV based methods L2FS and L1FS.
Tab.~\ref{tab:1} shows that FCM\_S2 has the worst performance among all. L1SS is much better than FCM\_S2 especially for high noise levels.
FLICM performs slightly better than FCM\_S2.  L1FS achieves the best performance which is slightly better than L2FS.
L2FS gives correct segmentation results at noise levels 10\%-30\%, while our method L1FS can give the correct segmentation results
at noise levels 10\%-50\%.
Fig.~\ref{fig:4} shows the results
for noise levels 20\% and 40\%. Again, we find that most of the errors occur
along the middle edge for FLICM and L1SS. Moreover, the results of L2FS in Fig.~\ref{fig:4j},
L1SS in Fig.~\ref{fig:4e} and Fig.~\ref{fig:4k} lose some contrast.

Next we analyze the contrast problem for TV based methods. For L2FS, the estimated intensity $c_i$ in each segmented region roughly equals the mean value of the intensities in that region. In the Gaussian noise case, the noise has zero mean and therefore it has almost no impact on the estimation of $c_i$. However, for both impulse noise cases, the noise has a significant influence on the estimation of $c_i$ by taking the average. More specifically, assuming that the impulse noise follows the uniform distribution, its impact on the estimation of $c_i$ is like this.
Given an image with intensity range $[0,255]$, for the region $\Omega_i$ with true intensity $\beta\leq 128$, if there are more noisy pixels with intensity greater than $\beta$ than those with intensity less than $\beta$, then $c_i\geq\beta$ after taking the average and vice versa. Hence, the range of the image will be shrunk by applying L2FS even when the segmentation is correct, and thereby the recovered image will suffer from contrast loss. Note that the contrast loss problem has also been reported for the TVL2 restoration model in impulse noise removal~\cite{chan2005salt}. For L1SS algorithm, one step of ADMM is used to solve the $\vC$-subproblem approximately. Thus $c_i$ is not accurate for the first few iterations. However, since L1SS uses the L1-norm fidelity, the loss of contrast becomes more and more subtle as the number of iterations increases. In L1FS, we solve the $\mathbf{C}$-subproblem exactly. Thus, L1FS can preserve contrast well in the segmentation process.

\begin{table}[htbp]\scriptsize{\centering
\begin{threeparttable}
\caption{\label{tabGNaaa} SA comparison of different methods
applied on Fig.~\ref{fig:1b} with different levels of GN, SPIN, and RVIN.} \label{tab:2}
\begin{tabular}{lcccccccc} \toprule GN ($\sigma$) & 10 & 20&30 & 40 & 50&60&70&80\\
\midrule FCM& 0.9987 &0.8191 &0.6634 &0.5849 &0.5233&0.4718&0.4319&0.4017\\
L2FS&\textbf{1}&\textbf{0.9999}&\textbf{0.9994}&{0.9978}&0.9959&\textbf{0.9950}&\textbf{0.9931}&\textbf{0.9918}\\
L1FS &\textbf{1}&\textbf{0.9999}&0.9993&\textbf{0.9980}&\textbf{0.9964}&\textbf{0.9950}&\textbf{0.9931}&0.9905\\
\midrule SPIN (\%) & 10\% & 20\%&30\% & 40\% & 50\%&60\%&-&-\\
\midrule FCM& 0.9202&0.8431&0.7638&0.6847&0.6096&0.5296&-&-\\
L2FS&0.9926&0.9877&0.9713&0.9673&-&-&-&-\\
L1FS &\textbf{0.9977}&\textbf{0.9948}&\textbf{0.9923}&\textbf{0.9894}&\textbf{0.9848}&\textbf{0.9782}&-&-\\
\midrule RVIN (\%) & 10\% & 20\%&30\% & 40\% & -&-&-&-\\
\midrule FCM& 0.9922&0.9809&0.96672&0.9248&-&-&-&-\\
L2FS&0.9949&0.9923&0.9880&0.9731&-&-&-&-\\
L1FS&\textbf{0.9976}&\textbf{0.9957}&\textbf{0.9922}&\textbf{0.9868}-&-&-&-\\
\bottomrule
\end{tabular}
\end{threeparttable}}
\end{table}

\begin{table}[htbp]\scriptsize{\centering
 \begin{threeparttable}
 \caption{ SA comparison of different methods
applied on Fig.~\ref{fig:1c} with different levels of GN, SPIN, and RVIN.} \label{tab:3}
\begin{tabular}{lcccccccc}
\toprule GN ($\sigma$) & 10 & 20&30 & 40 & 50&60&70&80\\
\midrule
FCM& \textbf{1}&\textbf{1}&0.9998&0.9958&0.7927&0.7772&0.7488&0.7205\\
L2FS&\textbf{1}&\textbf{1}&\textbf{1}&\textbf{1}&0.9996&0.9992&\textbf{0.9989}&\textbf{0.9978}\\
L2L0&\textbf{1}&\textbf{1}&\textbf{1}&0.9999&0.9996&0.9991&0.9983&0.9967\\
L1FS&\textbf{1}&\textbf{1}&\textbf{1}&\textbf{1}&\textbf{0.9998}&\textbf{0.9994}&0.9985&0.9973\\
\midrule
 SPIN (\%) & 10\% & 20\%&30\% & 40\% & 50\%&60\%&-&-\\
\midrule
FCM&0.8498&0.7294&0.6128&0.5092&0.4248&0.3501&-&-\\
L2FS&0.9960&0.9925&0.9883&0.9822&0.9772&-&-&-\\
L2L0&0.9951&0.9880&0.9819&0.9740&0.9401&0.8752&-&-\\
L1FS&\textbf{0.9973}&\textbf{0.9937}&\textbf{0.9897}&\textbf{0.9854}&\textbf{0.9810}&\textbf{0.9732}&-&-\\
\midrule
RVIN (\%) & 10\% & 20\%&30\% & 40\% & 50\% & 60\% &-&-\\
\midrule
FCM& 0.8971&0.7992&0.6967&0.6085&0.5196&0.4294&-&-\\
L2FS&0.9974&0.9957&0.9899&0.9853&0.9841&0.8988&-&-\\
L2L0&0.9971&0.9955&0.9906&0.9856&0.9727&0.9488&-&-\\
L1FS&\textbf{0.9988}&\textbf{0.9963}&\textbf{0.9939}&\textbf{0.9906}&\textbf{0.9881}&\textbf{0.9777}&-&-\\
\bottomrule
\end{tabular}
\end{threeparttable}}
\end{table}

\begin{figure*}[!htbp]
\begin{center}
\subfloat[{GN $\sigma$=30}]{\includegraphics[width=1.893cm]{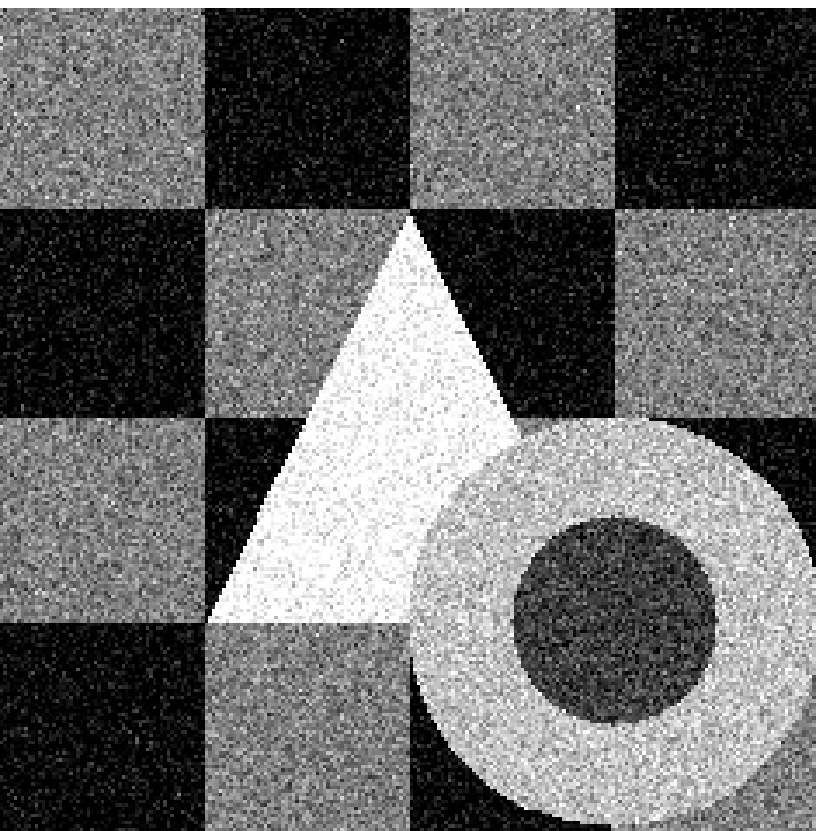}\label{fig:5a}}\
\subfloat[{GN $\sigma$=60}]{\includegraphics[width=1.893cm]{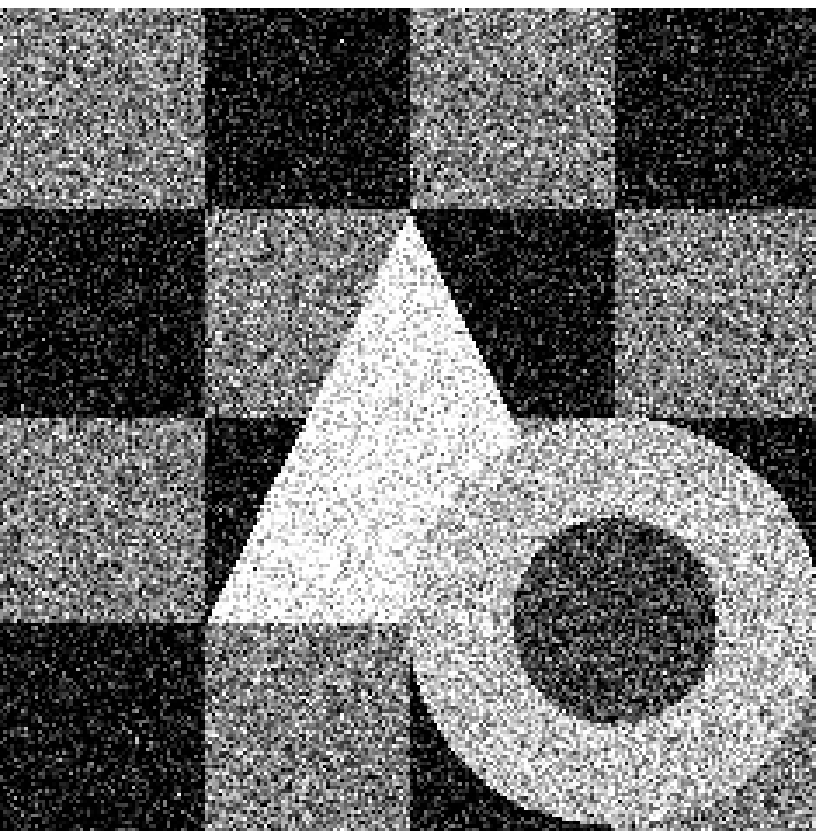}\label{fig:5b}}\
\subfloat[{SPIN 20\%}]{\includegraphics[width=1.893cm]{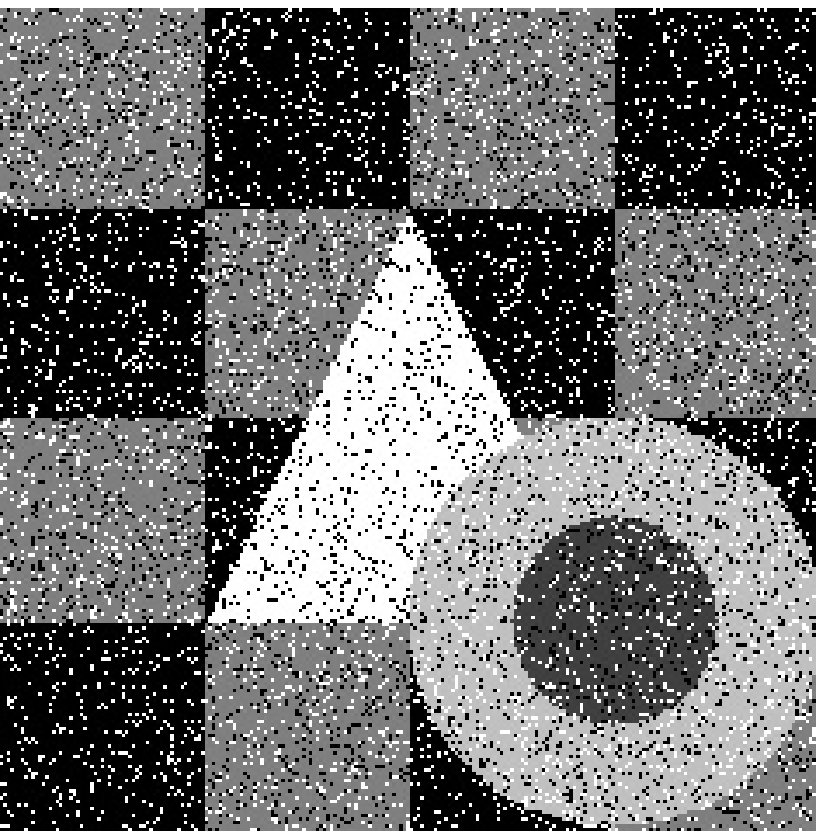}\label{fig:5c}}\
\subfloat[{SPIN 40\%}]{\includegraphics[width=1.893cm]{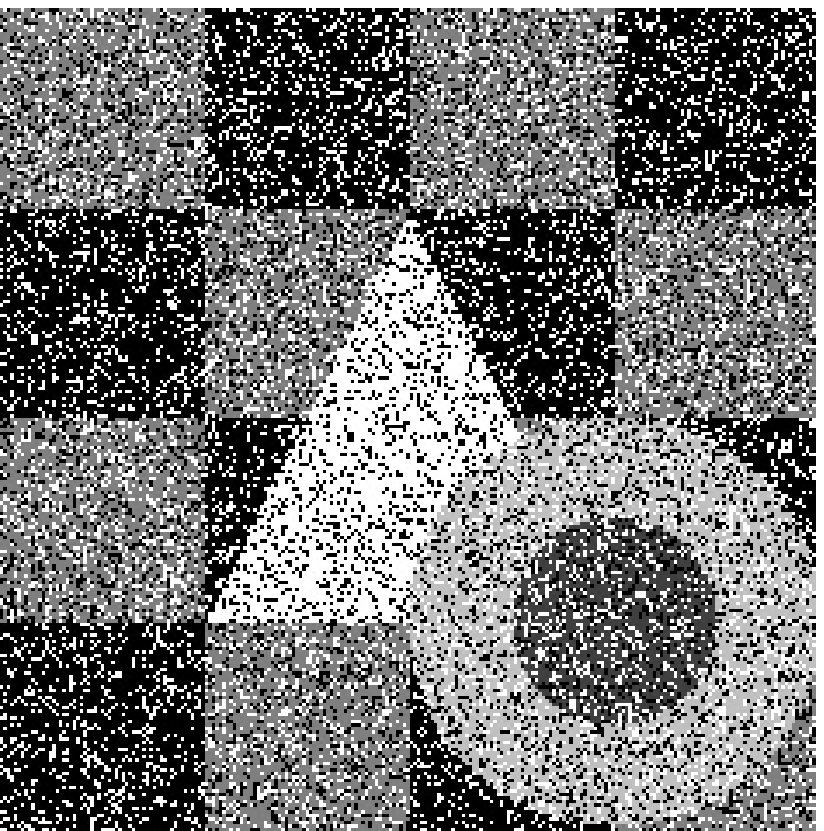}\label{fig:5d}}\
\subfloat[{RVIN 20\%}]{\includegraphics[width=1.893cm]{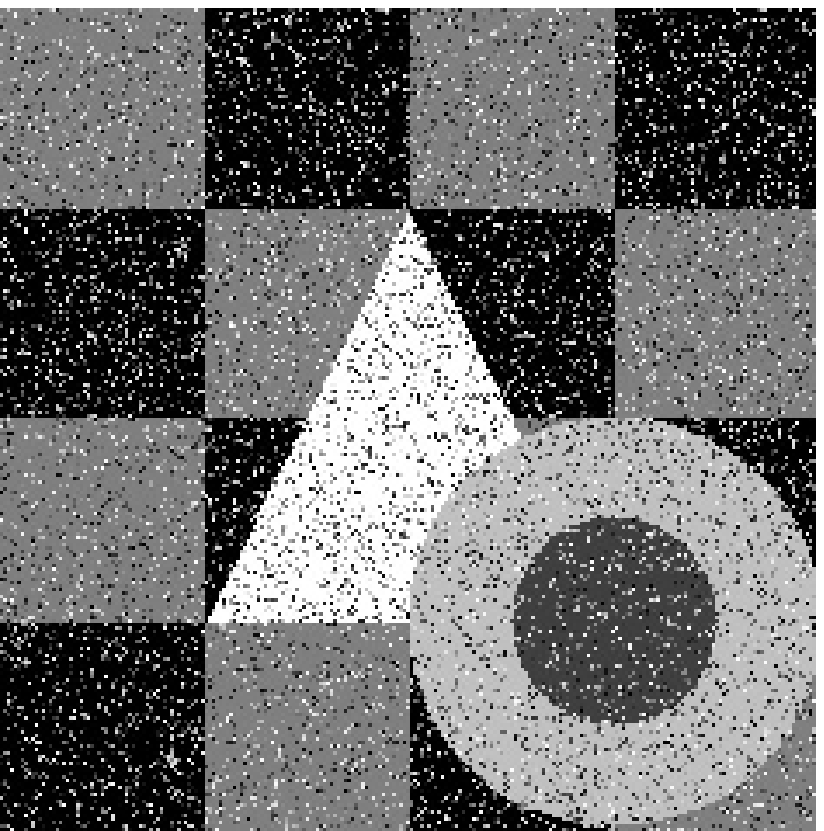}\label{fig:5e}}\
\subfloat[{RVIN 40\%}]{\includegraphics[width=1.893cm]{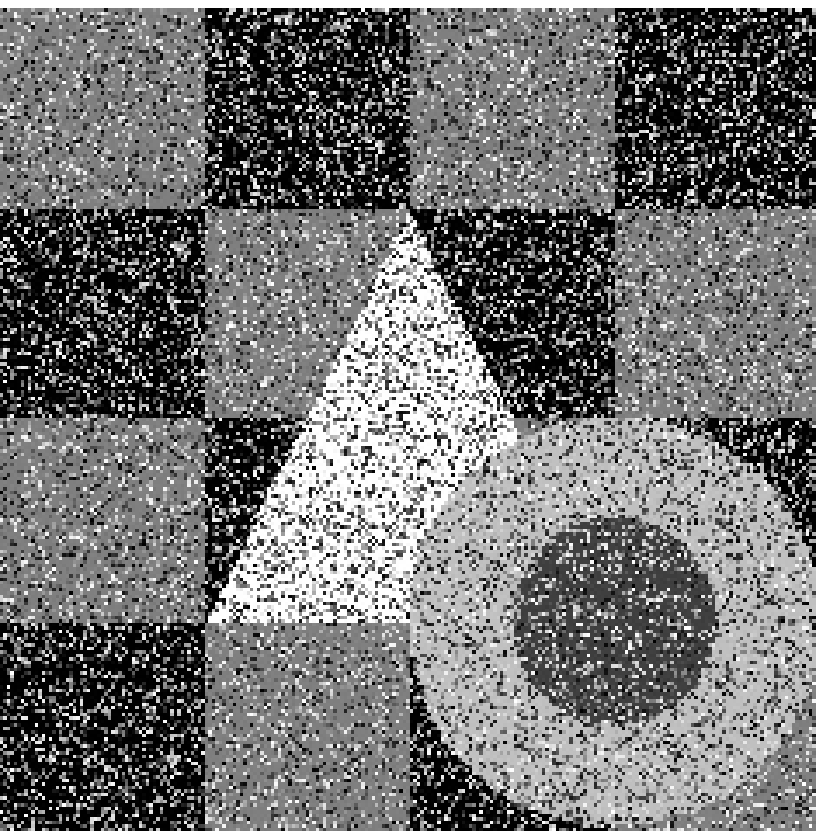}\label{fig:5f}}\\
\subfloat[{{0.6634}}]{\includegraphics[width=1.893cm]{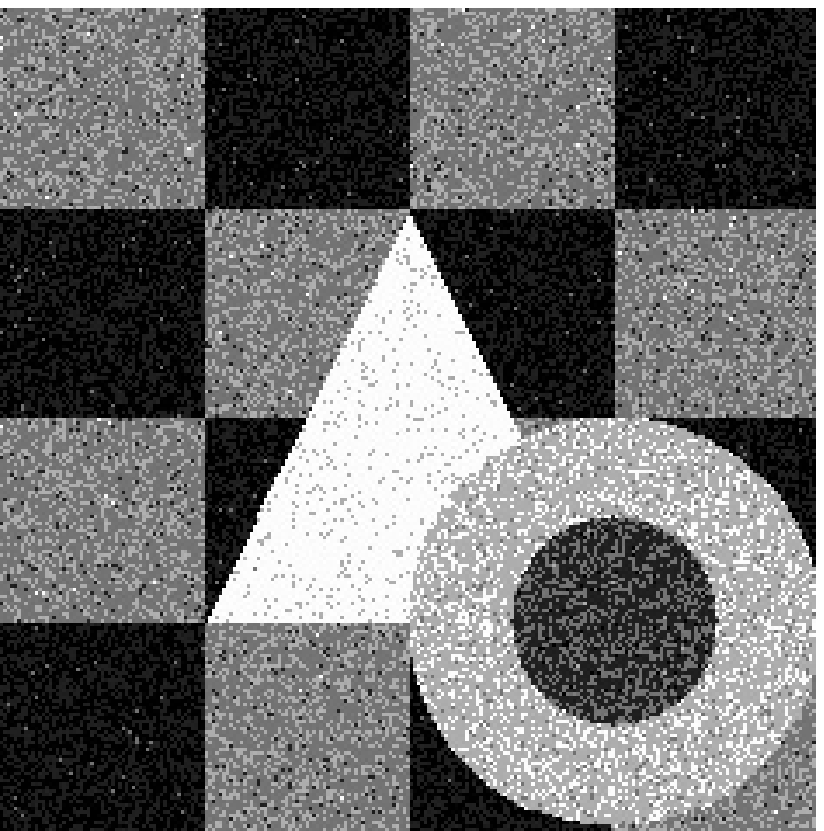}\label{fig:5g}}\
\subfloat[{0.4718}]{\includegraphics[width=1.893cm]{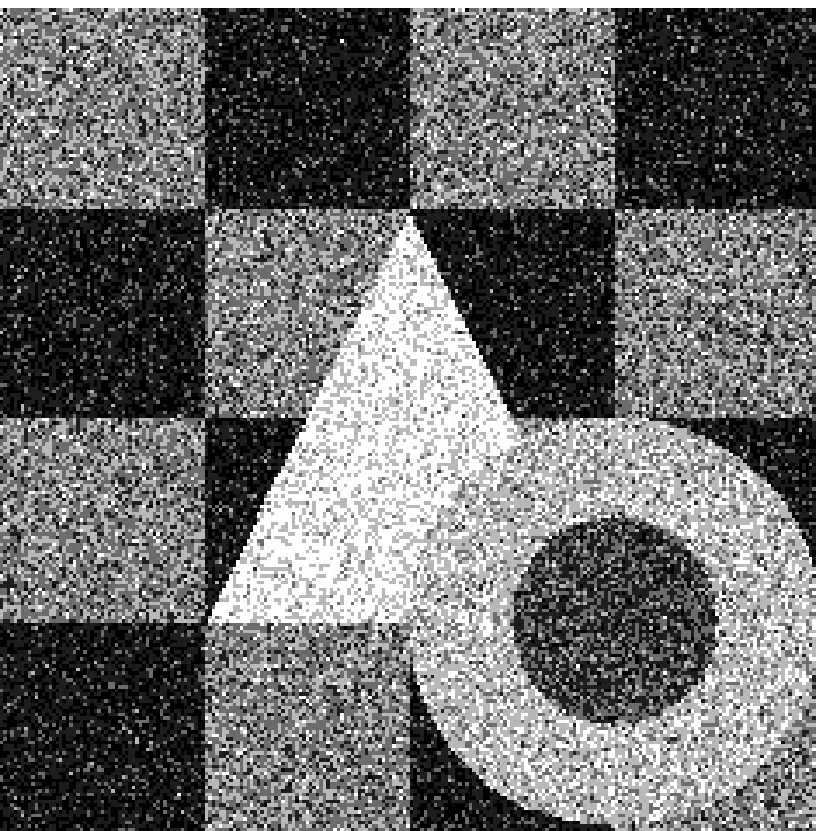}\label{fig:5h}}\
\subfloat[{0.8431}]{\includegraphics[width=1.893cm]{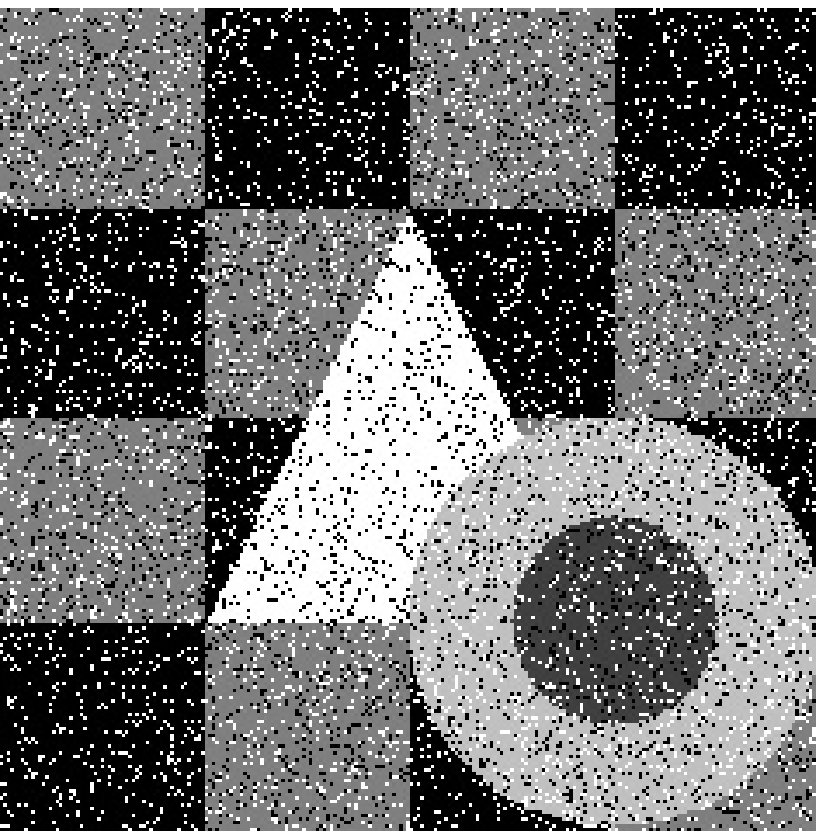}\label{fig:5i}}\
\subfloat[{0.6847}]{\includegraphics[width=1.893cm]{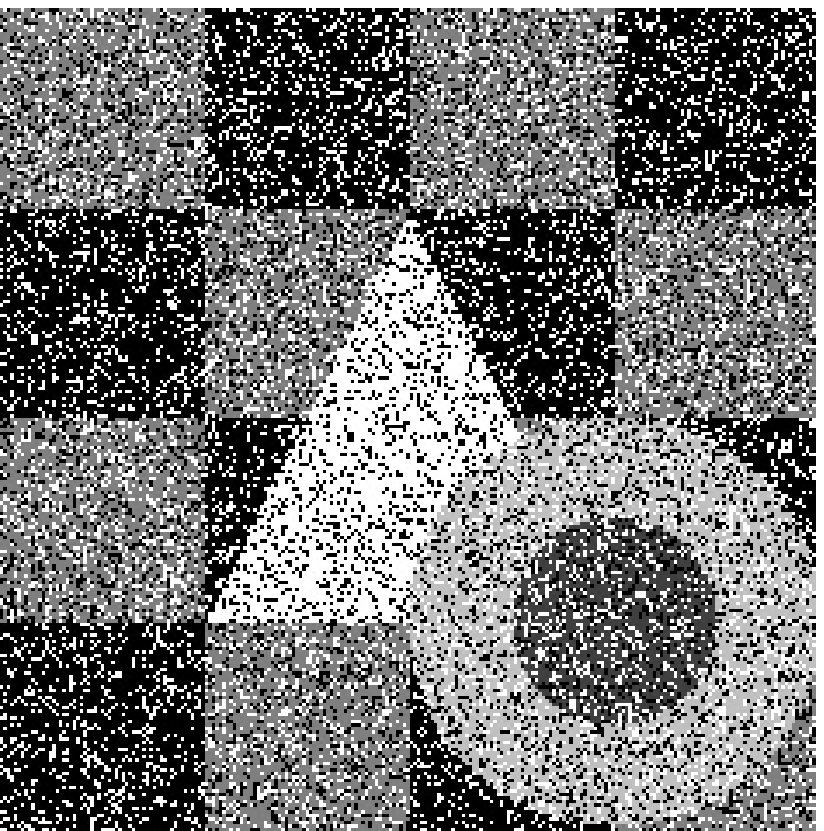}\label{fig:5j}}\
\subfloat[{0.9809}]{\includegraphics[width=1.893cm]{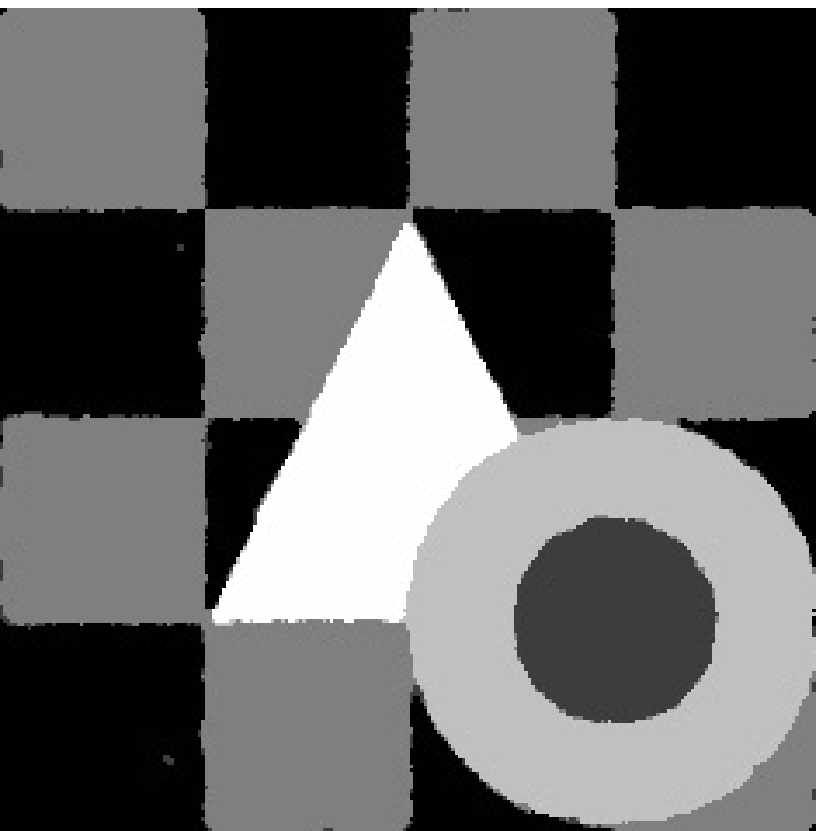}\label{fig:5k}}\
\subfloat[{0.9248}]{\includegraphics[width=1.893cm]{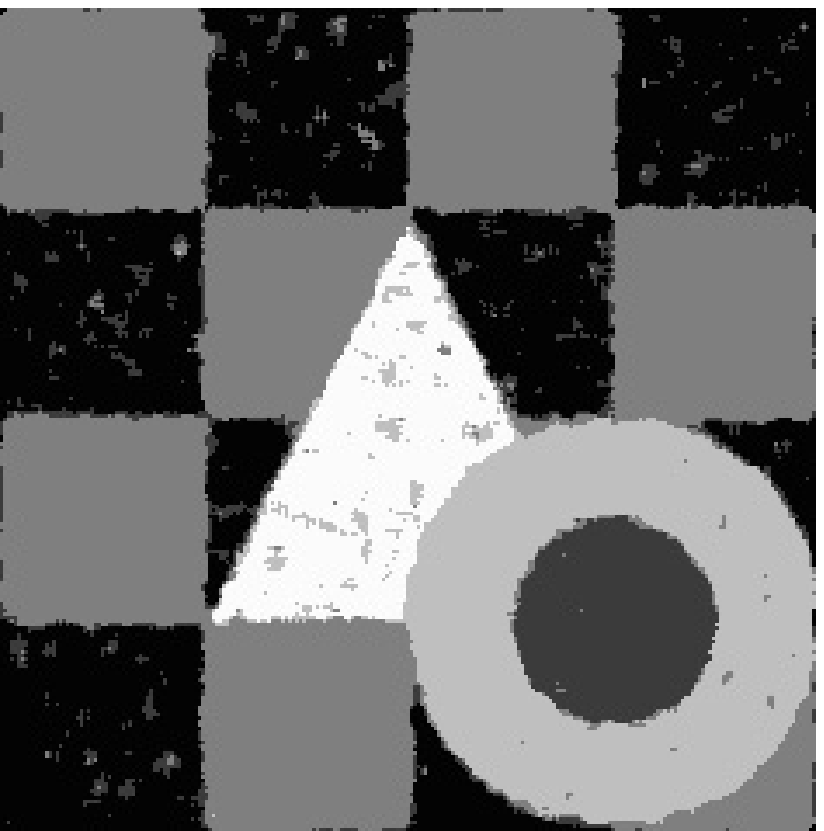}\label{fig:5l}}\\
\subfloat[{0.9994}]{\includegraphics[width=1.893cm]{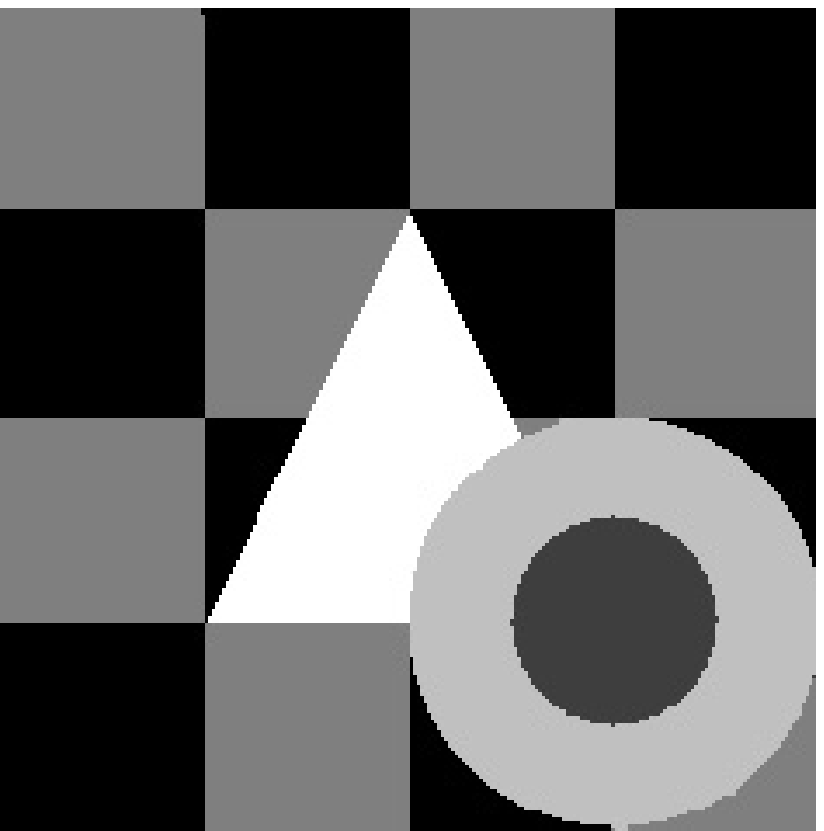}\label{fig:5m}}\
\subfloat[{0.9950}]{\includegraphics[width=1.893cm]{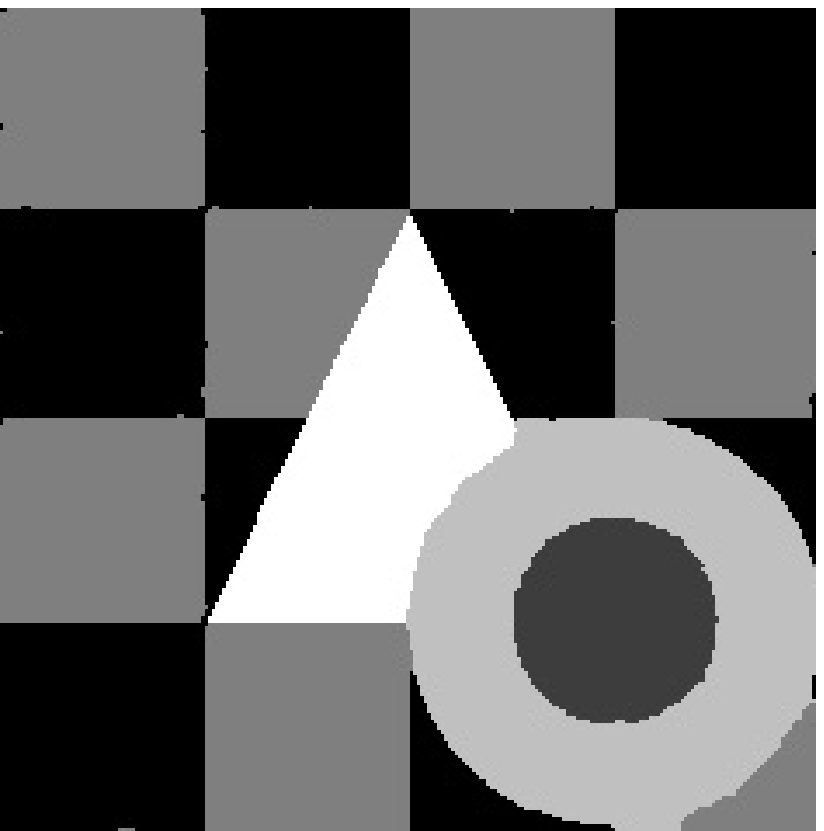}\label{fig:5n}}\
\subfloat[{0.9877}]{\includegraphics[width=1.893cm]{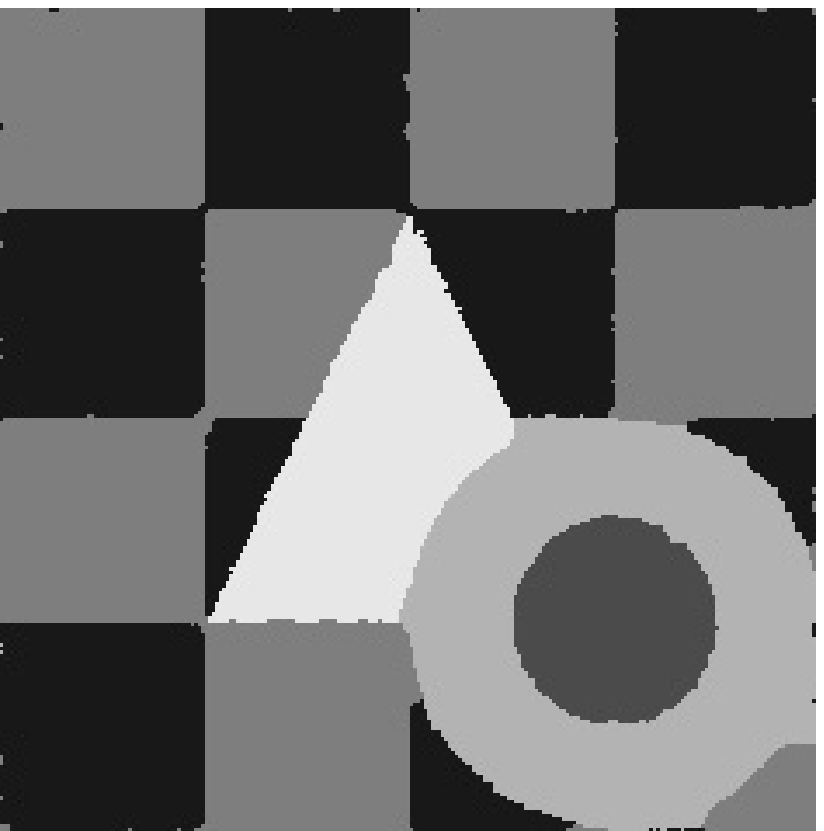}\label{fig:5o}}\
\subfloat[{0.9673}]{\includegraphics[width=1.893cm]{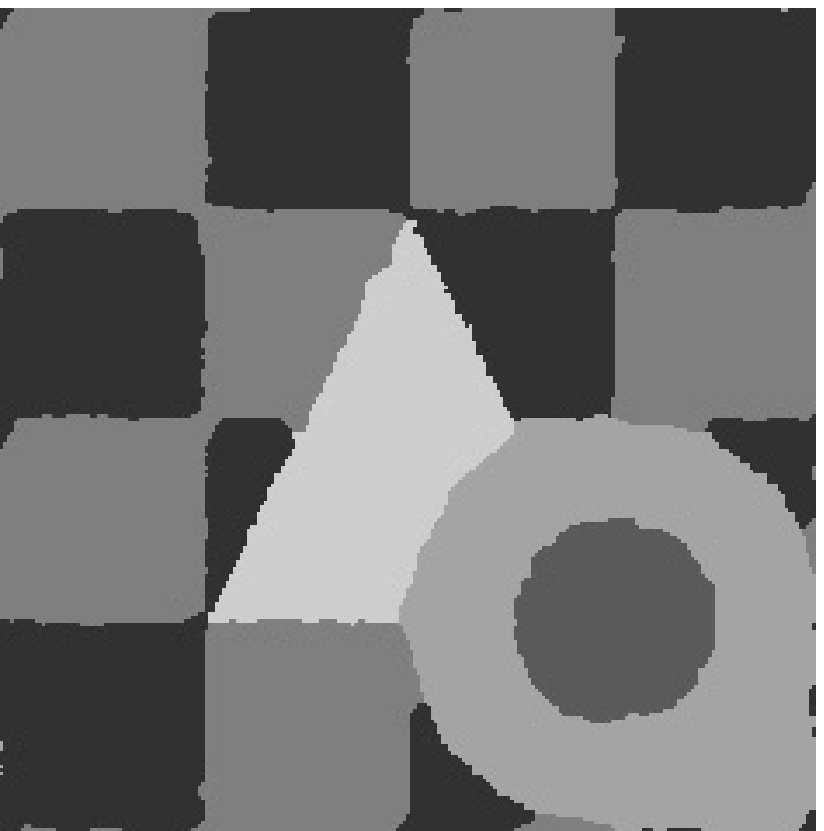}\label{fig:5p}}\
\subfloat[{0.9923}]{\includegraphics[width=1.893cm]{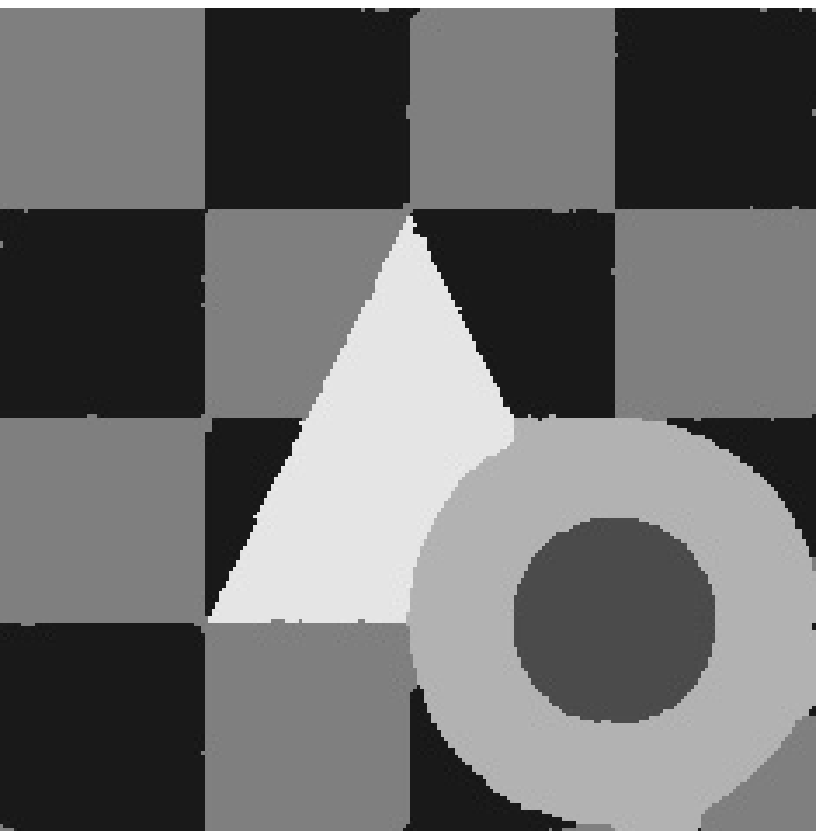}\label{fig:5q}}\
\subfloat[{0.9731}]{\includegraphics[width=1.893cm]{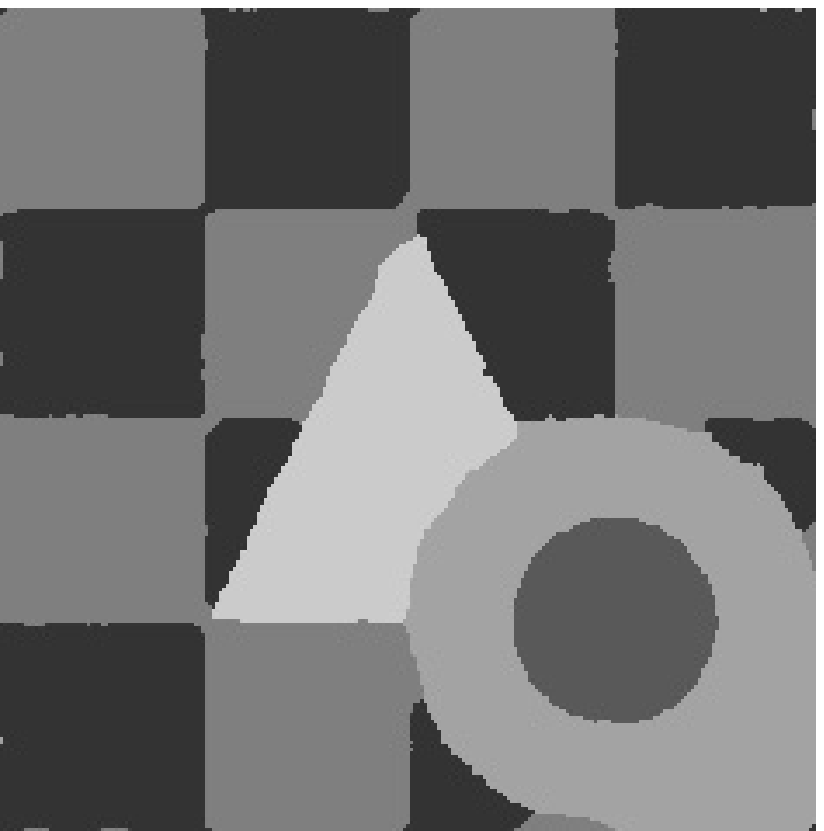}\label{fig:5r}}\\
\subfloat[{0.9993}]{\includegraphics[width=1.893cm]{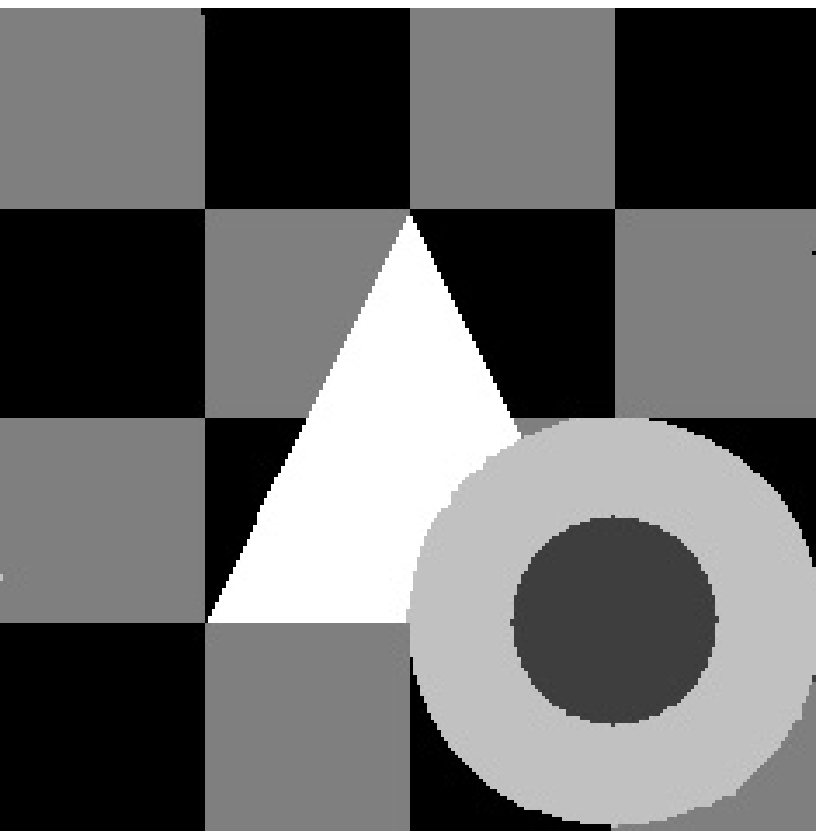}\label{fig:5s}}\
\subfloat[{0.9950}]{\includegraphics[width=1.893cm]{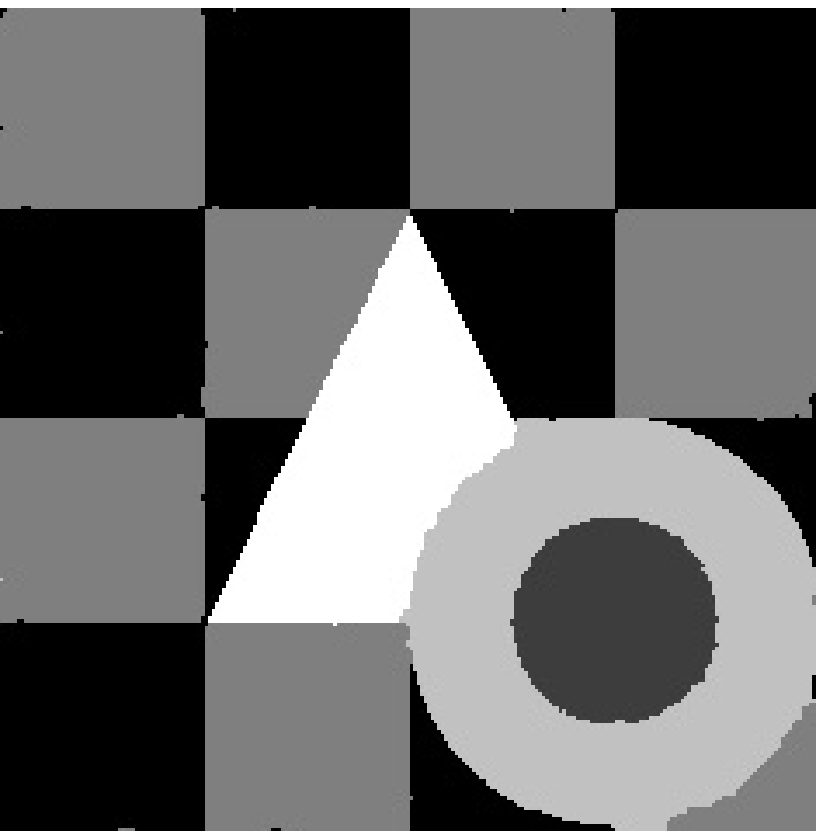}\label{fig:5t}}\
\subfloat[{0.9948}]{\includegraphics[width=1.893cm]{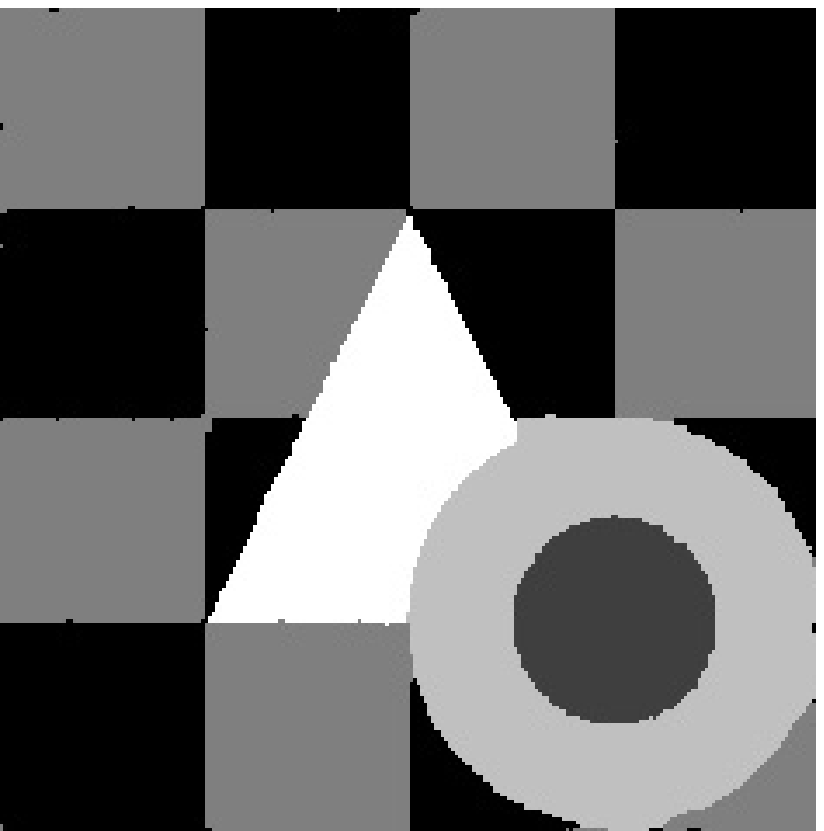}\label{fig:5u}}\
\subfloat[{0.9894}]{\includegraphics[width=1.893cm]{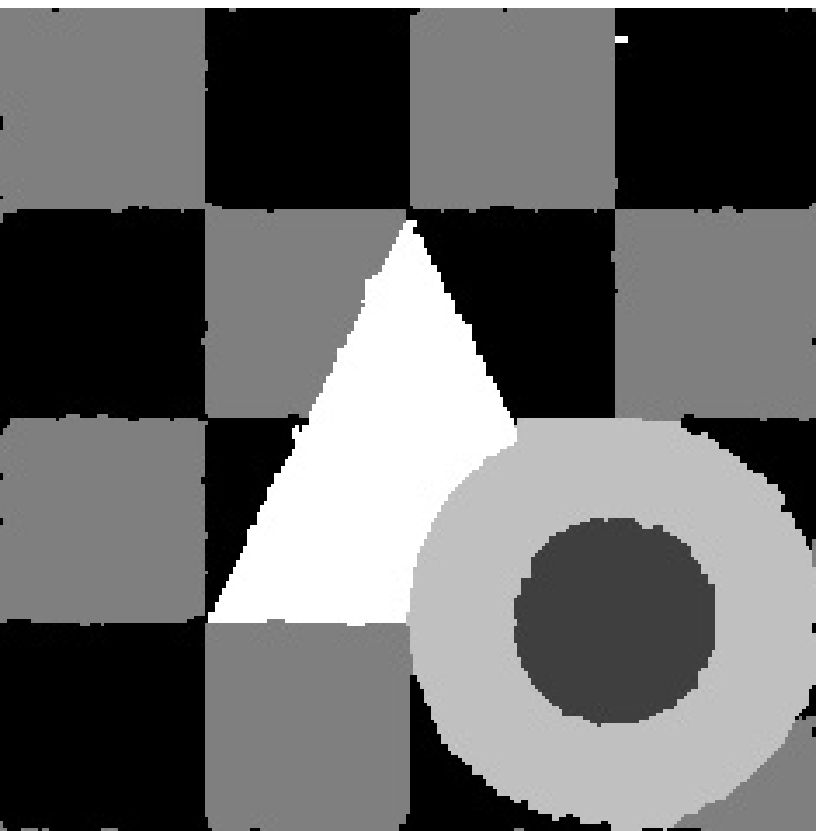}\label{fig:5v}}\
\subfloat[{0.9957}]{\includegraphics[width=1.893cm]{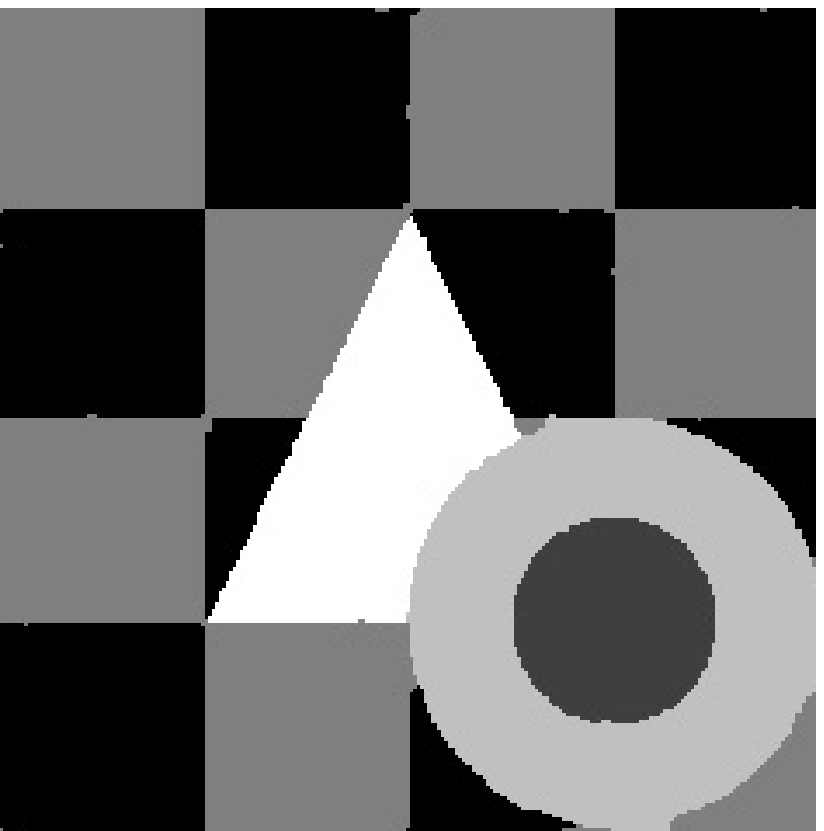}\label{fig:5w}}\
\subfloat[{0.9868}]{\includegraphics[width=1.893cm]{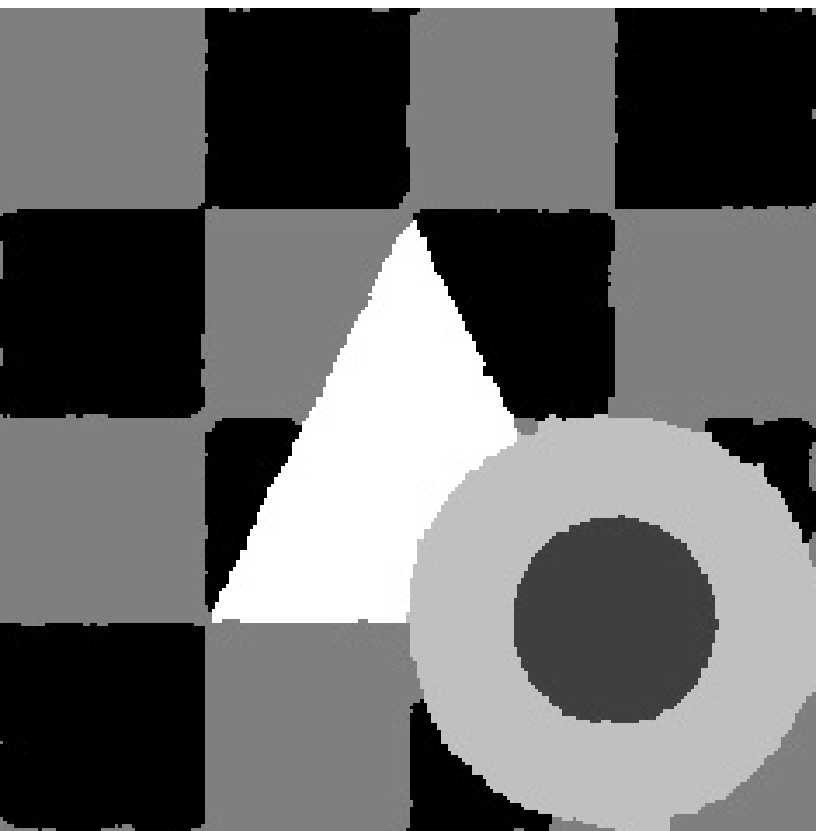}\label{fig:5x}}
\end{center}
\caption{Multiphase segmentation on the synthetic image Fig.~\ref{fig:1b} with different levels of GN, SPIN, and RVIN.
First row: images contaminated different types of noise with different levels. Second row to last row: results of
FCM, L2FS and L1FS, respectively. The SA values are reported below each segmentation result.}\label{fig:5}
\end{figure*}

\subsection{Test on Fig.~\ref{fig:1b}}
The performance comparison for the multiphase synthetic piecewise constant gray image Fig.~\ref{fig:1b} is shown in Tab.~\ref{tab:2} and Fig.~\ref{fig:5}.

As shown in Tab.~\ref{tab:2}, FCM performs poorly for GN, while L2FS and L1FS perform relatively well with similar SA. For the noise level $\sigma=10$, both L2FS and L1FS give a correct segmentation result. For SPIN, FCM also gives the worst performance in terms of SA, while L1FS achieves the best performance. L2FS fails to yield a correct segmentation result when noise levels $\sigma\ge50$. For RVIN, FCM\_S2 achieves high SA values since it can smooth out some noise in the segmentation process. L2FS performs much better than FCM\_S2, and L1FS performs the best among all methods.

Fig.~\ref{fig:5} shows some results corresponding to Tab.\ref{tab:2}. The first row is the noisy images being tested. The second row shows the results of FCM (Fig.~\ref{fig:5g}-\ref{fig:5j}) and FCM\_S2 (Fig.~\ref{fig:5k}-\ref{fig:5l}). Most of them looks very ``noisy'' except Fig.~\ref{fig:5k}. L2FS and L1FS give very clean results in the third row and last row, respectively. Comparing the results by L2FS and L1FS for Gaussian noise, both of them have high visual qualities. However, for SPIN and RVIN, it is obvious that L1FS preserves the contrast much better than L2FS and has better segmentation results.

\subsection{Test on Fig.~\ref{fig:1c}}

\begin{figure}[!htbp]
\captionsetup[subfigure]{labelformat=empty}
\begin{center}
\subfloat[{GN $\sigma$=30}]{\includegraphics[width=1.893cm]{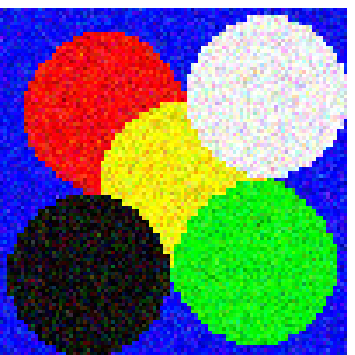}\label{fig:6a}}\
\subfloat[{GN $\sigma$=60}]{\includegraphics[width=1.893cm]{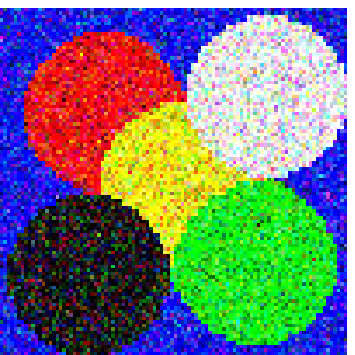}\label{fig:6b}}\
\subfloat[{SPIN 20\%}]{\includegraphics[width=1.893cm]{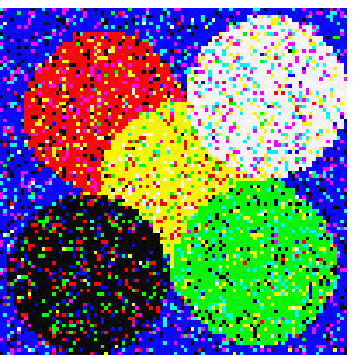}\label{fig:6c}}\
\subfloat[{SPIN 40\%}]{\includegraphics[width=1.893cm]{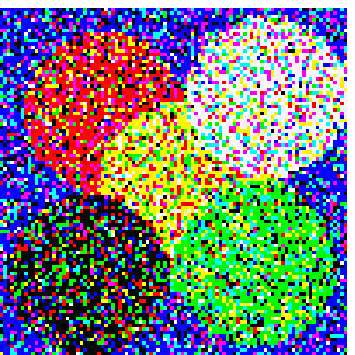}\label{fig:6d}}\
\subfloat[{RVIN 20\%}]{\includegraphics[width=1.893cm]{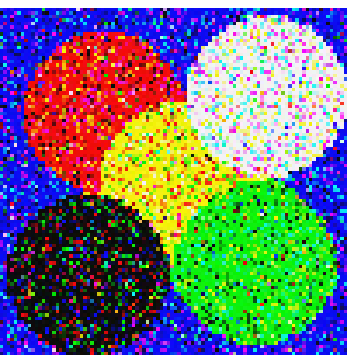}\label{fig:6e}}\
\subfloat[{RVIN 40\%}]{\includegraphics[width=1.893cm]{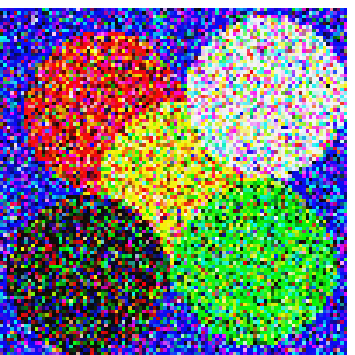}\label{fig:6f}}\\
\subfloat[{0.9998}]{\includegraphics[width=1.893cm]{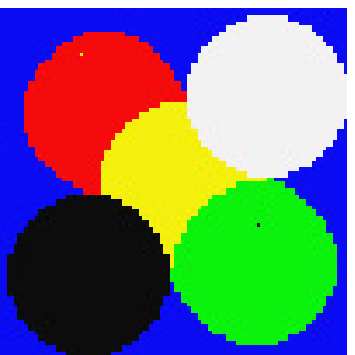}\label{fig:6g}}\
\subfloat[{0.7772}]{\includegraphics[width=1.893cm]{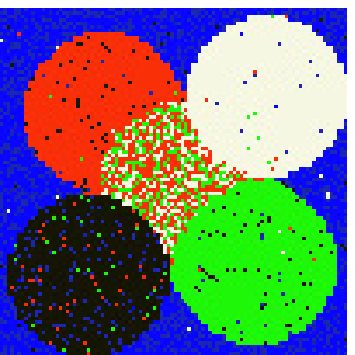}\label{fig:6h}}\
\subfloat[{0.7294}]{\includegraphics[width=1.893cm]{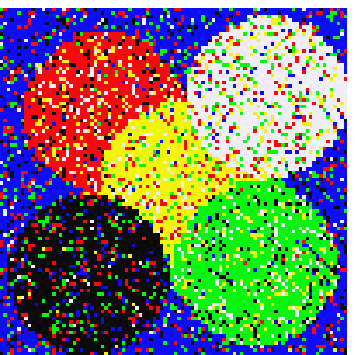}\label{fig:6i}}\
\subfloat[{0.5092}]{\includegraphics[width=1.893cm]{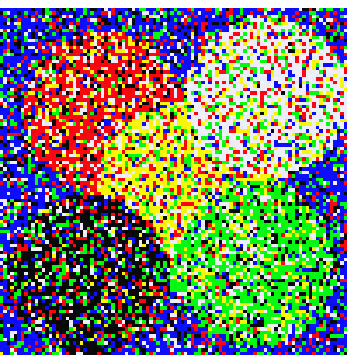}\label{fig:6j}}\
\subfloat[{0.7992}]{\includegraphics[width=1.893cm]{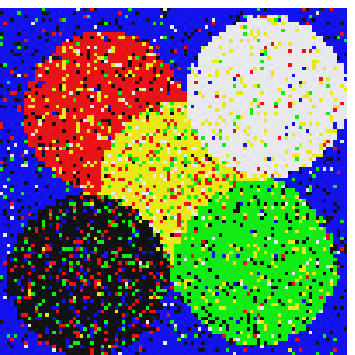}\label{fig:6k}}\
\subfloat[{0.6085}]{\includegraphics[width=1.893cm]{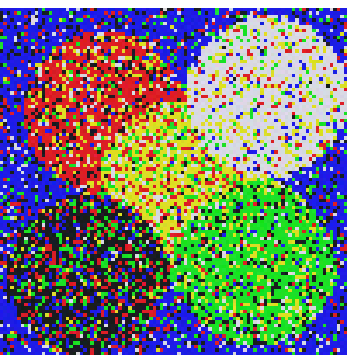}\label{fig:6l}}\\
\subfloat[{1}]{\includegraphics[width=1.893cm]{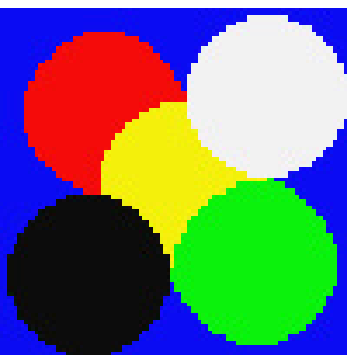}\label{fig:6m}}\
\subfloat[{0.9992}]{\includegraphics[width=1.893cm]{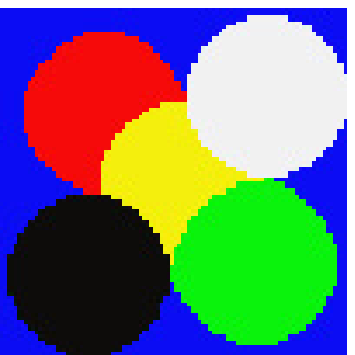}\label{fig:6n}}\
\subfloat[{0.9925}]{\includegraphics[width=1.893cm]{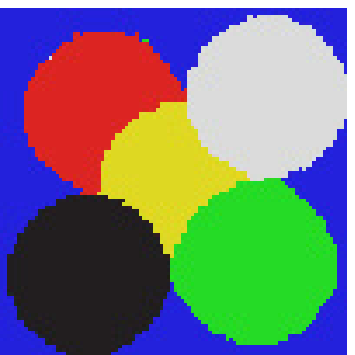}\label{fig:6o}}\
\subfloat[{0.9822}]{\includegraphics[width=1.893cm]{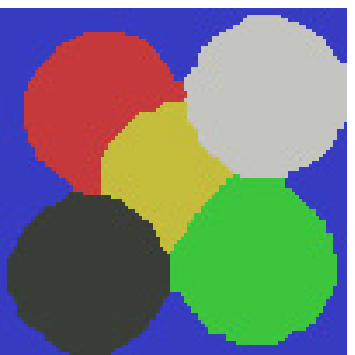}\label{fig:6p}}\
\subfloat[{0.9957}]{\includegraphics[width=1.893cm]{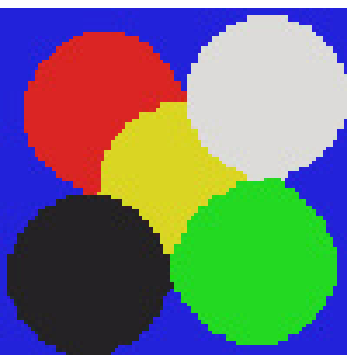}\label{fig:6q}}\
\subfloat[{0.9853}]{\includegraphics[width=1.893cm]{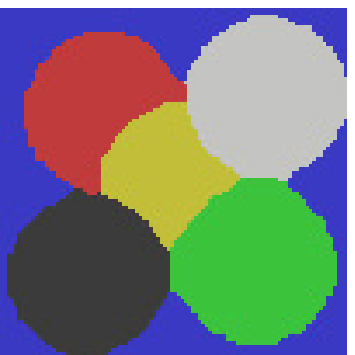}\label{fig:6r}}\\
\subfloat[{1}]{\includegraphics[width=1.893cm]{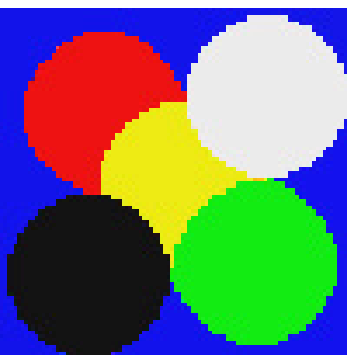}\label{fig:6y}}\
\subfloat[{0.9991}]{\includegraphics[width=1.893cm]{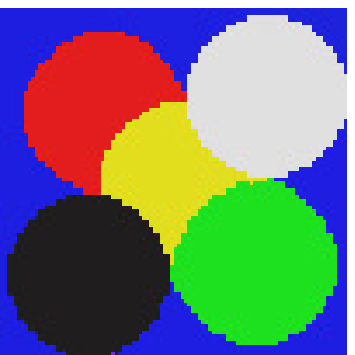}\label{fig:6z}}\
\subfloat[{0.9880}]{\includegraphics[width=1.893cm]{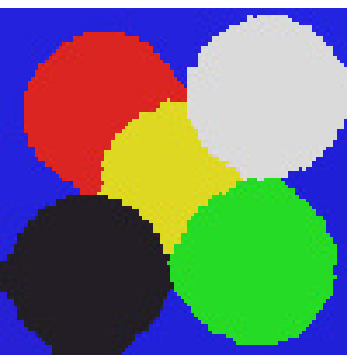}}\
\subfloat[{0.9740}]{\includegraphics[width=1.893cm]{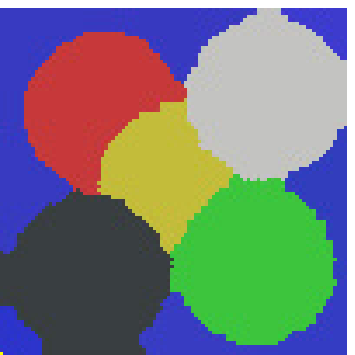}}\
\subfloat[{0.9955}]{\includegraphics[width=1.893cm]{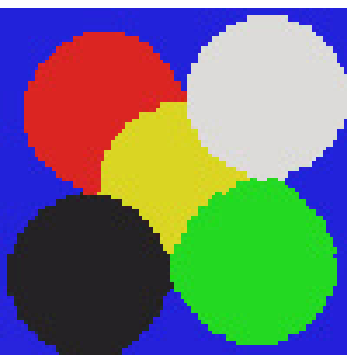}}\
\subfloat[{0.9856}]{\includegraphics[width=1.893cm]{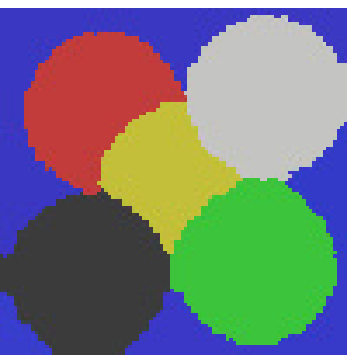}}\\
\subfloat[{1}]{\includegraphics[width=1.893cm]{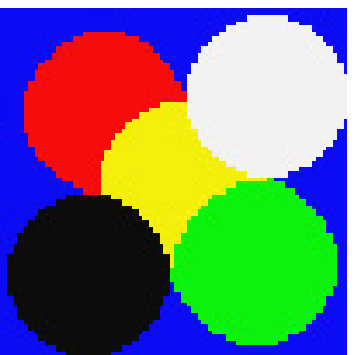}\label{fig:6s}}\
\subfloat[{0.9994}]{\includegraphics[width=1.893cm]{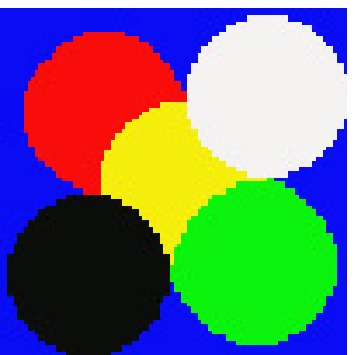}\label{fig:6t}}\
\subfloat[{0.9937}]{\includegraphics[width=1.893cm]{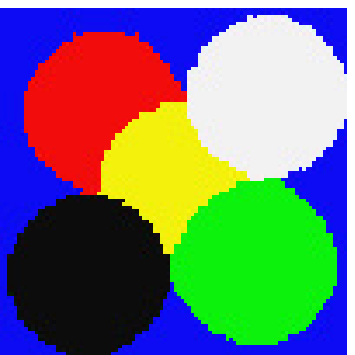}\label{fig:6u}}\
\subfloat[{0.9854}]{\includegraphics[width=1.893cm]{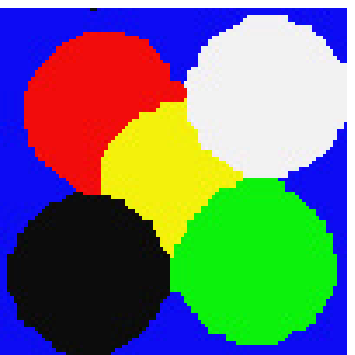}\label{fig:6v}}\
\subfloat[{0.9963}]{\includegraphics[width=1.893cm]{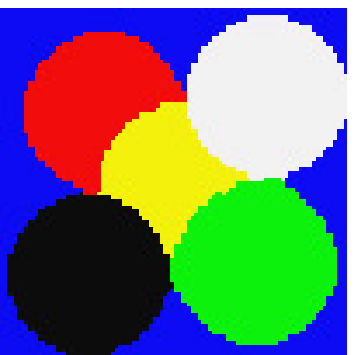}\label{fig:6w}}\
\subfloat[{0.9906}]{\includegraphics[width=1.893cm]{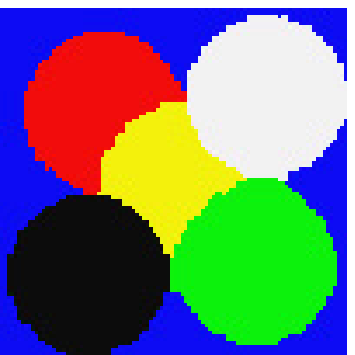}\label{fig:6x}}
\end{center}
\caption{Multiphase segmentation on the synthetic color image Fig.~\ref{fig:1c} with different levels of
GN, SPIN, and RVIN. First row: images contaminated different types of noise with different levels.
Second row to last row: results of FCM, L2FS, L2L0, and L1FS, respectively. The SA values are reported below each segmentation result.}\label{fig:6}
\end{figure}

In Tab.~\ref{tab:3} and Fig.~\ref{fig:6}, we test the multiphase synthetic piecewise constant color image Fig.~\ref{fig:1c} with various levels of GN, SPIN, and RVIN.

From Tab.~\ref{tab:3}, in the GN case, when the standard deviation of noise $\sigma\le20 $, all the four methods, including FCM, L2FS, L2L0 and L1FS, give correct segmentation results. Moreover, both L2FS and L1FS yield correct segmentation results when $\sigma\le40$. When $\sigma\ge50$, the performance of FCM decreases rapidly, while L2FS, L2L0, and L1FS still achieve very large SA values. Note that we initialize $\mathbf{U}$ randomly for GN in this test. For SPIN and RVIN, FCM has the worst performance which is far lower than that of the other three methods. It is also obvious that L1FS outperforms L2L0 and L2FS.

Fig.~\ref{fig:6} shows some results corresponding to Tab.~\ref{tab:3}. The first row shows the tested noisy images. The results of FCM in the second row seems to be ``noisy'' in most cases. The results of L2FS (the third row), L2L0 (the fourth row), and L1FS (the last row) are very clean. However, in terms of contrast,
it is obvious that L1FS outperforms L2L0 and L2FS particularly for SPIN and RVIN. 

\subsection{Test on real images}
We test some real images including two
medical images and six natural images
in Fig.~\ref{fig:7} without artificial noise.
However, the images can be regarded as containing some types
of noise due to the acquisition and transmission processes.

The results of FCM and L1FS are displayed for comparison.
One can see that FCM tends to produce some tiny components
and irregular segmentation boundaries. By contrast,
L1FS tends to smooth out tiny components to generate large ones and smooth boundaries between regions,
which is more natural for human vision system
and good for postprocessing. This smoothing effect is mainly
achieved by total variation regularization in the L1FS model.
Moreover, L1FS preserves slightly better
contrast in the piecewise constant approximation than FCM, which is mainly
achieved by the use of the L1-norm fidelity.

In Figs.~\ref{fig:7d}-\ref{fig:7f},  FCM and L1FS give quite different segmentation results. Obviously, FCM fails to segment the blue color part of the clothes
in the original image, while the proposed L1FS works well. To illustrate the difference of these two methods, we display, in Fig.~\ref{fig:8}, the corresponding six segmented regions of FCM and L1FS for the women image Fig.~\ref{fig:7d}, respectively.
We find that the segmented regions of FCM in Fig.~\ref{fig:8a}-\ref{fig:8e} are somehow ``noisy''.
In particular, the background lattice pattern is partitioned into five regions as shown
in Fig.~\ref{fig:8a}-\ref{fig:8e}.
Compared with FCM, the proposed L1FS gives quite clean segmented
regions in the second row in Fig.~\ref{fig:8}. Especially, the background lattice pattern are
classified into only two classes as shown in Fig.~\ref{fig:8g}-\ref{fig:8h}.
We further compare the blue parts of the clothes corresponding to Fig.~\ref{fig:8c} and Fig.~\ref{fig:8i}.
In Fig.~\ref{fig:8c}, some background pattern heavily affects the estimation of $\mathbf{C}$, and therefore the color is not blue any more. However, in Fig.~\ref{fig:8i}, since the background is clean, the correct color can be obtained. To sum up, Fig.~\ref{fig:8} demonstrates that the proposed L1FS gives smoother segmentations than FCM.

\begin{figure}[!htbp]\label{real}
\begin{center}
\subfloat[N=2]{\includegraphics[width=1.893cm,height=1.893cm]{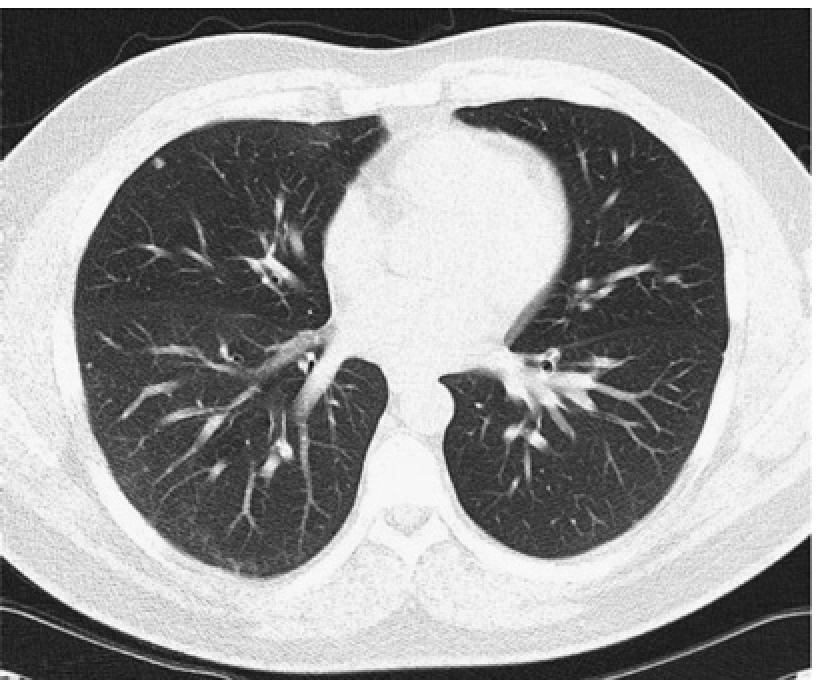}\label{fig:7s}}\
\subfloat[ ]{\includegraphics[width=1.893cm,height=1.893cm]{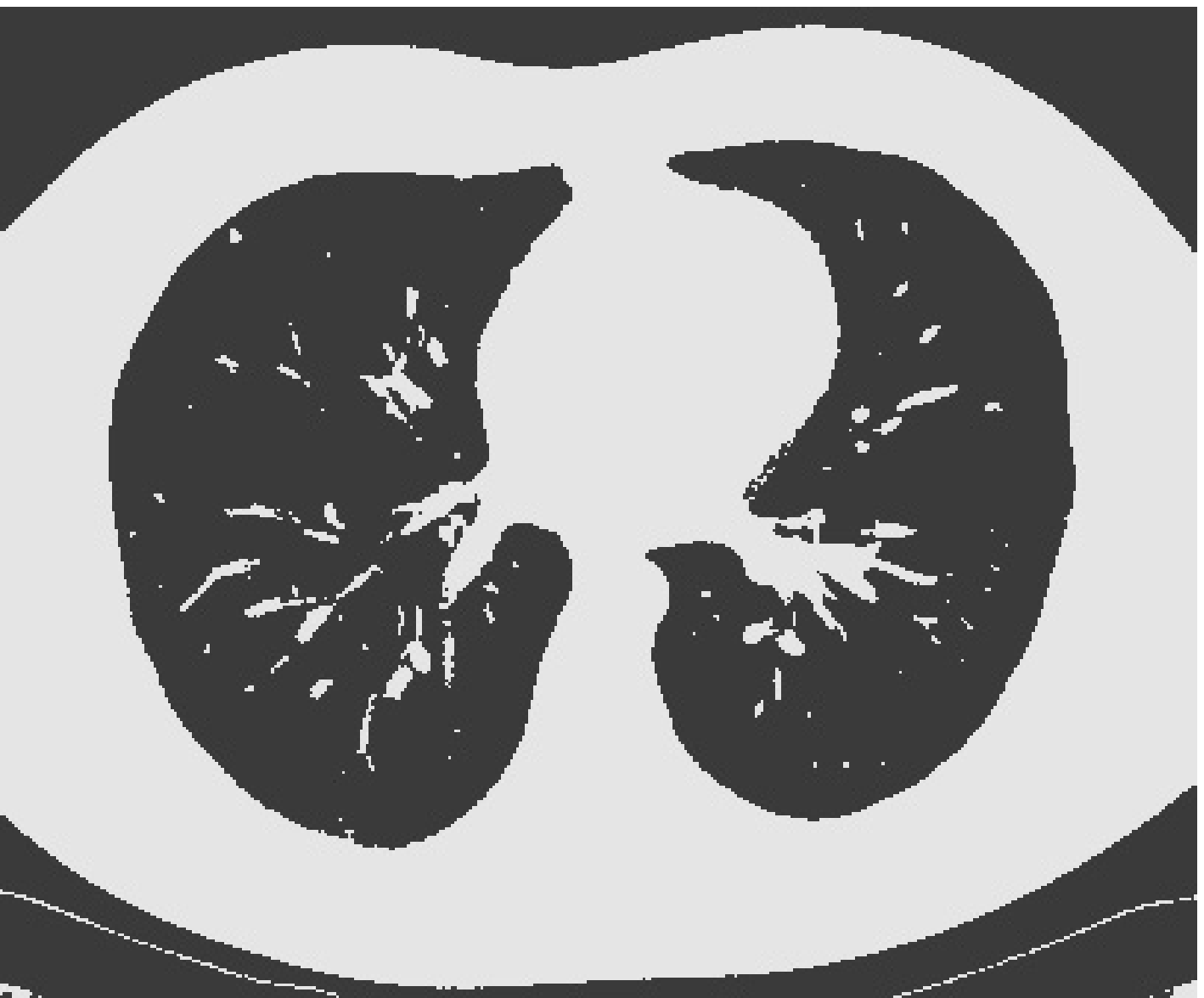}\label{fig:7t}}\
\subfloat[ ]{\includegraphics[width=1.893cm,height=1.893cm]{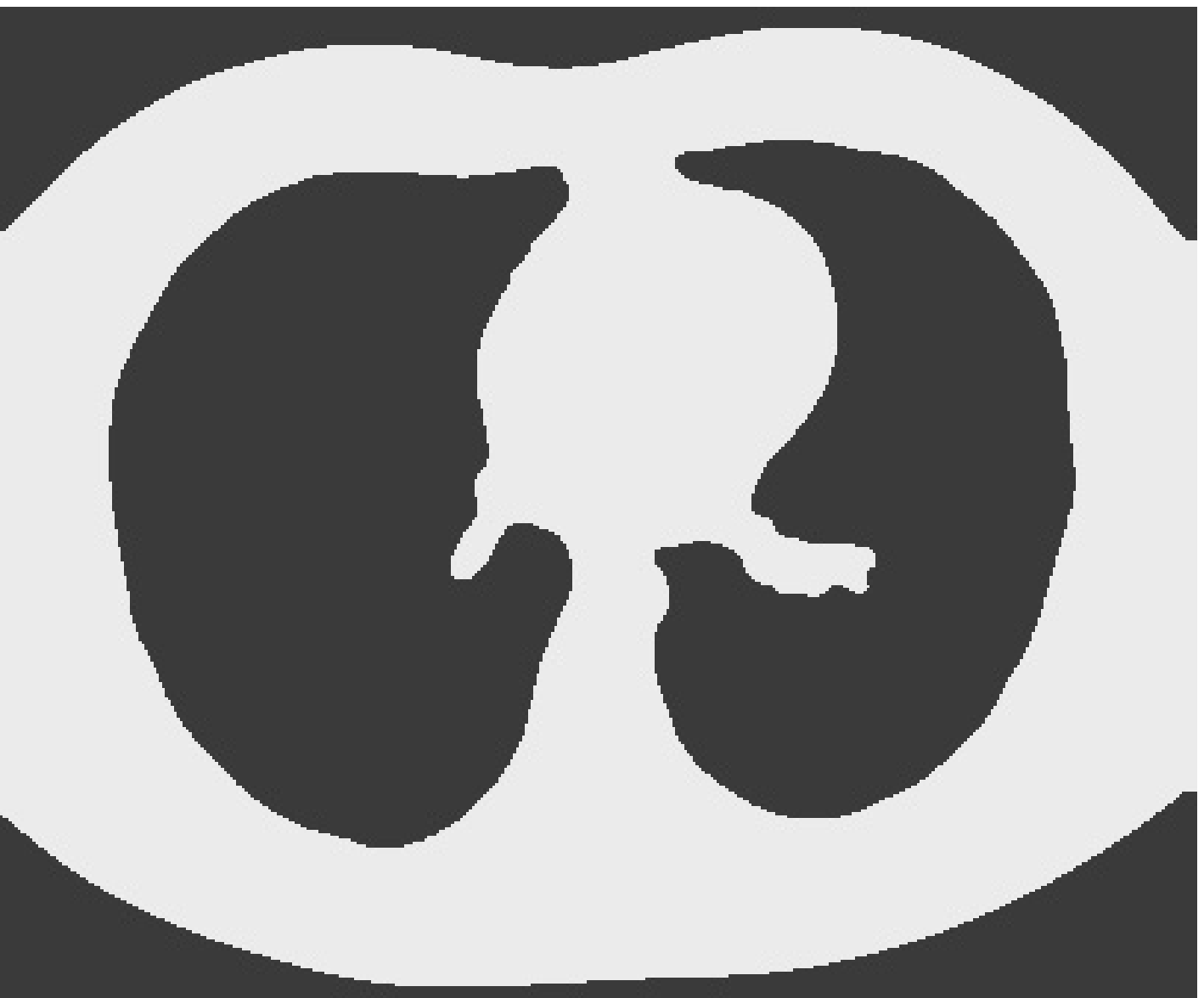}\label{fig:7u}}\
\subfloat[N=4]{\includegraphics[width=1.893cm]{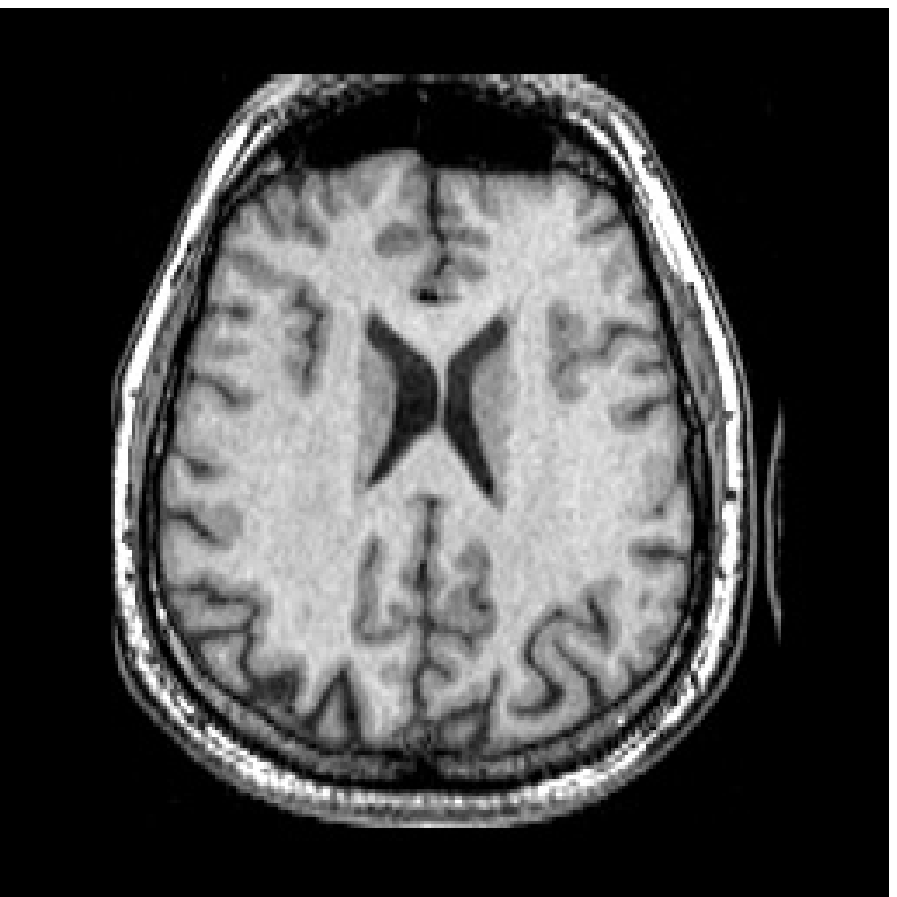}\label{fig:7v}}\
\subfloat[ ]{\includegraphics[width=1.893cm]{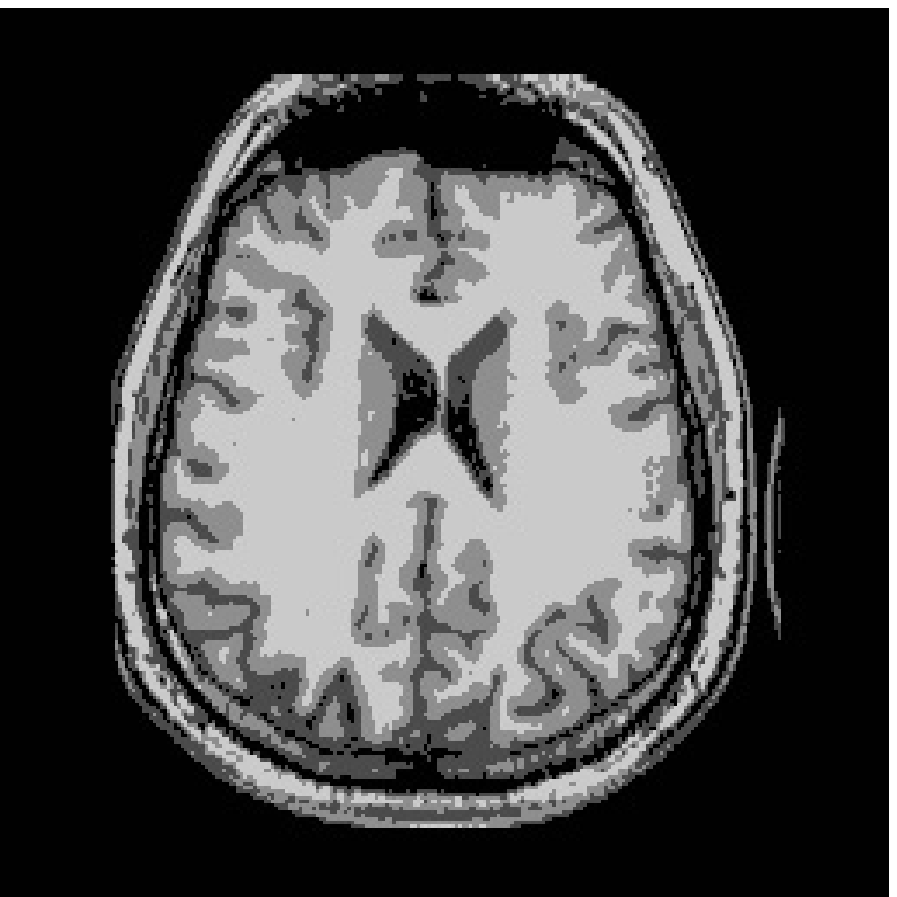}\label{fig:7w}}\
\subfloat[ ]{\includegraphics[width=1.893cm]{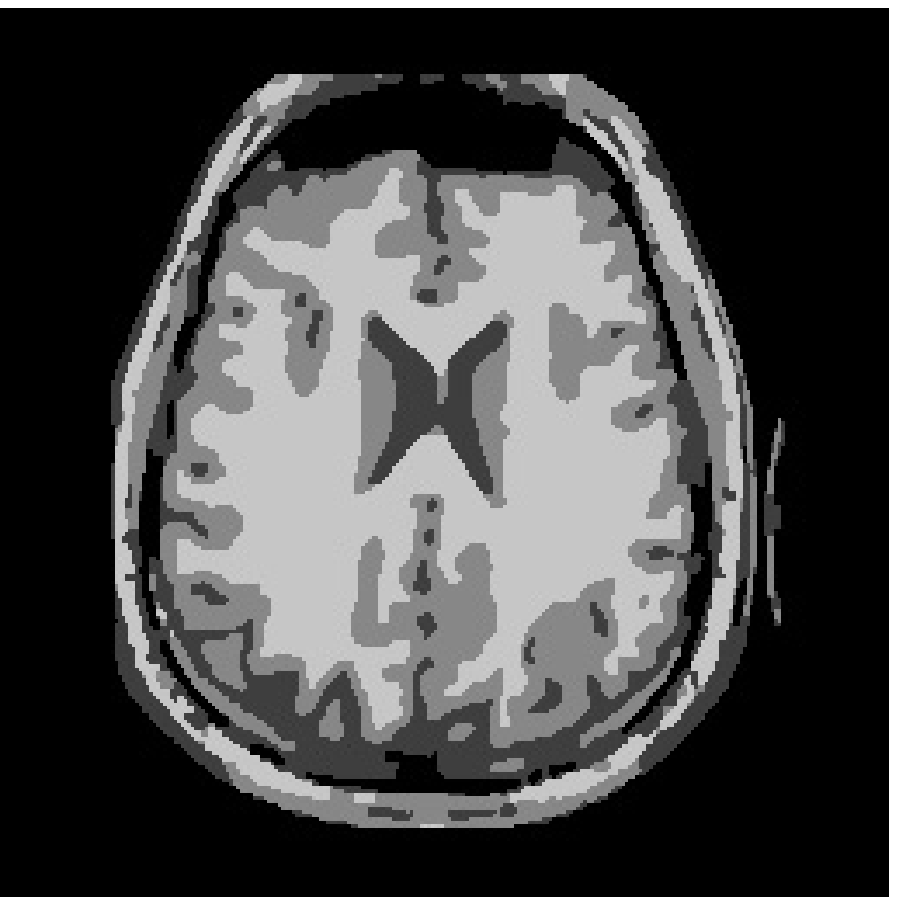}\label{fig:7x}}\\
\subfloat[N=2]{\includegraphics[width=1.893cm]{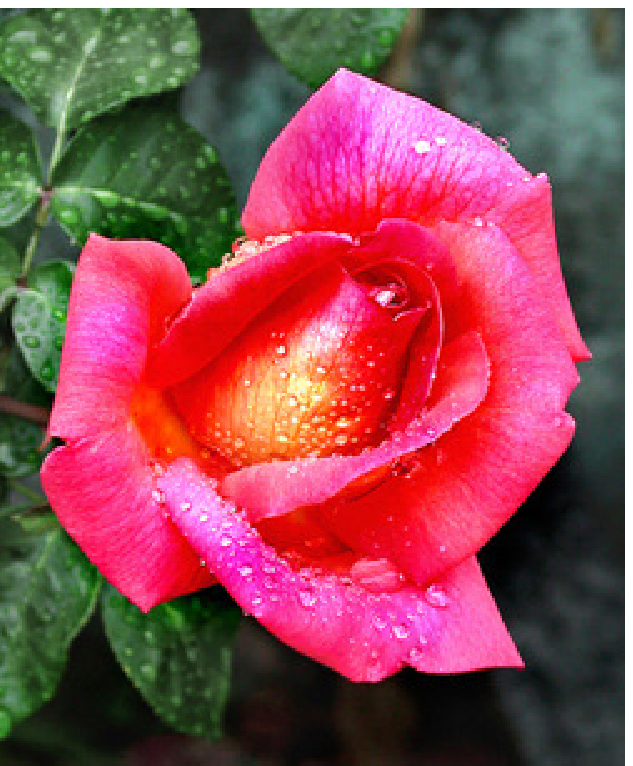}\label{fig:7a}}\
\subfloat[ ]{\includegraphics[width=1.893cm]{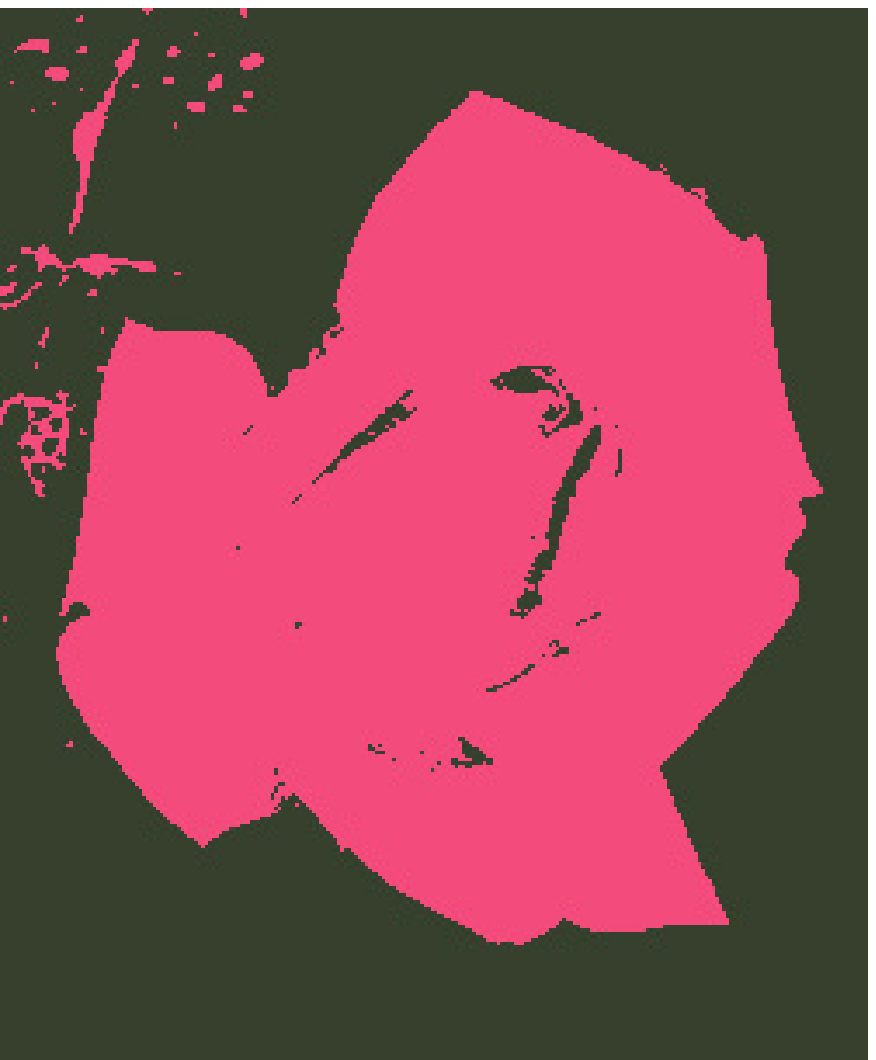}\label{fig:7b}}\
\subfloat[ ]{\includegraphics[width=1.893cm]{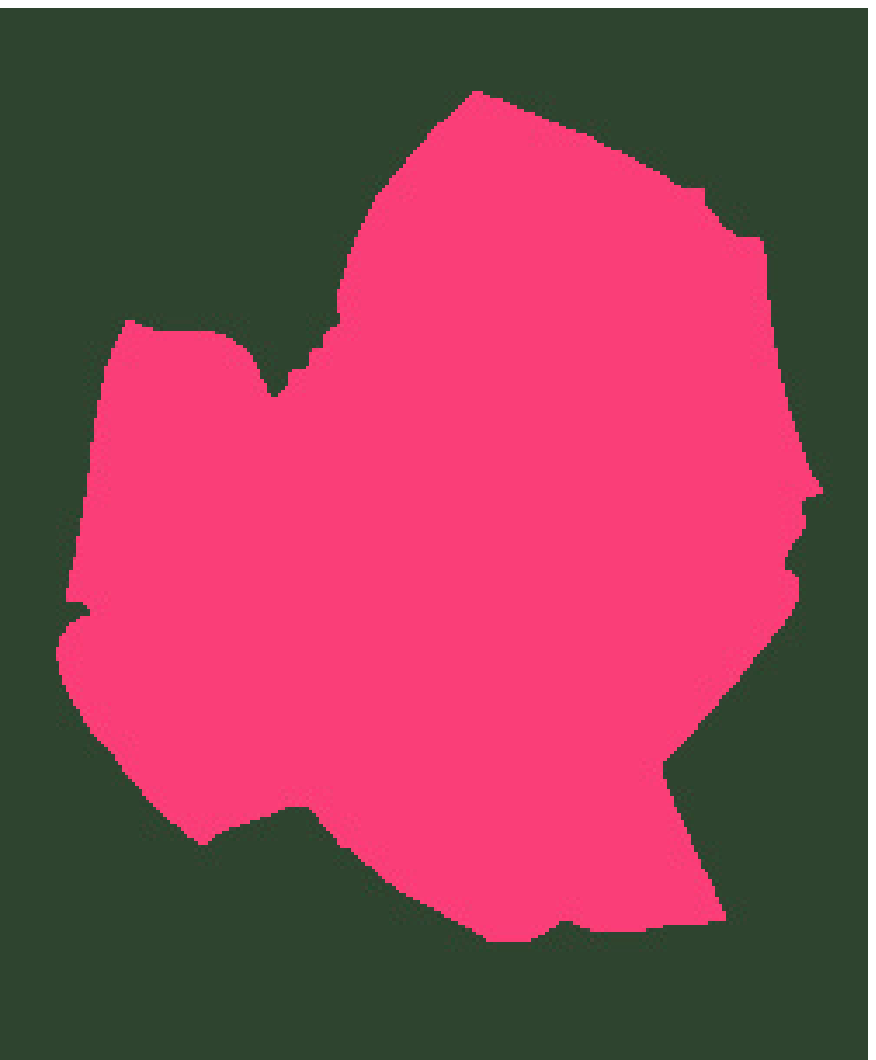}\label{fig:7c}}\
\subfloat[N=6]{\includegraphics[width=1.893cm]{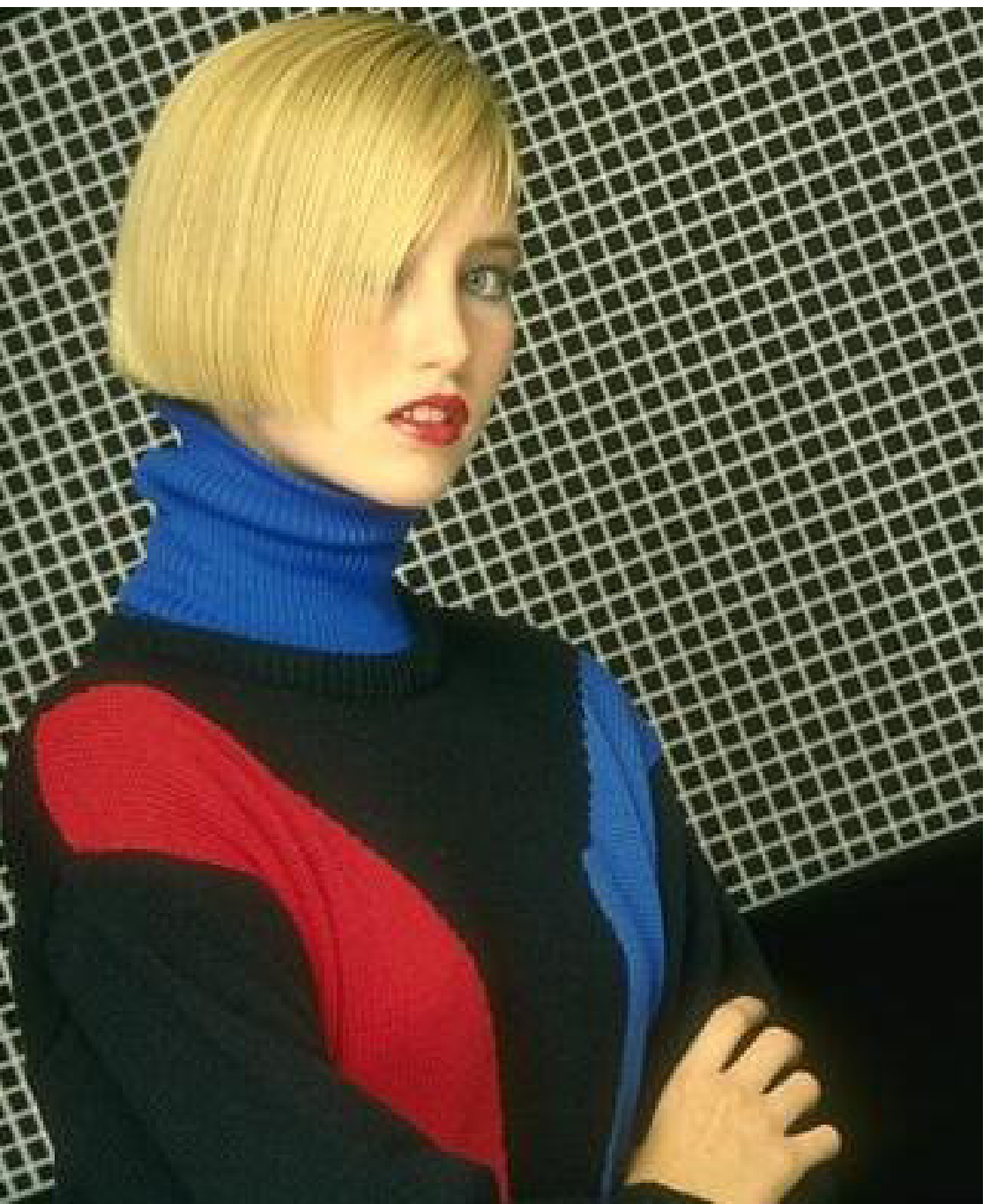}\label{fig:7d}}\
\subfloat[ ]{\includegraphics[width=1.893cm]{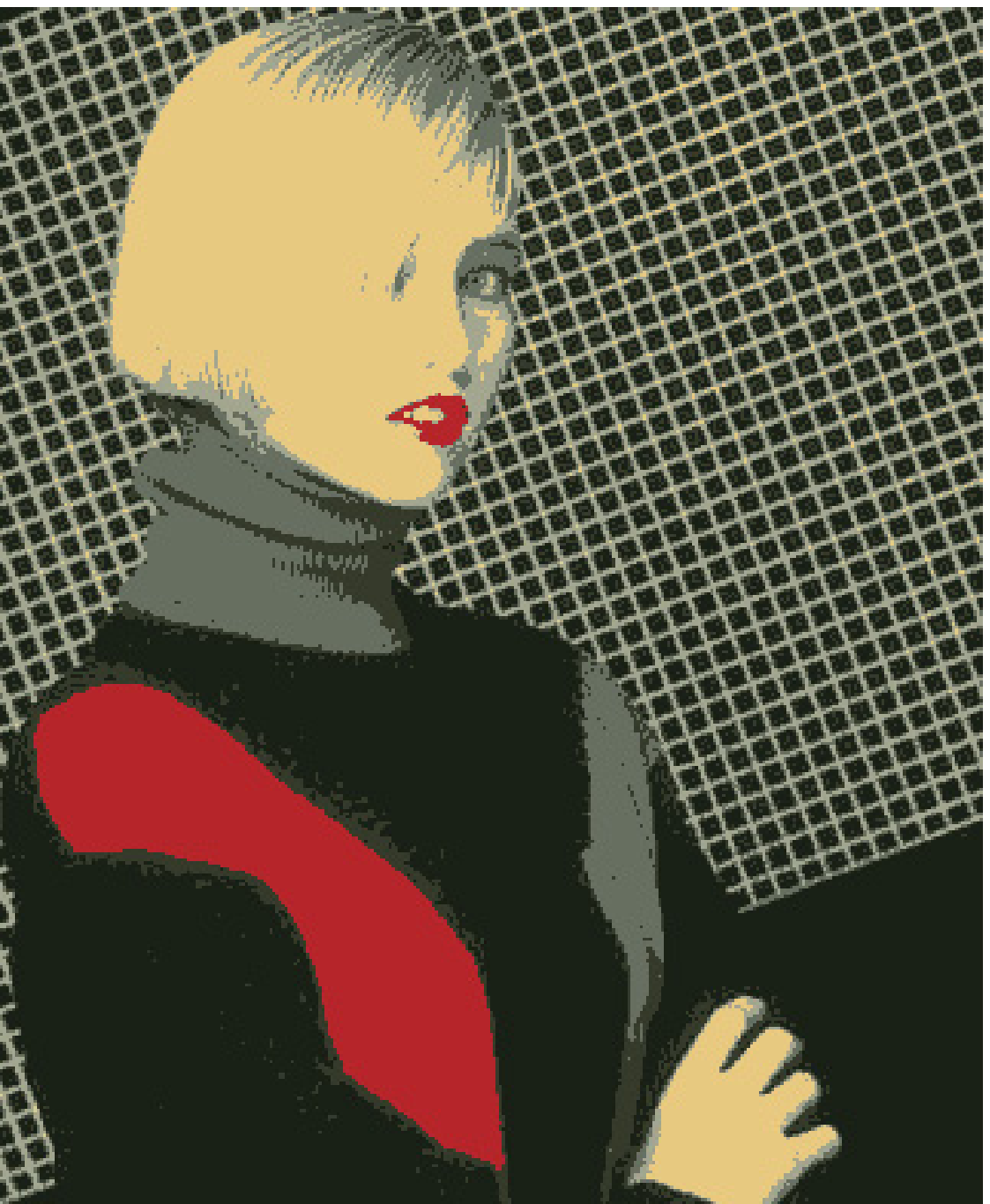}\label{fig:7e}}\
\subfloat[ ]{\includegraphics[width=1.893cm]{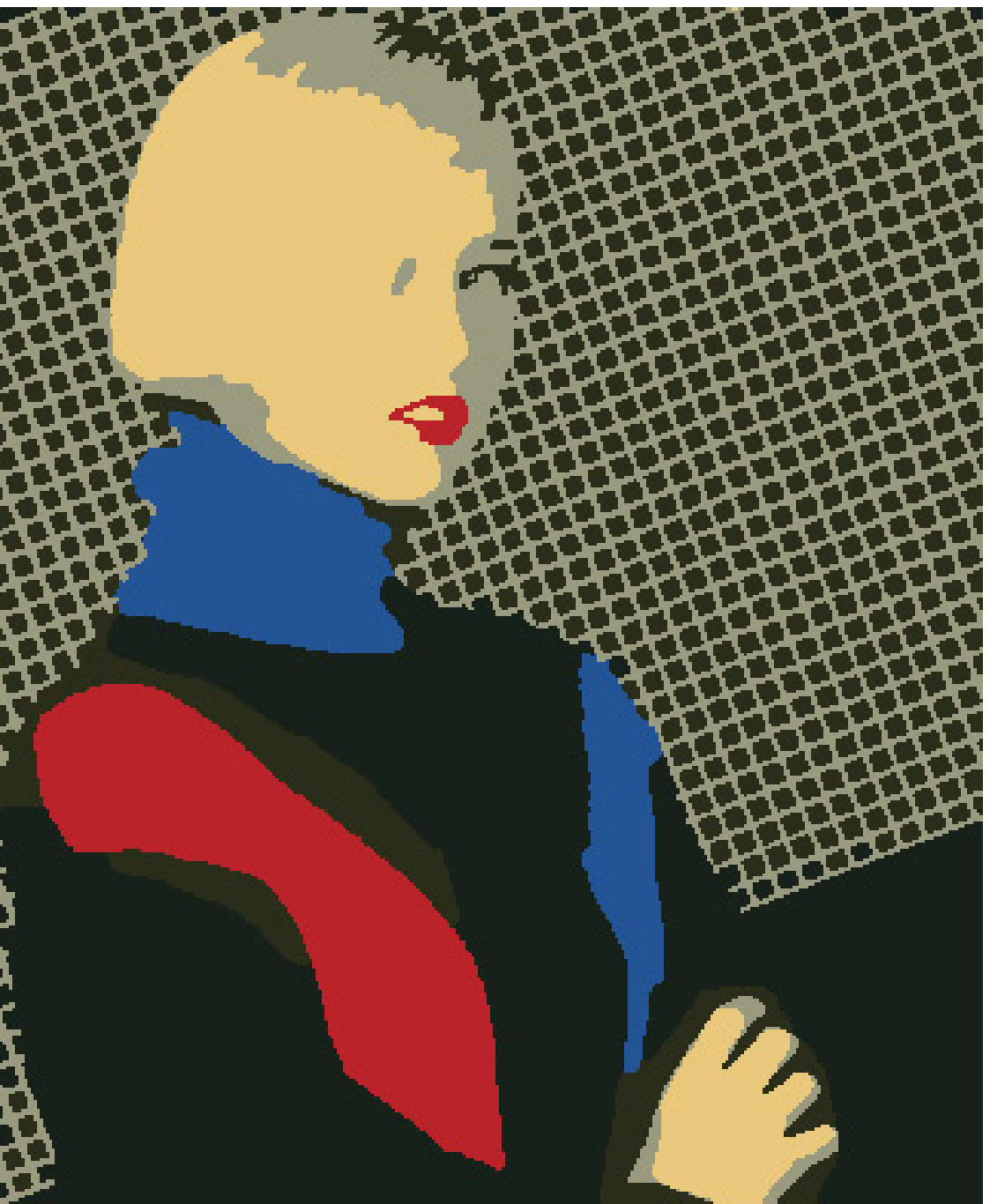}\label{fig:7f}}\\
\subfloat[N=8]{\includegraphics[width=1.893cm]{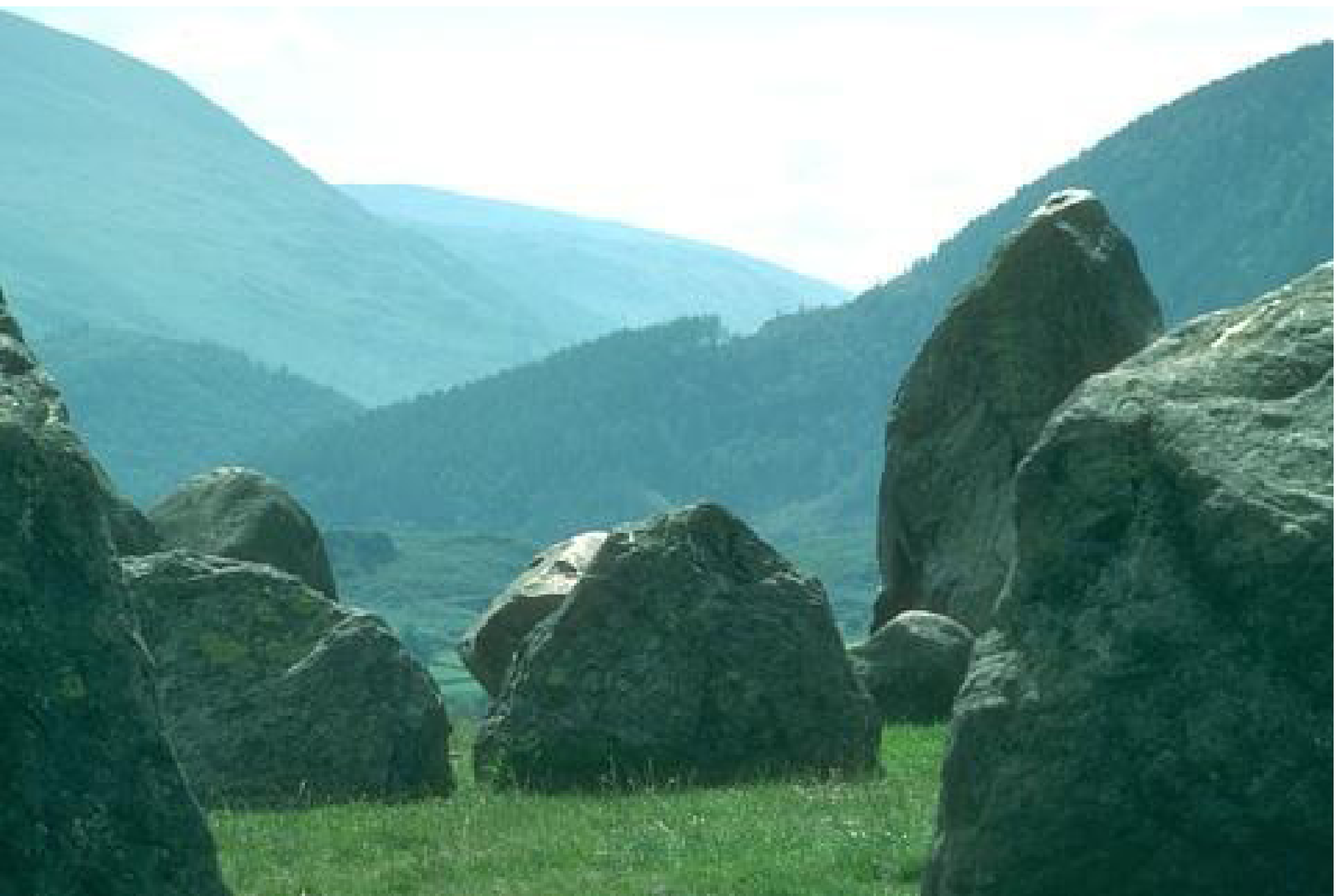}\label{fig:7g}}\
\subfloat[ ]{\includegraphics[width=1.893cm]{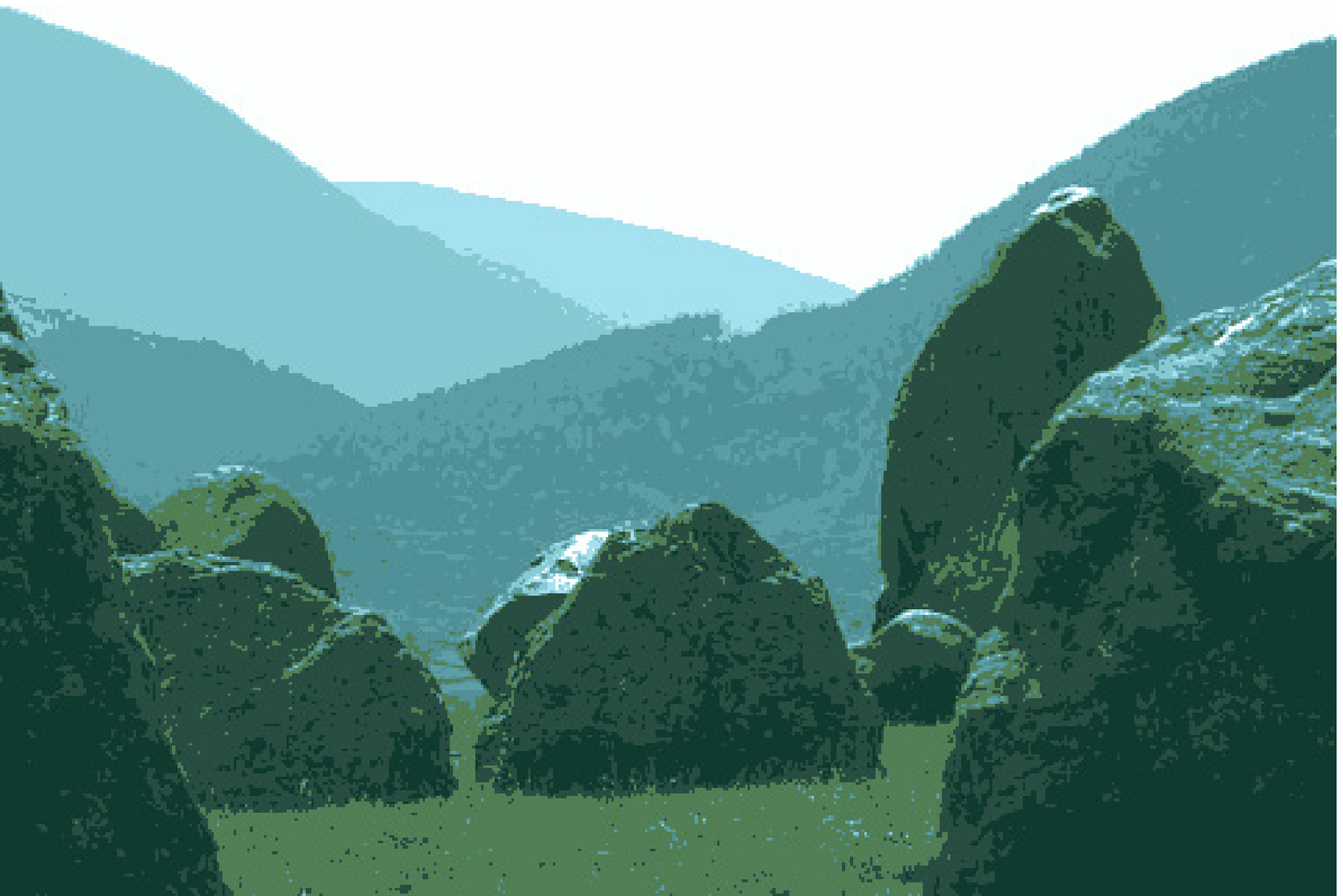}\label{fig:7h}}\
\subfloat[ ]{\includegraphics[width=1.893cm]{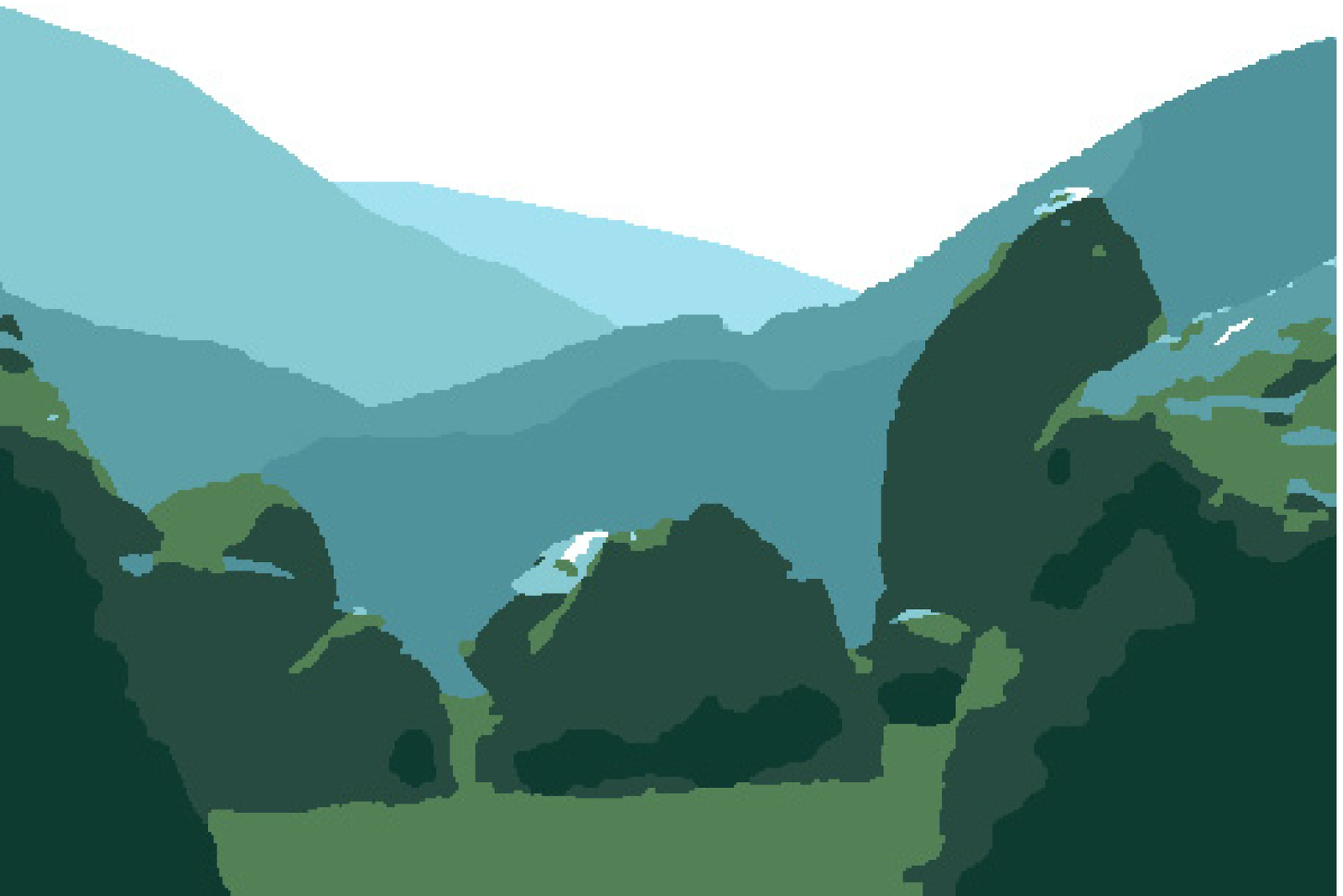}\label{fig:7i}}\
\subfloat[N=4]{\includegraphics[width=1.893cm]{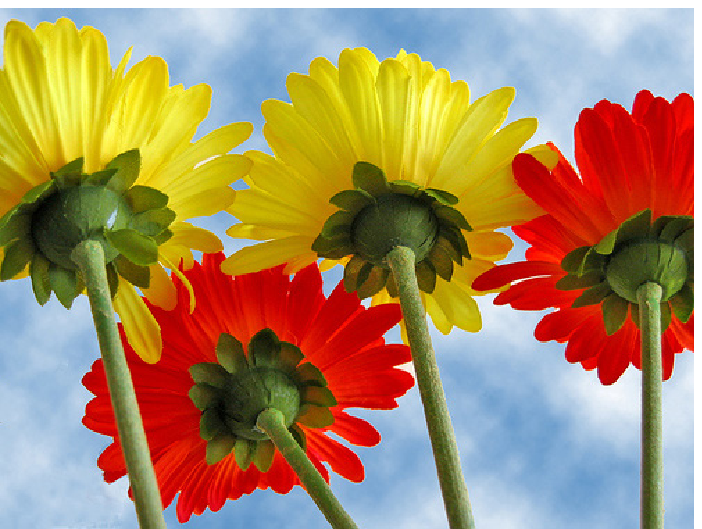}\label{fig:7j}}\
\subfloat[ ]{\includegraphics[width=1.893cm]{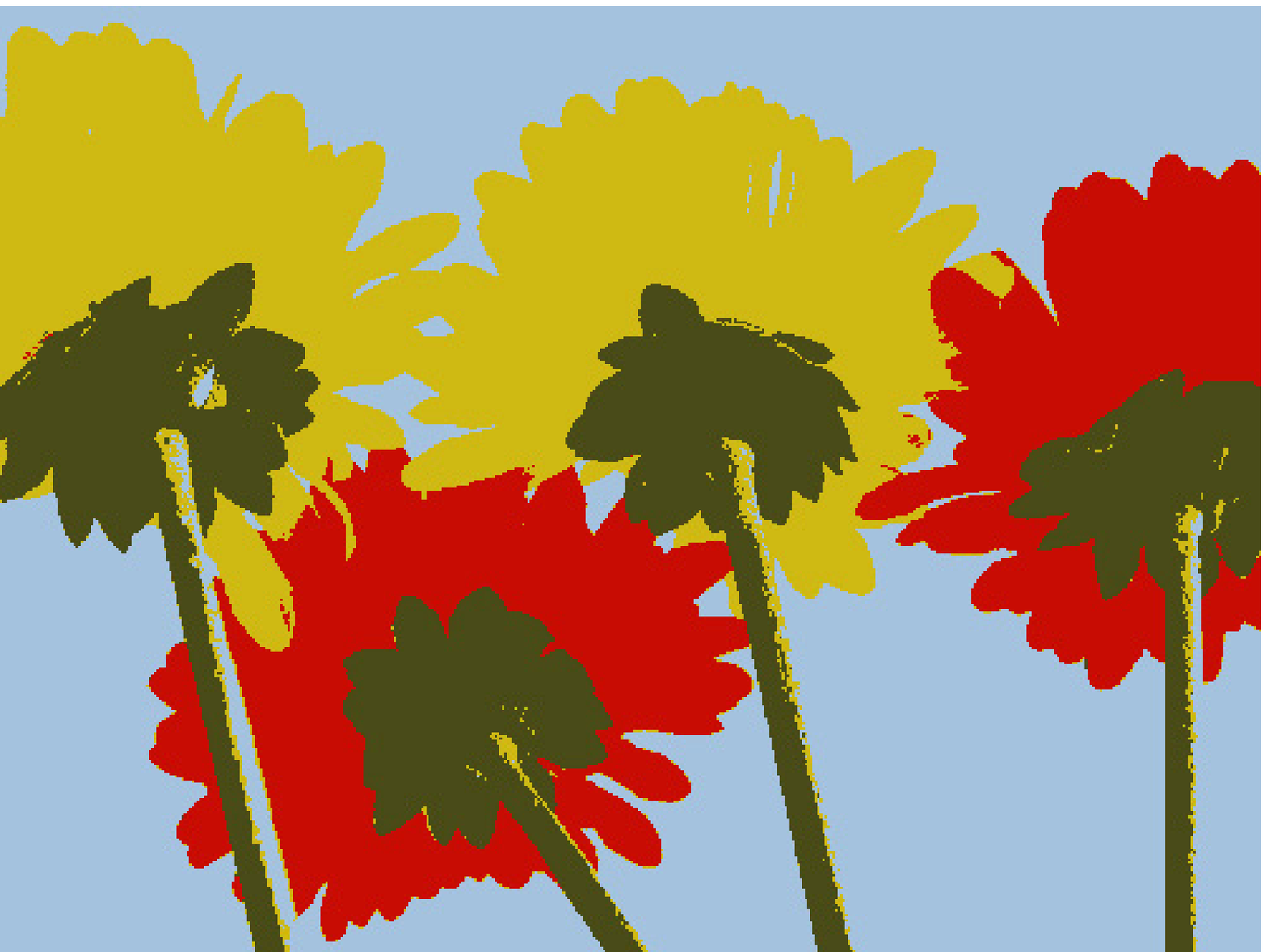}\label{fig:7k}}\
\subfloat[ ]{\includegraphics[width=1.893cm]{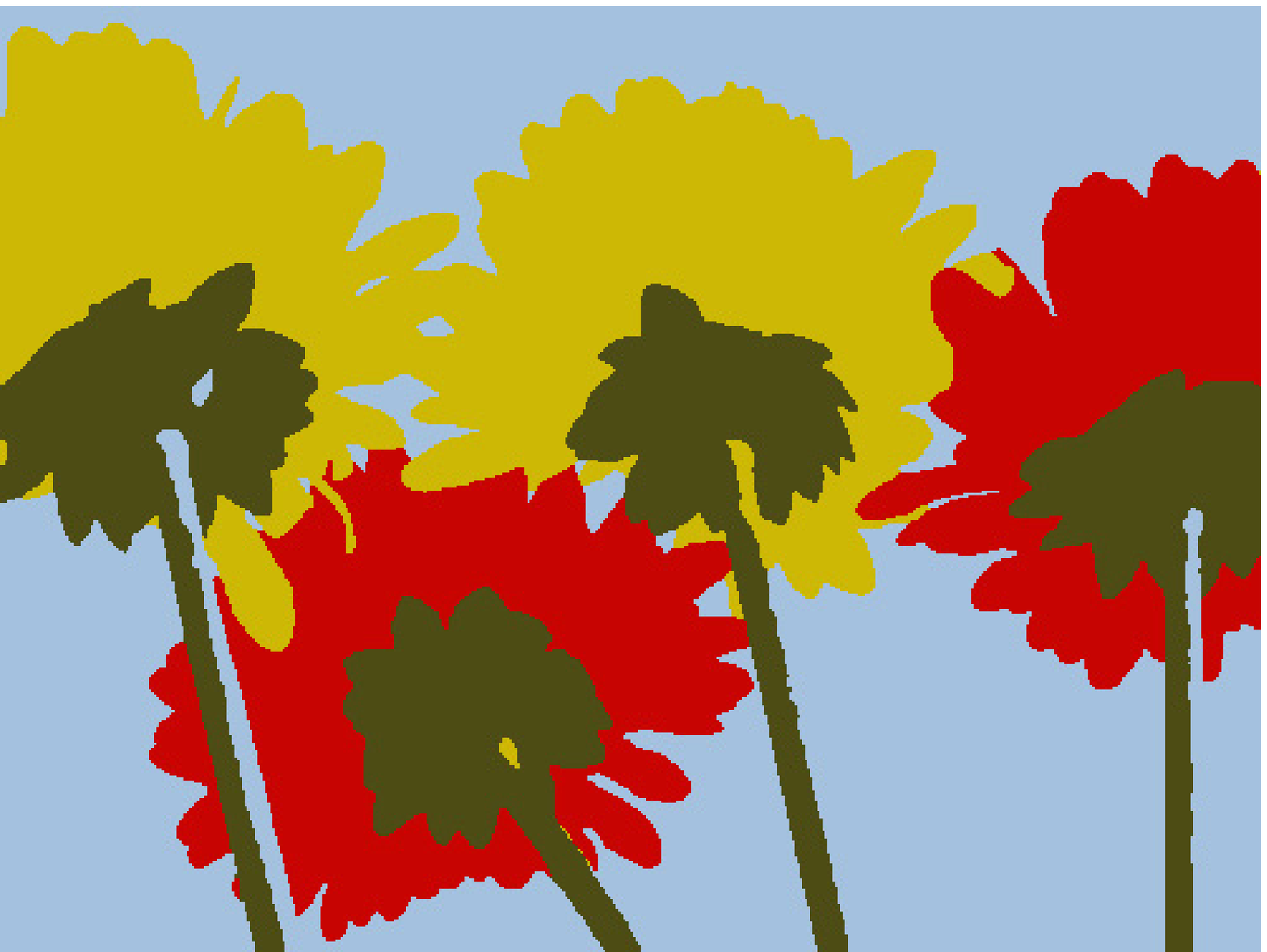}\label{fig:7l}}\\
\subfloat[N=3]{\includegraphics[width=1.893cm]{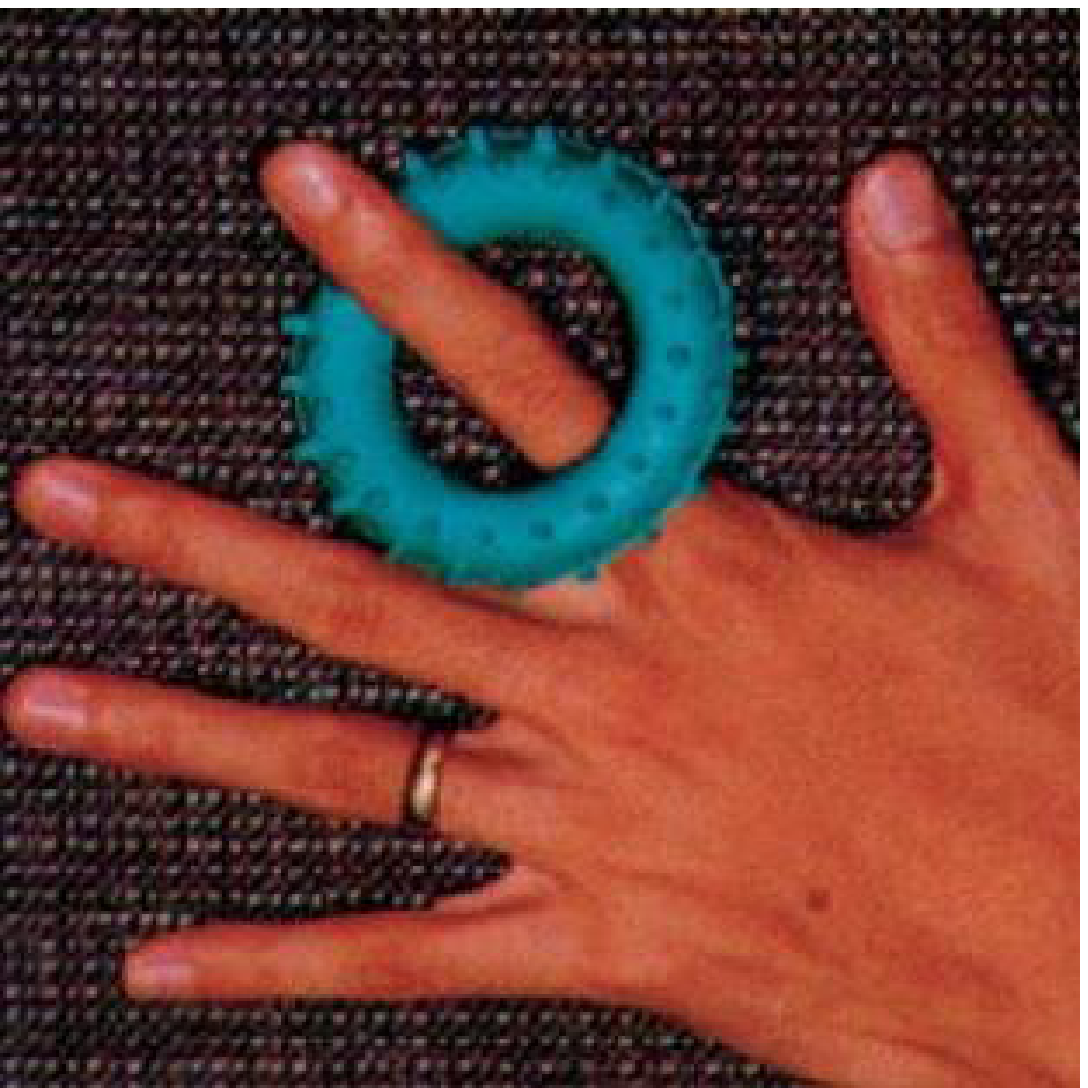}\label{fig:7m}}\
\subfloat[ ]{\includegraphics[width=1.893cm]{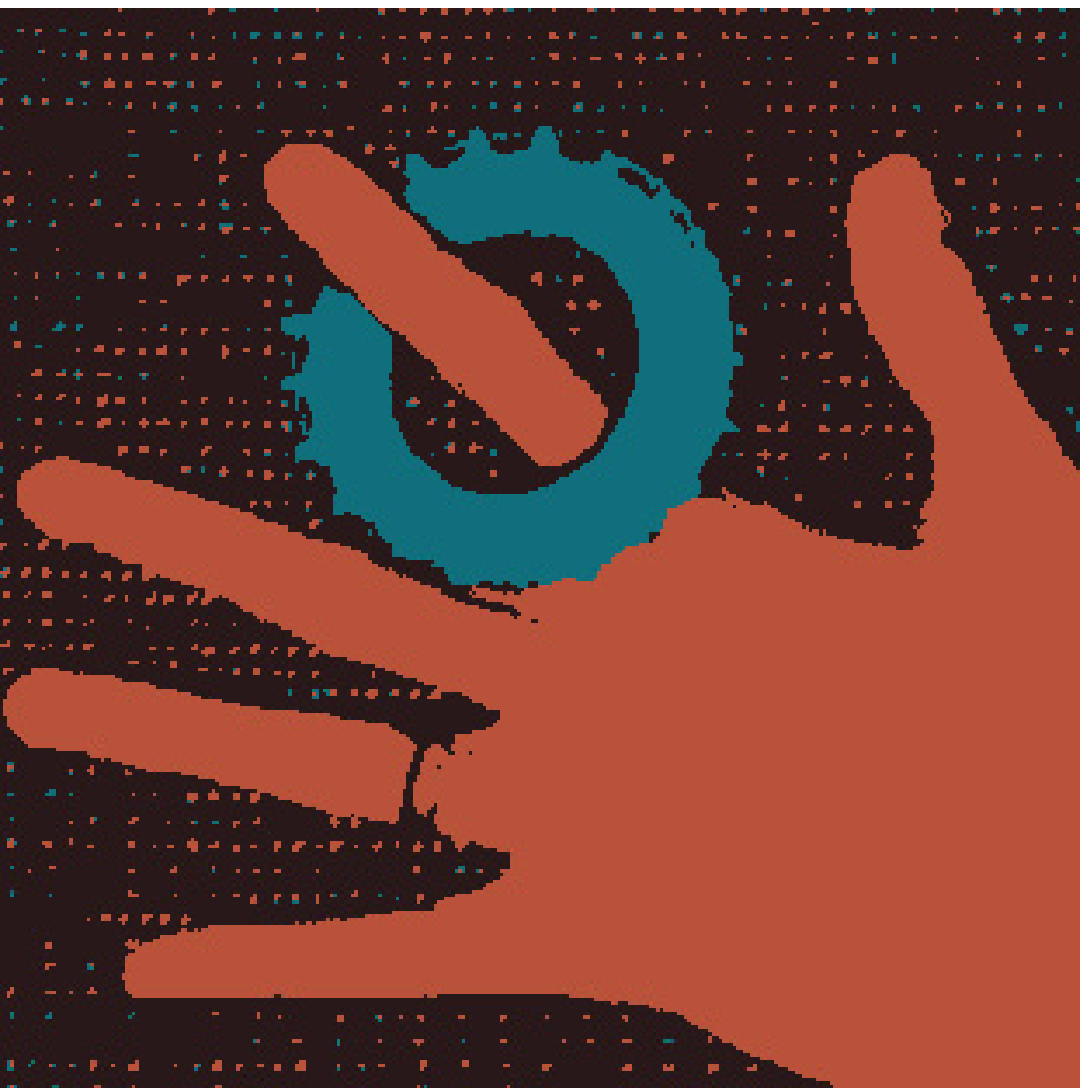}\label{fig:7n}}\
\subfloat[ ]{\includegraphics[width=1.893cm]{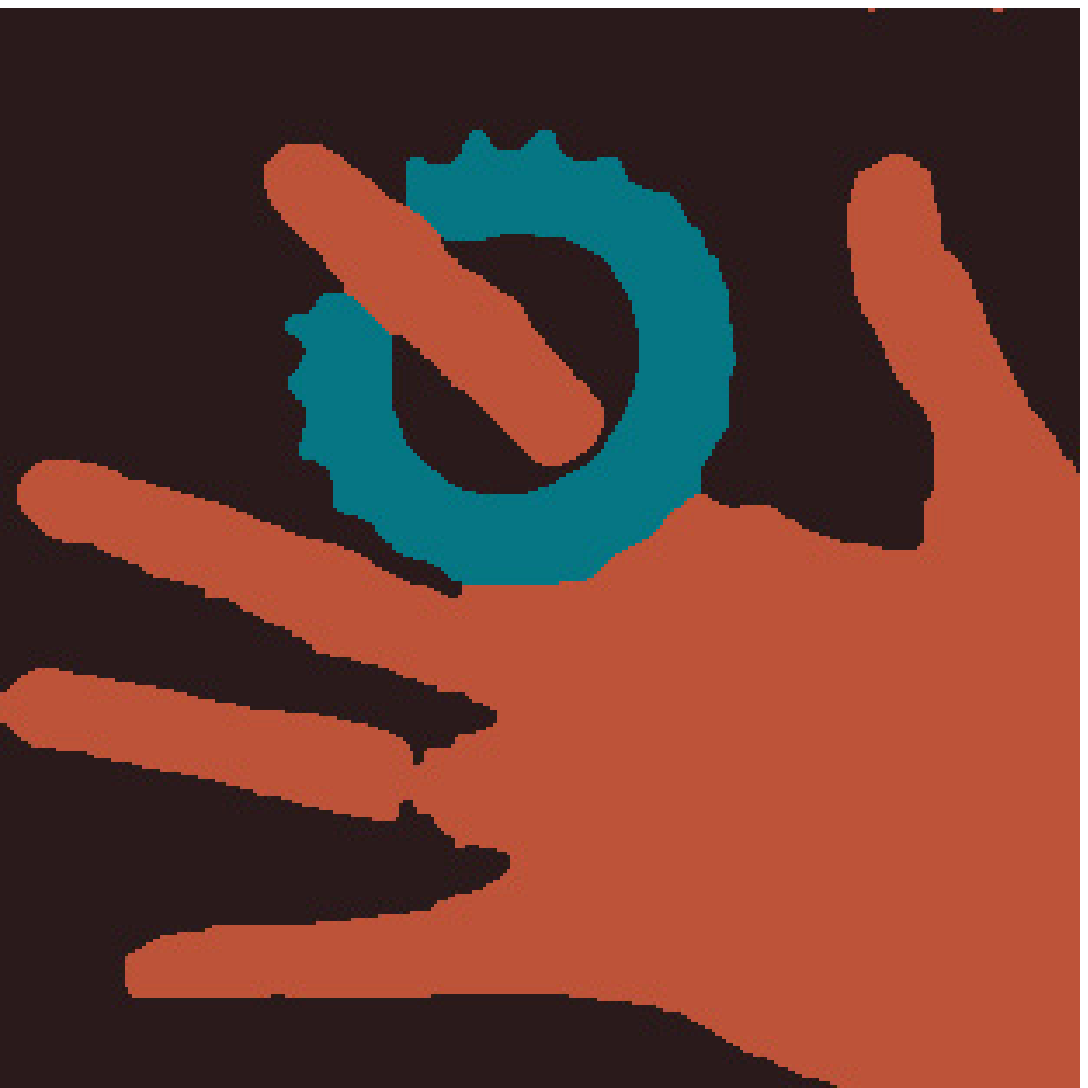}\label{fig:7o}}\
\subfloat[N=6]{\includegraphics[width=1.893cm]{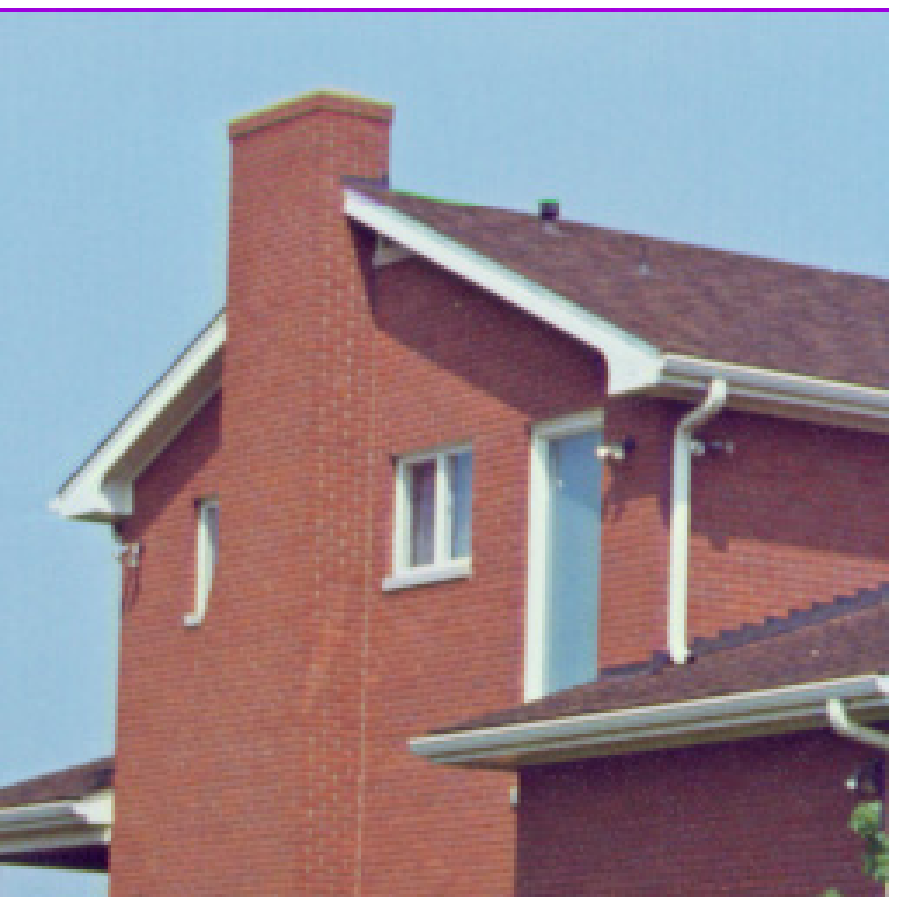}\label{fig:7p}}\
\subfloat[ ]{\includegraphics[width=1.893cm]{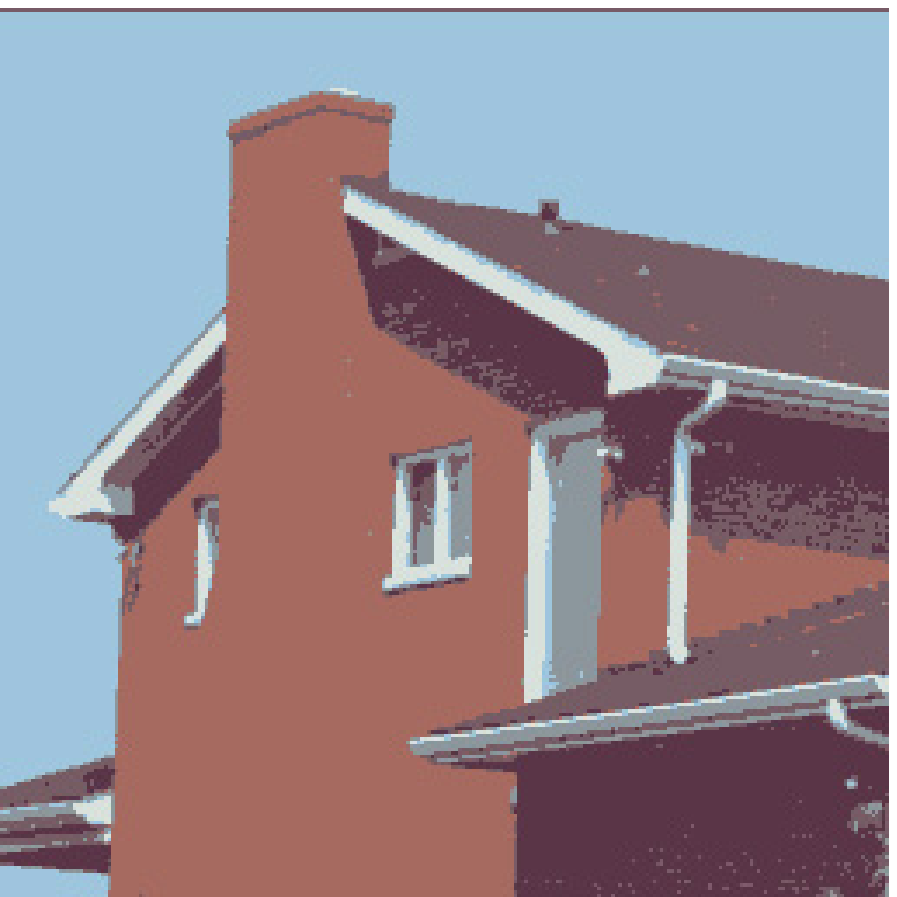}\label{fig:7q}}\
\subfloat[ ]{\includegraphics[width=1.893cm]{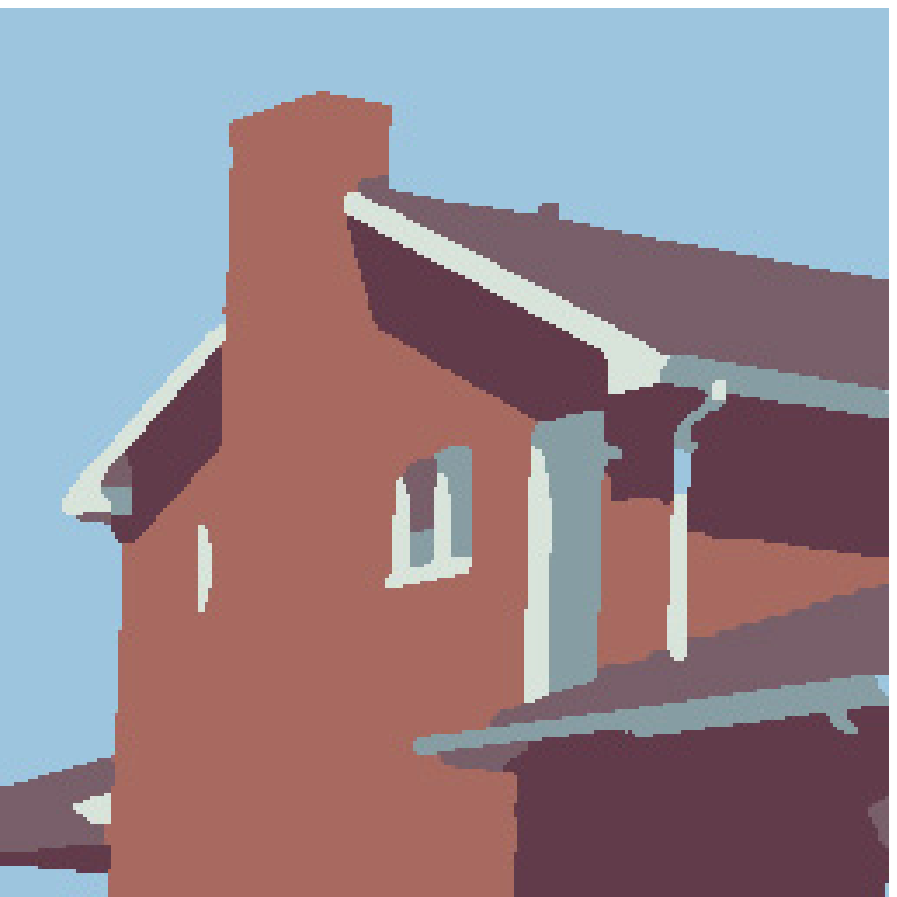}\label{fig:7r}}
\end{center}
\caption{Segmentation on real images. First column and fourth column: real color images;
Second column and fifth column: results of FCM; Third column and last column: results of
L1FS.}\label{fig:7}
\end{figure}

\begin{figure}[!htbp]\label{cluster}%
\begin{center}
\subfloat[ ]{\includegraphics[width=1.893cm]{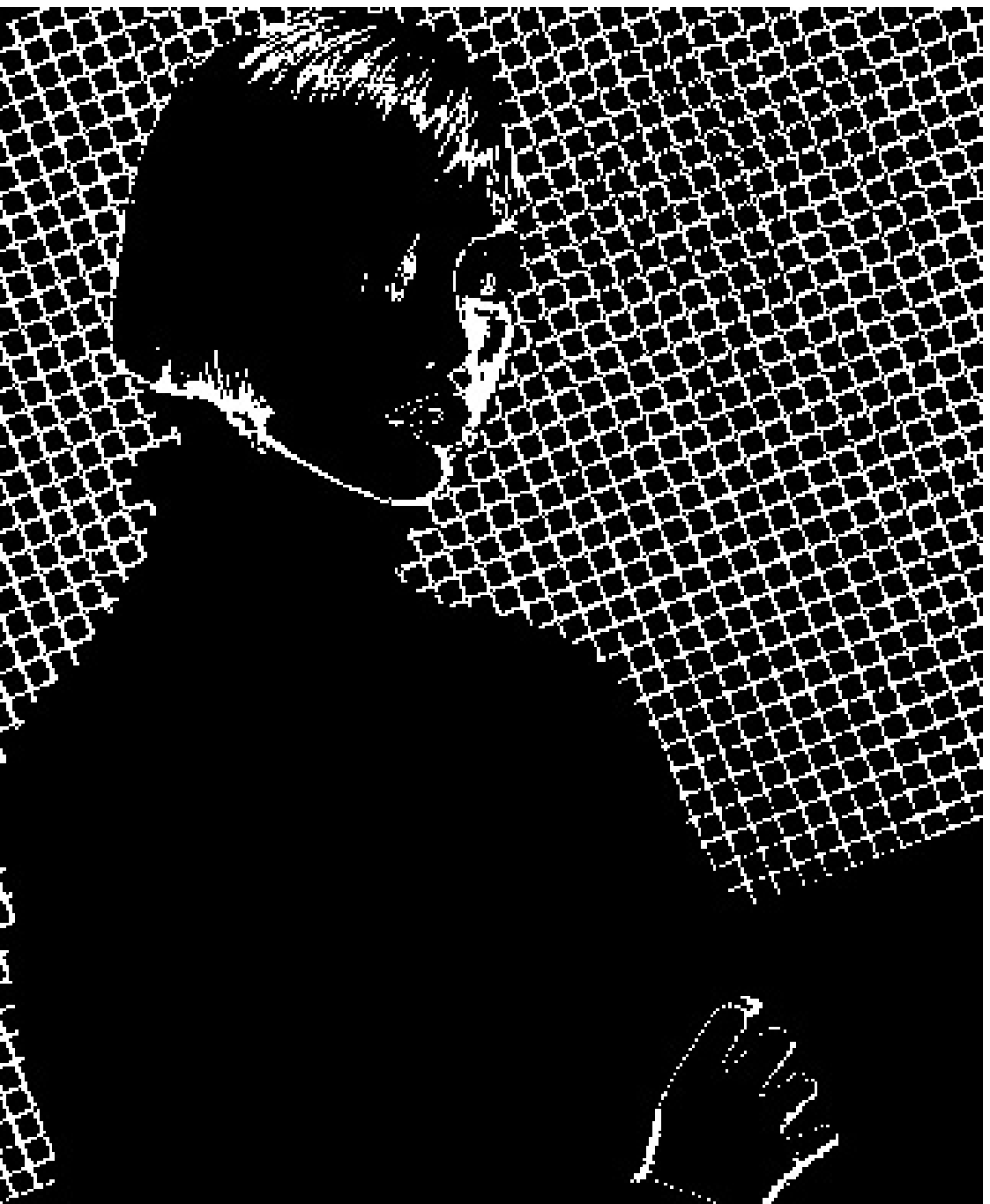}\label{fig:8a}}\
\subfloat[ ]{\includegraphics[width=1.893cm]{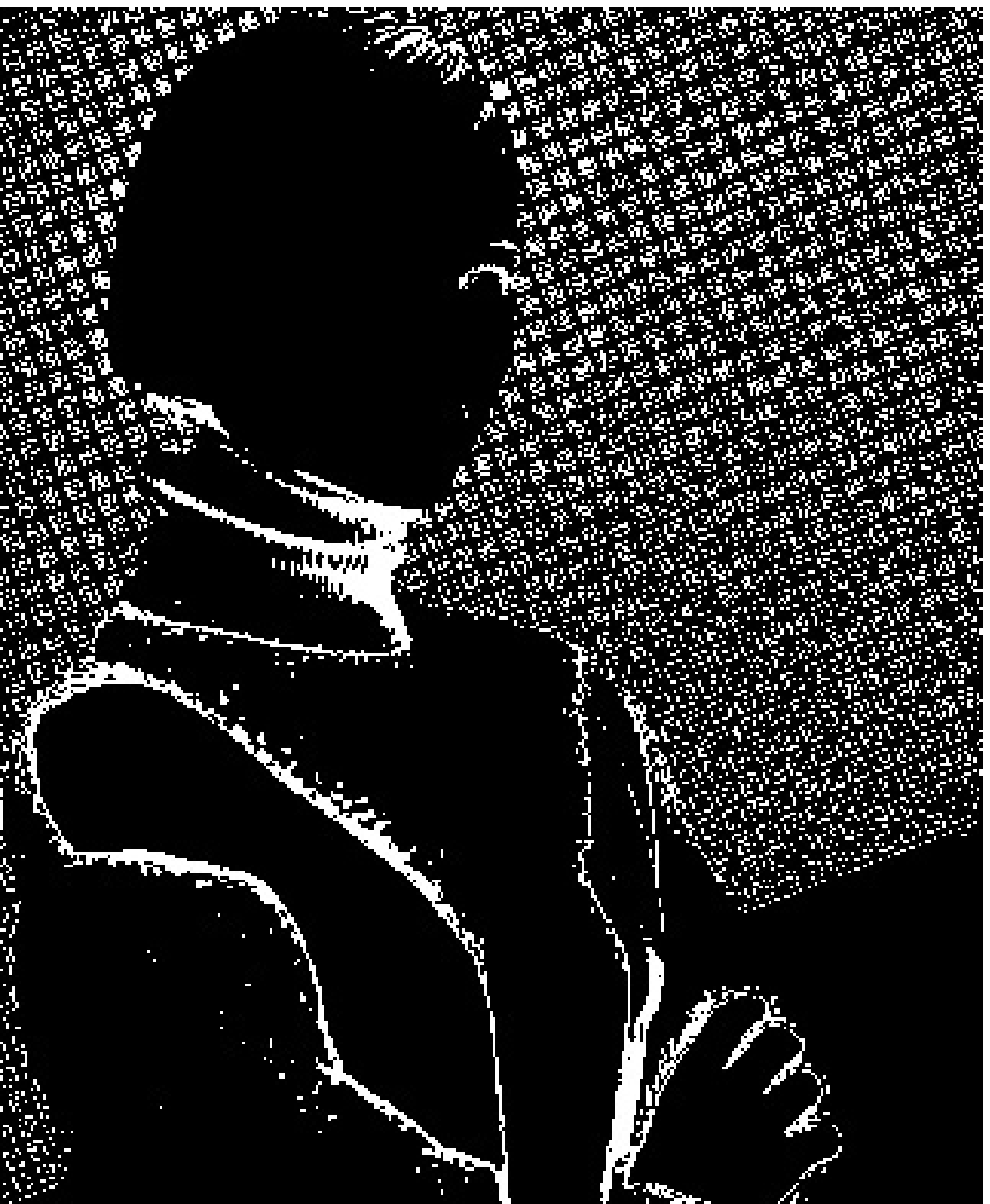}\label{fig:8b}}\
\subfloat[ ]{\includegraphics[width=1.893cm]{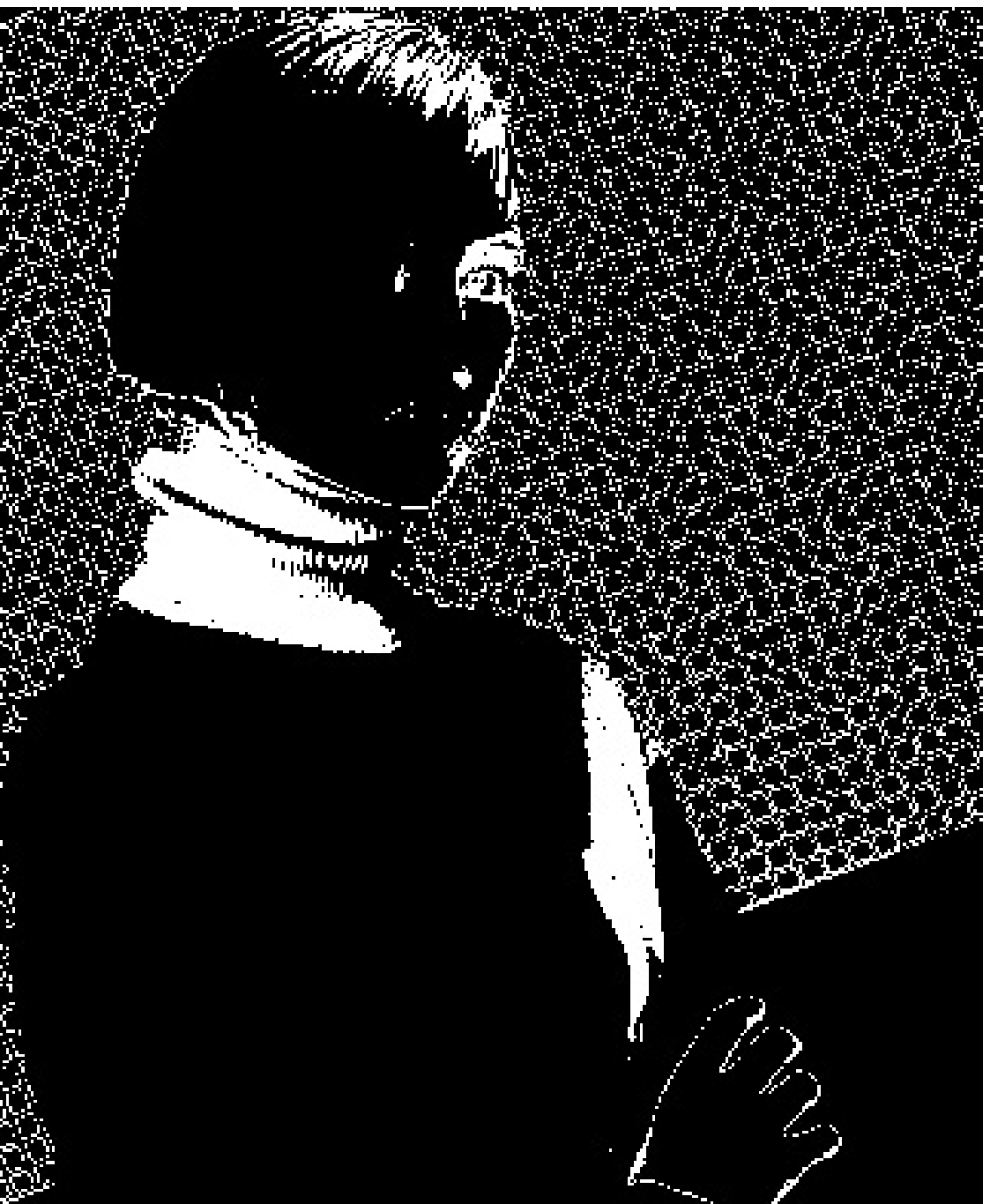}\label{fig:8c}}\
\subfloat[ ]{\includegraphics[width=1.893cm]{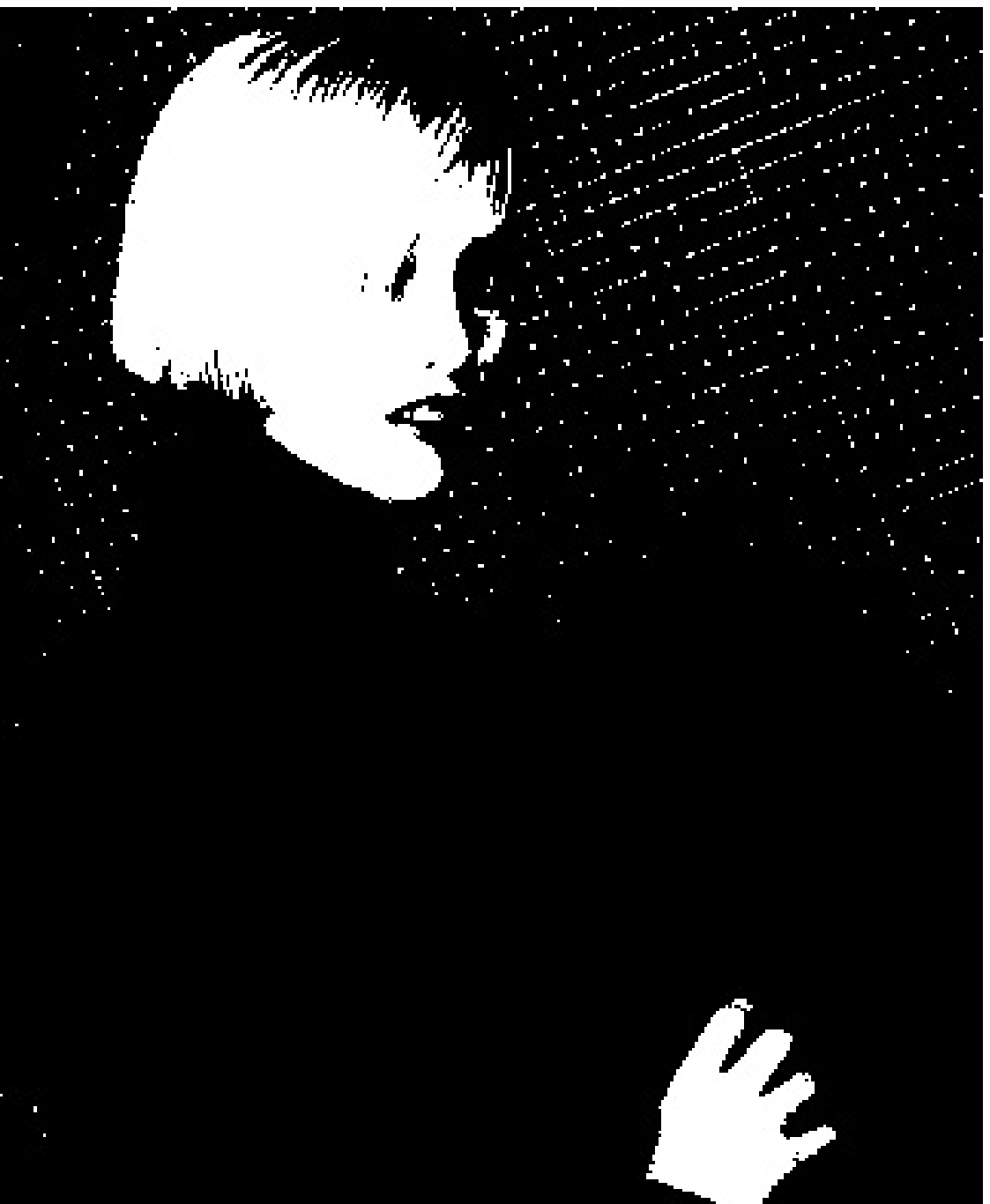}\label{fig:8d}}\
\subfloat[ ]{\includegraphics[width=1.893cm]{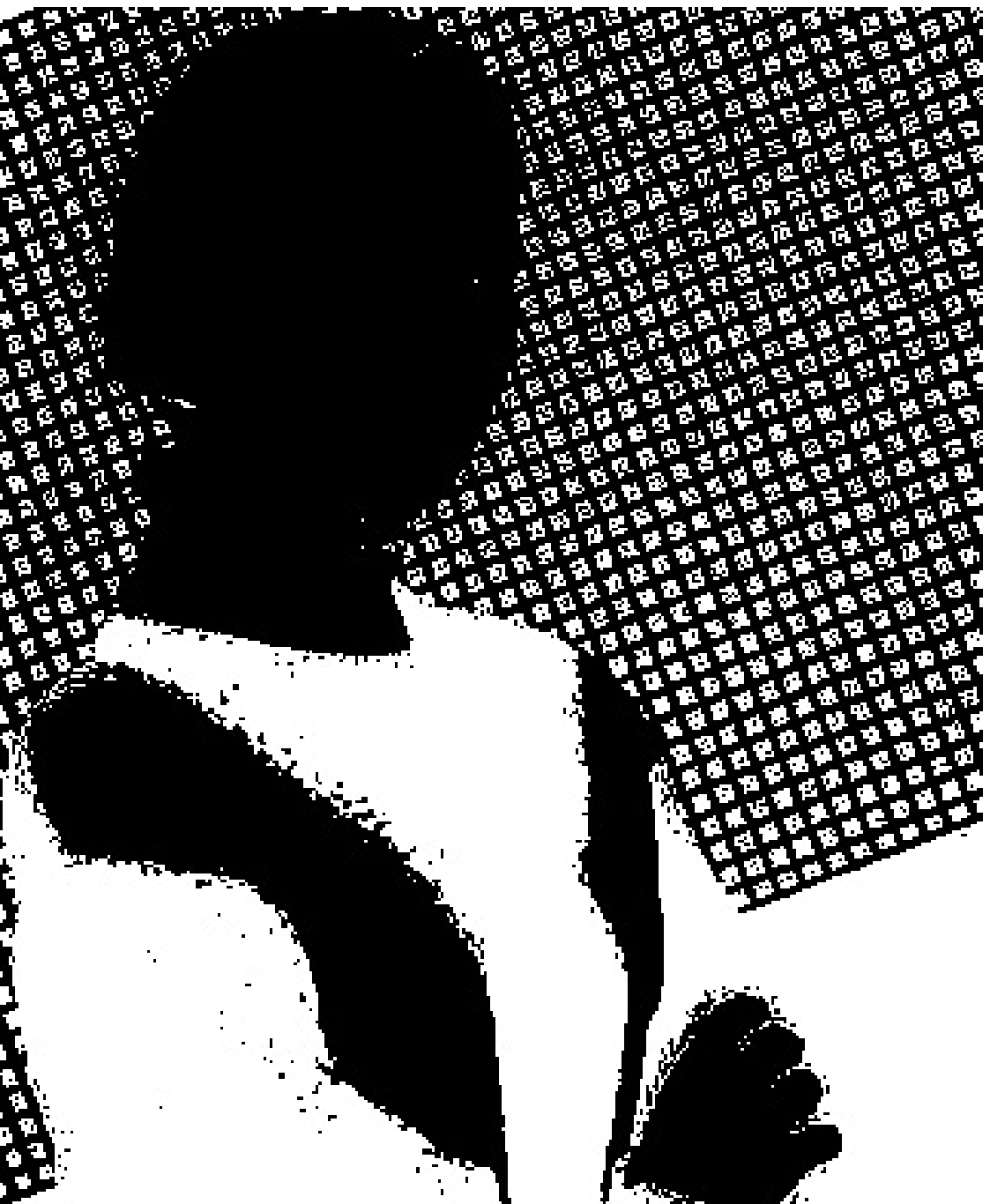}\label{fig:8e}}\
\subfloat[ ]{\includegraphics[width=1.893cm]{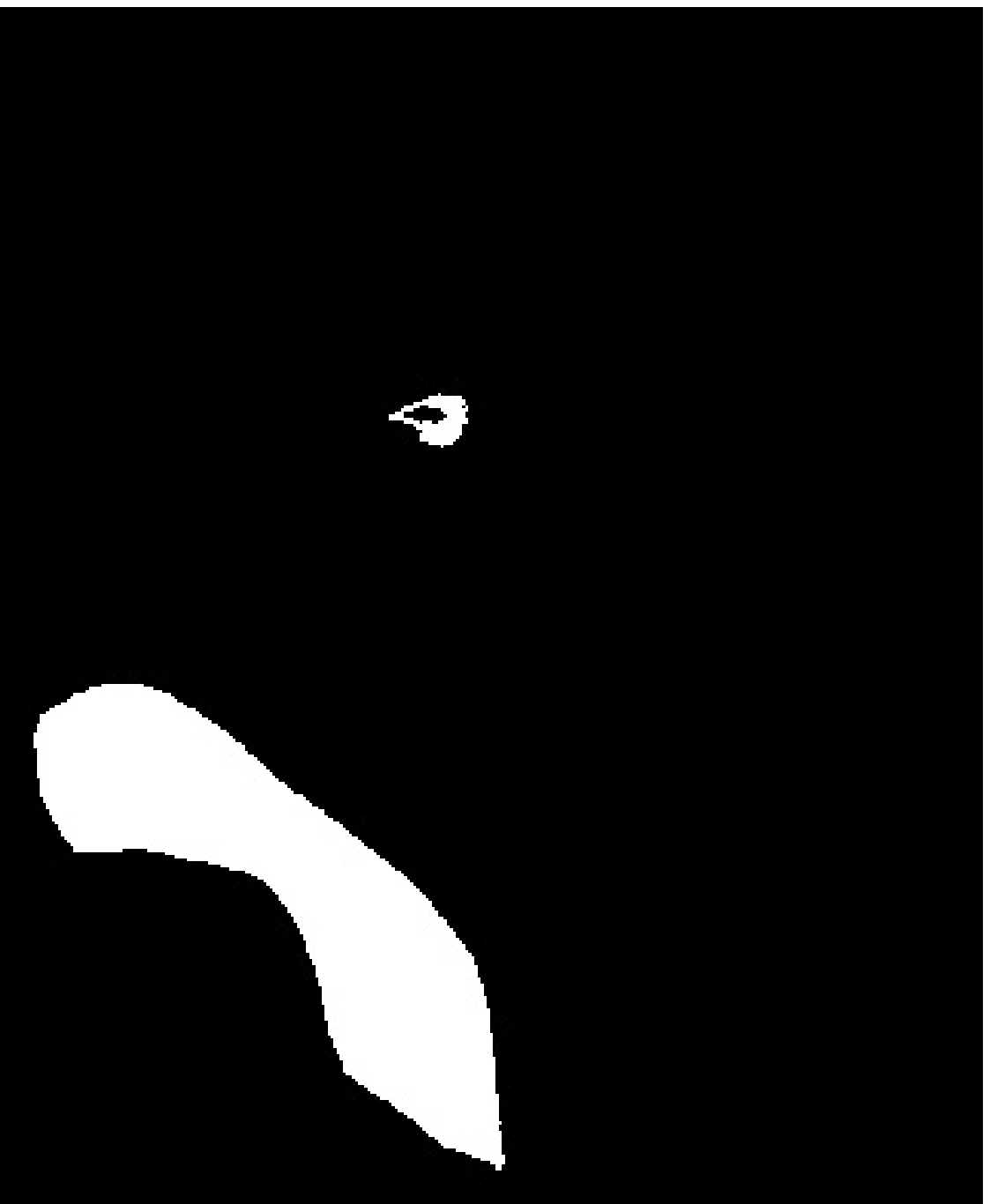}\label{fig:8f}}\\
\subfloat[ ]{\includegraphics[width=1.893cm]{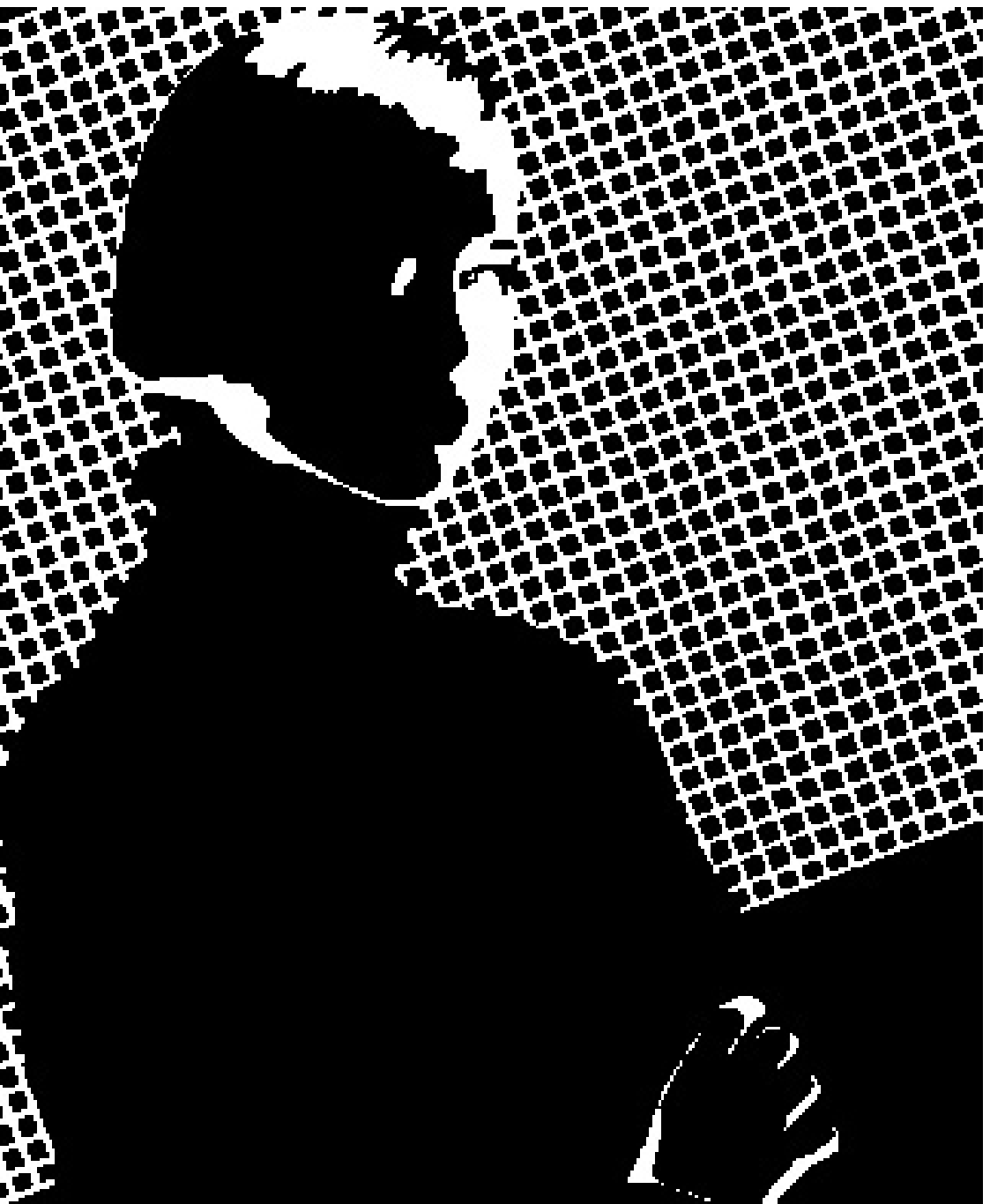}\label{fig:8g}}\
\subfloat[ ]{\includegraphics[width=1.893cm]{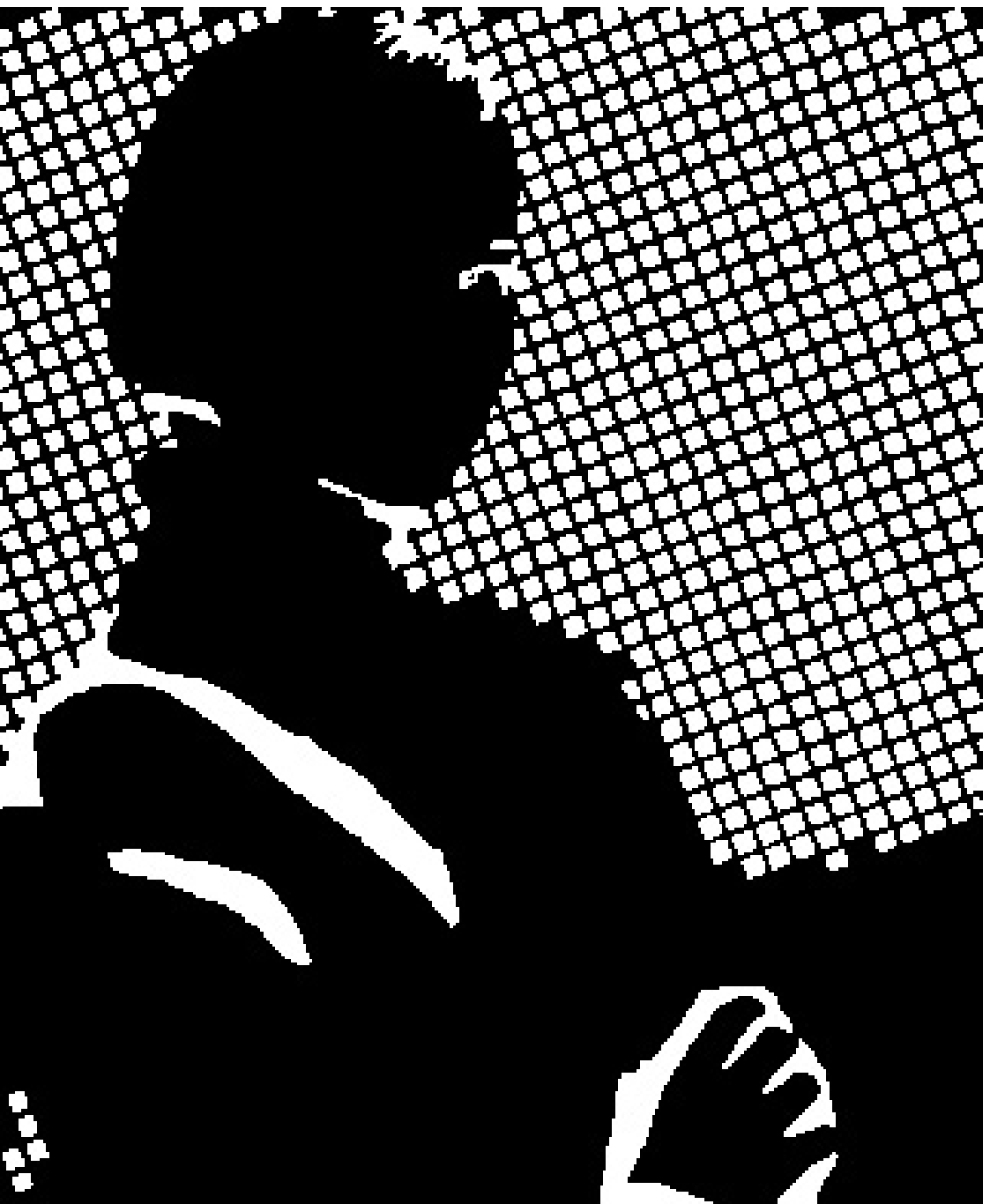}\label{fig:8h}}\
\subfloat[ ]{\includegraphics[width=1.893cm]{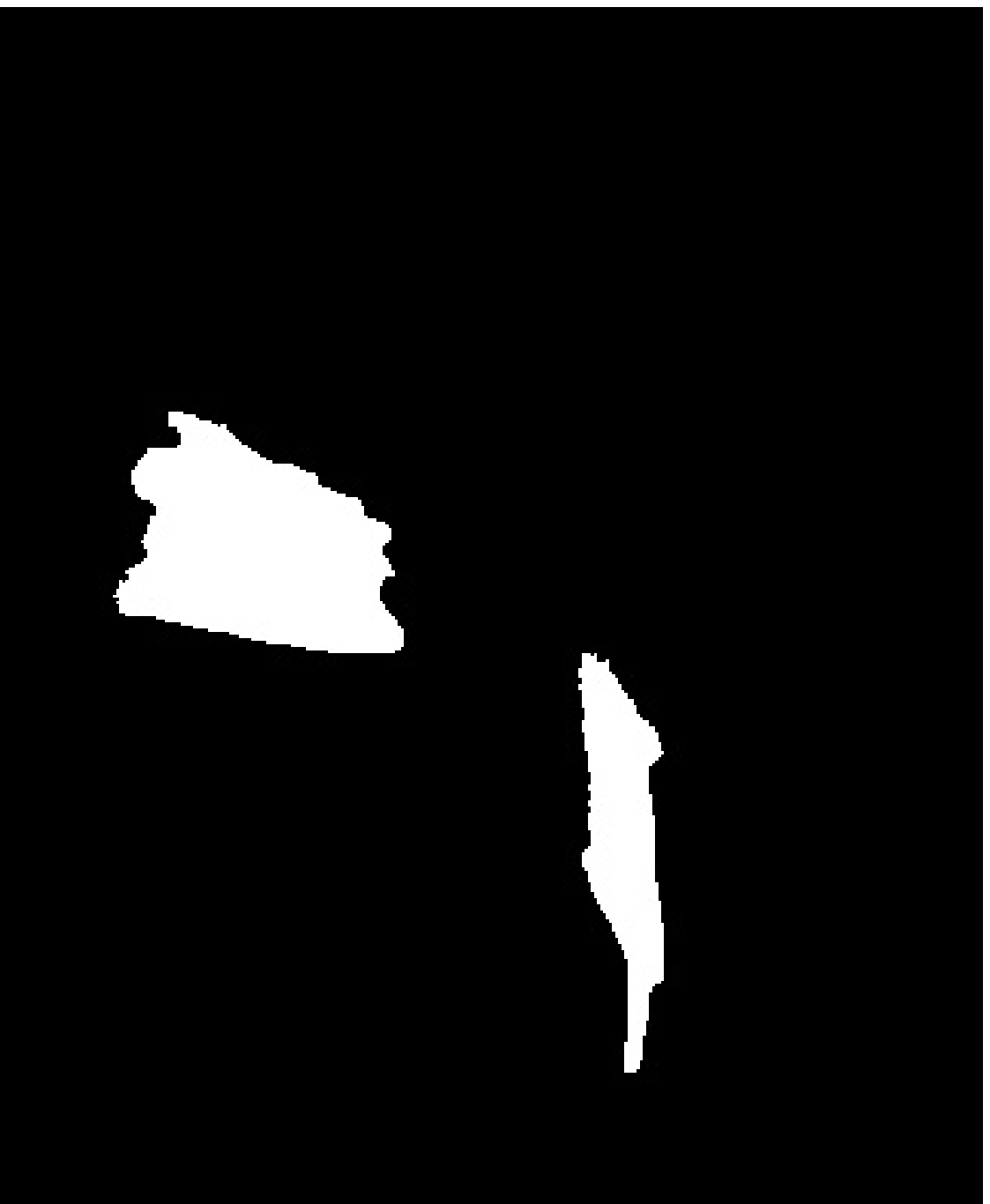}\label{fig:8i}}\
\subfloat[ ]{\includegraphics[width=1.893cm]{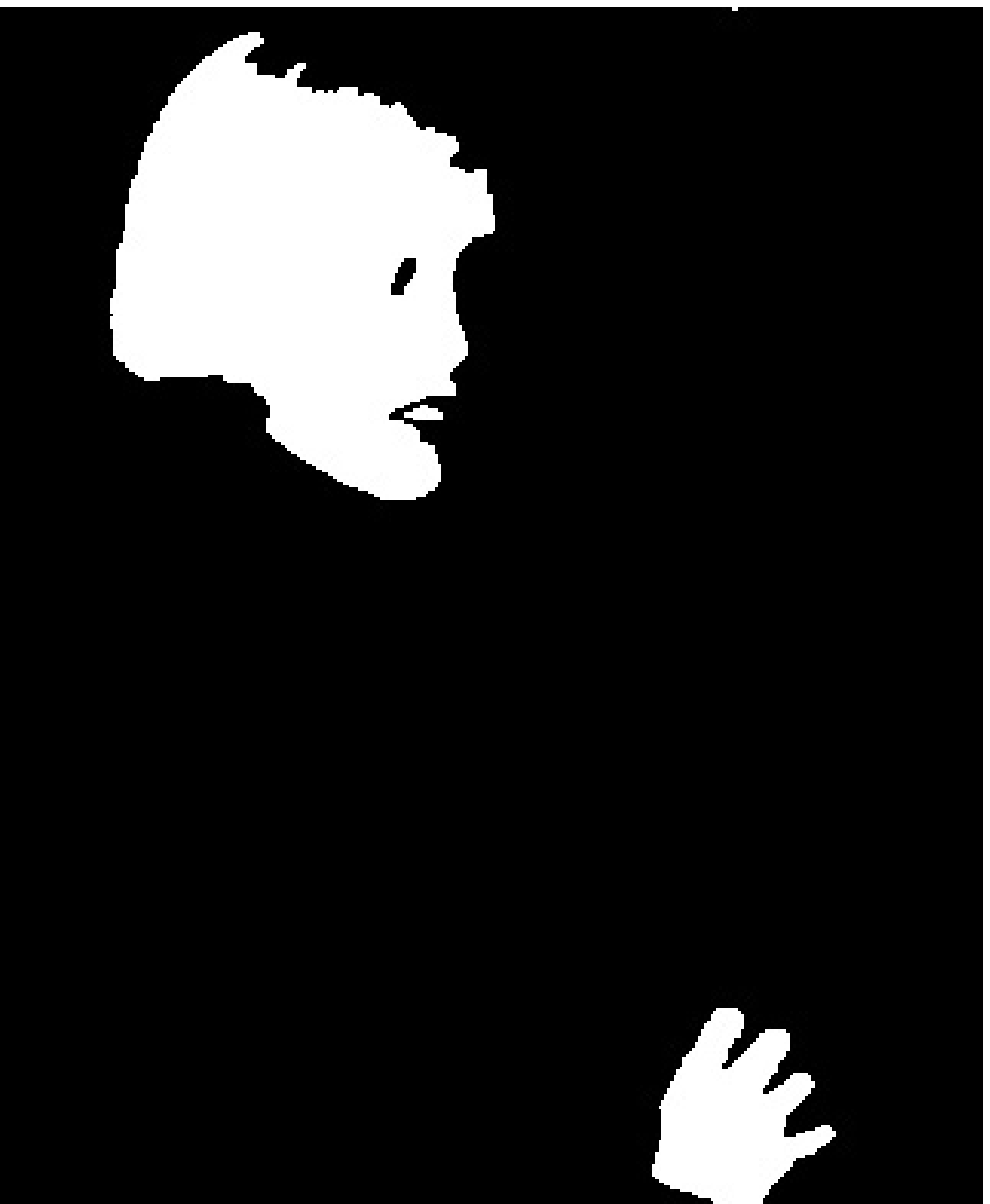}\label{fig:8j}}\
\subfloat[ ]{\includegraphics[width=1.893cm]{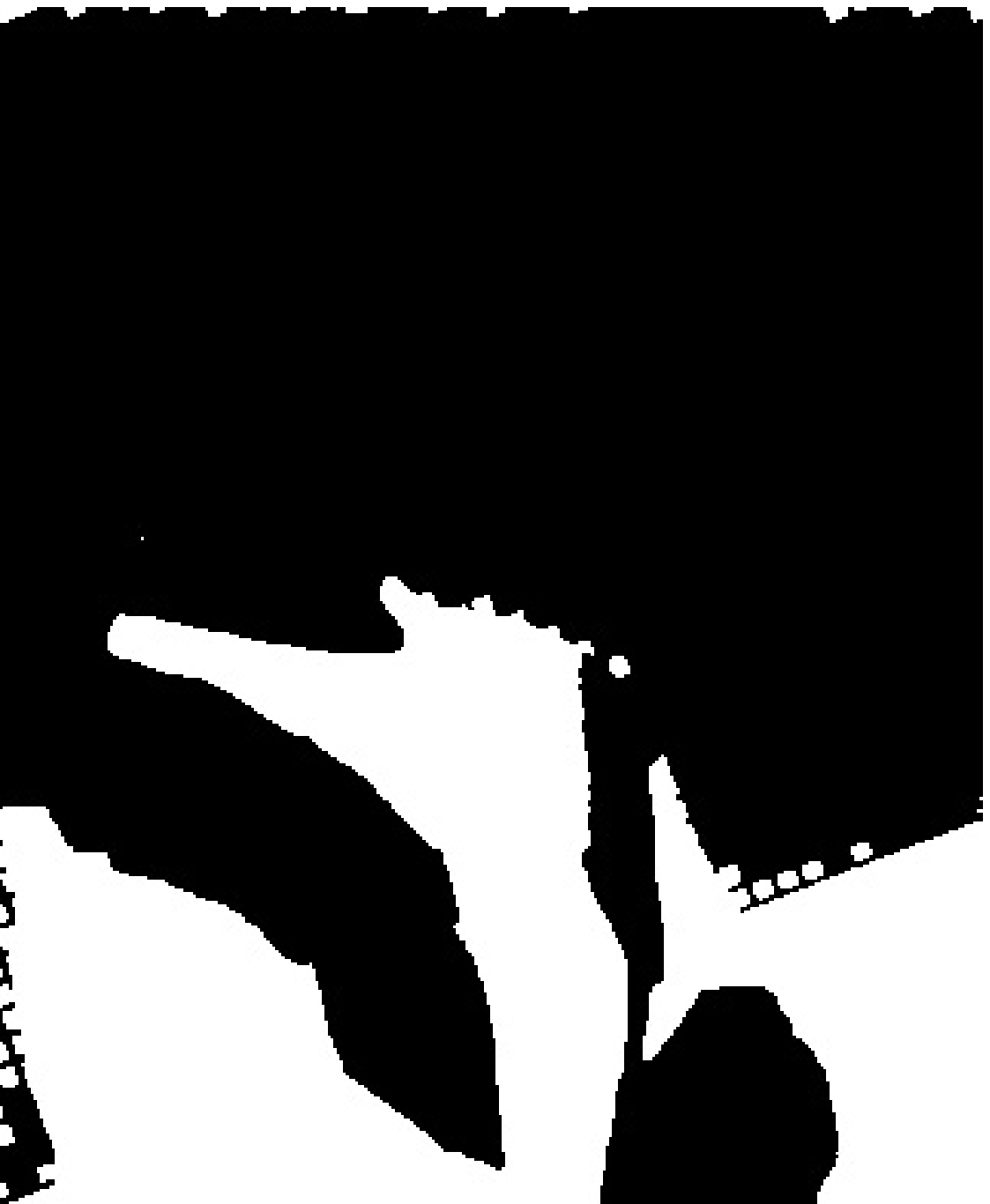}\label{fig:8k}}\
\subfloat[ ]{\includegraphics[width=1.893cm]{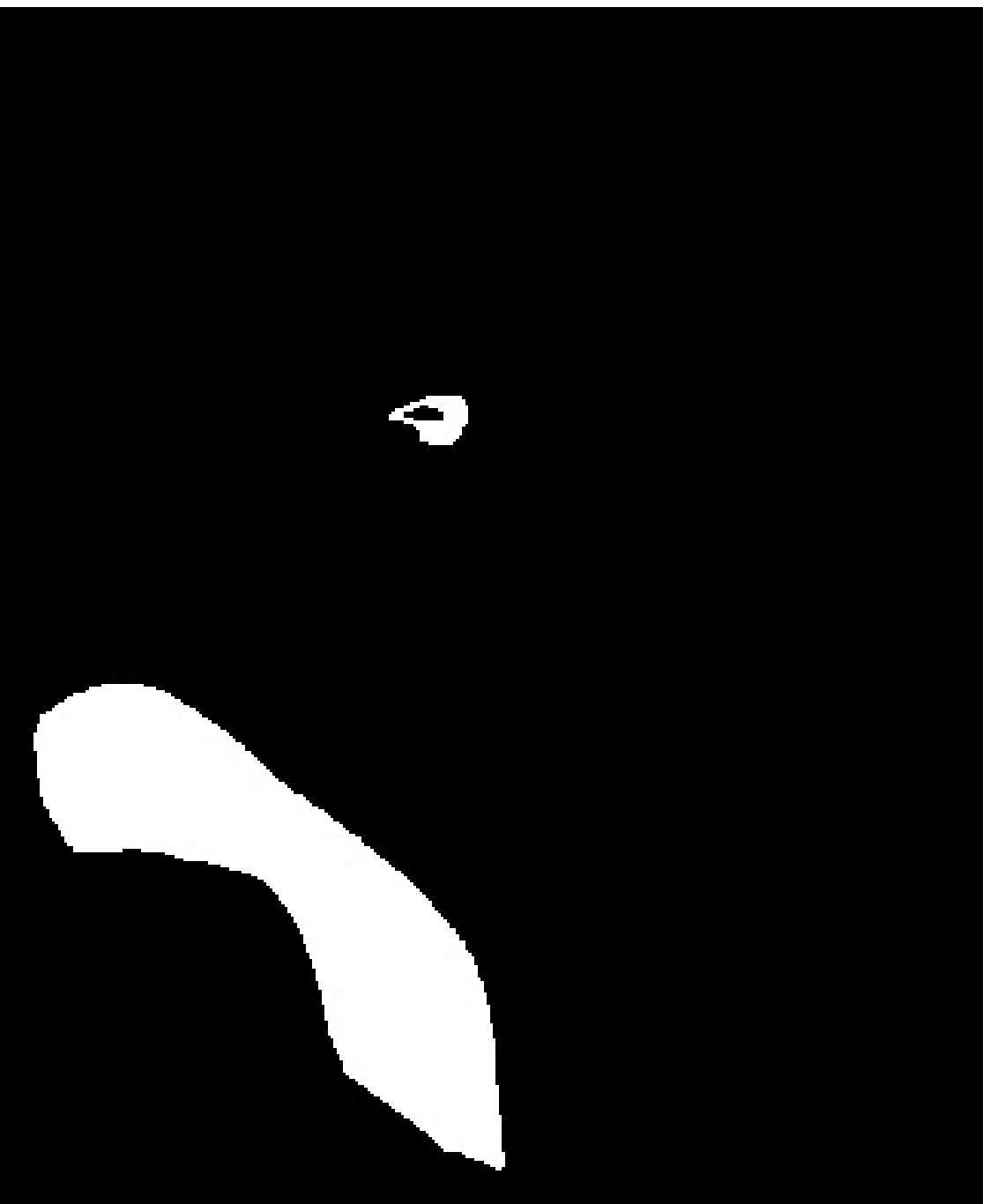}\label{fig:8l}}
\end{center}
\caption{Detailed comparison of FCM and L1FS on the woman image. First row: six segmented regions of FCM. Second row: six segmented regions of L1FS.}\label{fig:8}
\end{figure}

\section{Conclusion}\label{sec:conclusion}
This paper presents a novel piecewise constant image segmentation model based on fuzzy membership functions and L1-norm fidelity. ADMM is applied to derive an efficient numerical algorithm, in which each subproblem has a closed-form solution. In particular, an efficient algorithm is proposed to find the fuzzy median. The numerical results on both synthetic and real piecewise constant images demonstrate that the proposed method is superior to some existing state-of-the-art methods since it is more robust to impulse noise and can preserve better contrast. Even in the case of Gaussian noise, the proposed method can achieve similar results as its L2-fidelity counterpart. In this work, we assume that the image to be dealt with is piecewise constant, which works well on images with homogeneous image features. The future work is to extend this framework to piecewise smooth image segmentation.

\section*{Acknowledgment}
The research of F. Li was supported by the 973 Program 2011CB707104, the research of S. Osher and J. Qin was supported by ONR Grants N00014120838 and N00014140444, NSF Grants DMS-1118971 and CCF-0926127, and the Keck Foundation, the research of M. Yan was supported by NSF Grants DMS-1317602. The work was done when the first author was visiting UCLA Department of Mathematics.


\begin{thebibliography}{10}

\bibitem{ahmed2002modified}
Mohamed~N Ahmed, Sameh~M Yamany, Nevin Mohamed, Aly~A Farag, and Thomas
  Moriarty.
\newblock A modified fuzzy c-means algorithm for bias field estimation and
  segmentation of {MRI} data.
\newblock {\em Medical Imaging, IEEE Transactions on}, 21(3):193--199, 2002.

\bibitem{aubert2006mathematical}
Gilles Aubert and Pierre Kornprobst.
\newblock {\em Mathematical problems in image processing: partial differential
  equations and the calculus of variations}, volume 147.
\newblock Springer, 2006.

\bibitem{bauschke2011convex}
H.H. Bauschke and P.L. Combettes.
\newblock {\em Convex Analysis and Monotone Operator Theory in Hilbert Spaces}.
\newblock CMS Books in Mathematics. Springer New York, 2011.

\bibitem{FCM}
James~C Bezdek, LO~Hall, and LP~Clarke.
\newblock Review of {MR} image segmentation techniques using pattern
  recognition.
\newblock {\em Medical Physics}, 20(4):1033--1048, 1992.

\bibitem{ADMM}
Stephen Boyd, Neal Parikh, Eric Chu, Borja Peleato, and Jonathan Eckstein.
\newblock Distributed optimization and statistical learning via the alternating
  direction method of multipliers.
\newblock {\em Foundations and Trends{\textregistered} in Machine Learning},
  3(1):1--122, 2011.

\bibitem{bresson2007fast}
Xavier Bresson, Selim Esedo\=glu, Pierre Vandergheynst, Jean-Philippe Thiran,
  and Stanley Osher.
\newblock Fast global minimization of the active contour/snake model.
\newblock {\em Journal of Mathematical Imaging and vision}, 28(2):151--167,
  2007.

\bibitem{cai2015variational}
Xiaohao Cai.
\newblock Variational image segmentation model coupled with image restoration
  achievements.
\newblock {\em Pattern Recognition}, 48(6):2029--2042, 2015.

\bibitem{caselles1997geodesic}
Vicent Caselles, Ron Kimmel, and Guillermo Sapiro.
\newblock Geodesic active contours.
\newblock {\em International Journal of Computer Vision}, 22(1):61--79, 1997.

\bibitem{chambolle2012convex}
Antonin Chambolle, Daniel Cremers, and Thomas Pock.
\newblock A convex approach to minimal partitions.
\newblock {\em SIAM Journal on Imaging Sciences}, 5(4):1113--1158, 2012.

\bibitem{CP}
Antonin Chambolle and Thomas Pock.
\newblock A first-order primal-dual algorithm for convex problems with
  applications to imaging.
\newblock {\em Journal of Mathematical Imaging and Vision}, 40(1):120--145,
  2011.

\bibitem{chan2014two}
Raymond Chan, Hongfei Yang, and Tieyong Zeng.
\newblock A two-stage image segmentation method for blurry images with poisson
  or multiplicative gamma noise.
\newblock {\em SIAM Journal on Imaging Sciences}, 7(1):98--127, 2014.

\bibitem{chan2005salt}
Raymond~H Chan, Chung-Wa Ho, and Mila Nikolova.
\newblock Salt-and-pepper noise removal by median-type noise detectors and
  detail-preserving regularization.
\newblock {\em Image Processing, IEEE Transactions on}, 14(10):1479--1485,
  2005.

\bibitem{chan2005aspects}
Tony~F Chan and Selim Esedoglu.
\newblock Aspects of total variation regularized {L1} function approximation.
\newblock {\em SIAM Journal on Applied Mathematics}, 65(5):1817--1837, 2005.

\bibitem{chan2006algorithms}
Tony~F Chan, Selim Esedoglu, and Mila Nikolova.
\newblock Algorithms for finding global minimizers of image segmentation and
  denoising models.
\newblock {\em SIAM Journal on Applied Mathematics}, 66(5):1632--1648, 2006.

\bibitem{CV}
Tony~F Chan and Luminita~A Vese.
\newblock Active contours without edges.
\newblock {\em Image Processing, IEEE transactions on}, 10(2):266--277, 2001.

\bibitem{FCMS2}
Songcan Chen and Daoqiang Zhang.
\newblock Robust image segmentation using {FCM} with spatial constraints based
  on new kernel-induced distance measure.
\newblock {\em Systems, Man, and Cybernetics, Part B: Cybernetics, IEEE
  Transactions on}, 34(4):1907--1916, 2004.

\bibitem{chen2011projection}
Yunmei Chen and Xiaojing Ye.
\newblock Projection onto a simplex.
\newblock {\em arXiv preprint arXiv:1101.6081}, 2011.

\bibitem{dubrovina2015multi}
Anastasia Dubrovina, Guy Rosman, and Ron Kimmel.
\newblock Multi-region active contours with a single level set function.
\newblock {\em Pattern Analysis and Machine Intelligence, IEEE Transactions on}, 37(8):1585-1601, 2015.

\bibitem{ADMM2}
Daniel Gabay and Bertrand Mercier.
\newblock A dual algorithm for the solution of nonlinear variational problems
  via finite element approximation.
\newblock {\em Computers \& Mathematics with Applications}, 2(1):17--40, 1976.

\bibitem{ADMM1}
Roland Glowinski and A~Marroco.
\newblock Sur l'approximation, par \'{e}l\'{e}ments finis d'ordre un, et la
  r\'{e}solution, par p\'{e}nalisation-dualit\'{e} d'une classe de
  probl\`{e}mes de dirichlet non lin\'{e}aires.
\newblock {\em ESAIM: Mathematical Modelling and Numerical
  Analysis-Mod{\'e}lisation Math{\'e}matique et Analyse Num{\'e}rique},
  9(R2):41--76, 1975.

\bibitem{SB}
Tom Goldstein and Stanley Osher.
\newblock The split bregman method for {L1}-regularized problems.
\newblock {\em SIAM Journal on Imaging Sciences}, 2(2):323--343, 2009.

\bibitem{Qin2014}
Weihong Guo, Jing Qin, and Sibel Tari.
\newblock Automatic prior shape selection for image segmentation.
\newblock {\em Research in Shape Modeling}, 2015.

\bibitem{guo2009fast}
Xiaoxia Guo, Fang Li, and Michael~K Ng.
\newblock A fast $\ell_1$-{TV} algorithm for image restoration.
\newblock {\em SIAM Journal on Scientific Computing}, 31(3):2322--2341, 2009.

\bibitem{TVFCM}
Yanyan He, M~Yousuff~Hussaini, Jianwei Ma, Behrang Shafei, and Gabriele Steidl.
\newblock A new fuzzy c-means method with total variation regularization for
  segmentation of images with noisy and incomplete data.
\newblock {\em Pattern Recognition}, 45(9):3463--3471, 2012.

\bibitem{houhou2008fast}
Nawal Houhou, J~Thiran, and Xavier Bresson.
\newblock Fast texture segmentation model based on the shape operator and
  active contour.
\newblock In {\em Computer Vision and Pattern Recognition, 2008. CVPR 2008.
  IEEE Conference on}, pages 1--8. IEEE, 2008.

\bibitem{TVSEGL1}
Miyoun Jung, Myeongmin Kang, and Myungjoo Kang.
\newblock Variational image segmentation models involving non-smooth
  data-fidelity terms.
\newblock {\em Journal of Scientific Computing}, 59(2):277--308, 2014.

\bibitem{jung2007multiphase}
Yoon~Mo Jung, Sung~Ha Kang, and Jianhong Shen.
\newblock Multiphase image segmentation via {M}odica-{M}ortola phase
  transition.
\newblock {\em SIAM Journal on Applied Mathematics}, 67(5):1213--1232, 2007.

\bibitem{fuzzy_median}
P.R. Kersten.
\newblock Fuzzy order statistics and their application to fuzzy clustering.
\newblock {\em Fuzzy Systems, IEEE Transactions on}, 7(6):708--712, 1999.

\bibitem{FLICM}
Stelios Krinidis and Vassilios Chatzis.
\newblock A robust fuzzy local information c-means clustering algorithm.
\newblock {\em Image Processing, IEEE Transactions on}, 19(5):1328--1337, 2010.

\bibitem{li2010competition}
Fang Li, Michael~K Ng, Tie~Yong Zeng, and Chunli Shen.
\newblock A multiphase image segmentation method based on fuzzy region
  competition.
\newblock {\em SIAM Journal on Imaging Sciences}, 3(3):277--299, 2010.

\bibitem{li2010softseg}
Fang Li, Chaomin Shen, and Chunming Li.
\newblock Multiphase soft segmentation with total variation and {H}1
  regularization.
\newblock {\em Journal of Mathematical Imaging and Vision}, 37(2):98--111,
  2010.

\bibitem{PCLSM}
Johan Lie, Marius Lysaker, and Xue-Cheng Tai.
\newblock A piecewise constant level set framework.
\newblock {\em International Journal of Numerical Analysis and Modeling},
  2(4):422--438, 2005.

\bibitem{lie2006binary}
Johan Lie, Marius Lysaker, and Xue-Cheng Tai.
\newblock A binary level set model and some applications to {M}umford-{S}hah
  image segmentation.
\newblock {\em Image Processing, IEEE Transactions on}, 15(5):1171--1181, 2006.

\bibitem{mory2007fuzzy}
Benoit Mory and Roberto Ardon.
\newblock Fuzzy region competition: a convex two-phase segmentation framework.
\newblock In {\em Scale Space and Variational Methods in Computer Vision},
  pages 214--226. Springer, 2007.

\bibitem{MS}
David Mumford and Jayant Shah.
\newblock Optimal approximations by piecewise smooth functions and associated
  variational problems.
\newblock {\em Communications on Pure and Applied Mthematics}, 42(5):577--685,
  1989.

\bibitem{nikolova2004variational}
Mila Nikolova.
\newblock A variational approach to remove outliers and impulse noise.
\newblock {\em Journal of Mathematical Imaging and Vision}, 20(1-2):99--120,
  2004.

\bibitem{osher1988fronts}
Stanley Osher and James~A Sethian.
\newblock Fronts propagating with curvature-dependent speed: algorithms based
  on {H}amilton-{J}acobi formulations.
\newblock {\em Journal of Computational Physics}, 79(1):12--49, 1988.

\bibitem{pham2002fuzzy}
Dzung~L Pham.
\newblock Fuzzy clustering with spatial constraints.
\newblock In {\em Image Processing. 2002. Proceedings. 2002 International
  Conference on}, volume~2, pages II--65. IEEE, 2002.

\bibitem{ROF}
Leonid~I Rudin, Stanley Osher, and Emad Fatemi.
\newblock Nonlinear total variation based noise removal algorithms.
\newblock {\em Physica D: Nonlinear Phenomena}, 60(1):259--268, 1992.

\bibitem{sawatzky2013variational}
Alex Sawatzky, Daniel Tenbrinck, Xiaoyi Jiang, and Martin Burger.
\newblock A variational framework for region-based segmentation incorporating
  physical noise models.
\newblock {\em Journal of Mathematical Imaging and Vision}, 47(3):179--209,
  2013.

\bibitem{shen2002mathematical}
Jianhong Shen and Tony~F Chan.
\newblock Mathematical models for local nontexture inpaintings.
\newblock {\em SIAM Journal on Applied Mathematics}, 62(3):1019--1043, 2002.

\bibitem{softMS}
Jianhong~Jackie Shen.
\newblock A stochastic-variational model for soft {M}umford-{S}hah
  segmentation.
\newblock {\em International Journal of Biomedical Imaging}, 2006:92329, 2006.

\bibitem{storath2014fast}
Martin Storath and Andreas Weinmann.
\newblock Fast partitioning of vector-valued images.
\newblock {\em SIAM Journal on Imaging Sciences}, 7(3):1826--1852, 2014.

\bibitem{vese2002multiphase}
Luminita~A Vese and Tony~F Chan.
\newblock A multiphase level set framework for image segmentation using the
  {M}umford and {S}hah model.
\newblock {\em International Journal of Computer Vision}, 50(3):271--293, 2002.

\bibitem{wang2008new}
Yilun Wang, Junfeng Yang, Wotao Yin, and Yin Zhang.
\newblock A new alternating minimization algorithm for total variation image
  reconstruction.
\newblock {\em SIAM Journal on Imaging Sciences}, 1(3):248--272, 2008.

\bibitem{zhang2010alternating}
Yangyang Xu, Wotao Yin, Zaiwen Wen, and Yin Zhang.
\newblock An alternating direction algorithm for matrix completion with
  nonnegative factors.
\newblock {\em Frontiers of Mathematics in China}, 7(2):365--384, 2012.

\bibitem{Yan13}
Ming Yan.
\newblock Restoration of images corrupted by impulse noise and mixed {G}aussian
  impulse noise using blind inpainting.
\newblock {\em {SIAM} J. Imaging Sciences}, 6(3):1227--1245, 2013.

\bibitem{Yan13b}
Ming Yan, Alex A.~T. Bui, Jason Cong, and Luminita~A. Vese.
\newblock General convergent expectation maximization ({EM})-type algorithms
  for image reconstruction.
\newblock {\em Inverse Problems and Imaging}, 7(3):1007--1029, 2013.

\bibitem{zhu1996region}
Song~Chun Zhu and Alan Yuille.
\newblock Region competition: Unifying snakes, region growing, and
  {B}ayes/{MDL} for multiband image segmentation.
\newblock {\em Pattern Analysis and Machine Intelligence, IEEE Transactions
  on}, 18(9):884--900, 1996.

\end{thebibliography}

\end{document}